\documentclass[11pt]{amsart}

\usepackage[dvips]{graphicx}
\usepackage{calc}
\usepackage{color}
\usepackage{amsmath,amssymb,amscd,amsthm,amsbsy}
\usepackage{enumerate}
\usepackage[T1]{fontenc}
\usepackage{inputenc}
\usepackage{enumerate}
\usepackage[all]{xy}
\usepackage{hyperref}

\theoremstyle{plain}
\newtheorem{theo}{Theorem}
\newtheorem{prop}[theo]{Proposition}
\newtheorem{lemma}[theo]{Lemma}
\newtheorem*{notation*}{Notation}
\newtheorem*{notations*}{Notations}
\newtheorem*{claim*}{Claim}

\theoremstyle{definition}

\newtheorem{defin}[theo]{Definition}

\newlength{\espaceavantspecialthm}
\newlength{\espaceapresspecialthm}
\newtheorem{exple}[theo]{Example}

\newtheorem{quest}[theo]{Question}

\newenvironment{remark}[1][]{\refstepcounter{theo} 
\vskip \espaceavantspecialthm \noindent \textsc{Remark~\thetheo
#1.} }%
{\vskip \espaceapresspecialthm}

\newenvironment{Properties}
{\begin{list}{}{
\setlength{\topsep}{6pt}%
\setlength{\itemsep}{4pt}%
\setlength{\labelsep}{0pt}%
\setlength{\leftmargin}{0pt}%
\setlength{\labelwidth}{0pt}%
\setlength{\listparindent}{0pt}}%
\setlength{\parskip}{0pt}%
}
{\end{list}}

\def\bb#1{\mathbb{#1}} \def\m#1{\mathcal{#1}}
\def\del{\partial}
\def\co{\colon\thinspace}
\def\ham{\mathrm{Ham}}

\newcommand{\R}{{\mathbb{R}}}

\newcommand{\Z}{{\mathbb{Z}}}
\newcommand{\N}{{\mathbb{N}}}
\newcommand{\C}{{\mathbb{C}}}

\newcommand{\cA}{{\mathcal{A}}}
\newcommand{\cB}{{\mathcal{B}}}

\newcommand{\cD}{{\mathcal{D}}}

\newcommand{\cF}{{\mathcal{F}}}

\newcommand{\cH}{{\mathcal{H}}}

\newcommand{\cL}{{\mathcal{L}}}
\newcommand{\cM}{{\mathcal{M}}}

\newcommand{\cP}{{\mathcal{P}}}
\newcommand{\cQ}{{\mathcal{Q}}}

\newcommand{\cS}{{\mathcal{S}}}

\newcommand{\cU}{{\mathcal{U}}}
\newcommand{\cV}{{\mathcal{V}}}

\newcommand{\sfA}{{\mathsf{A}}}

\newcommand{\mfo}{{\mathfrak{o}}}

\newcommand{\fc}{{:\ }}

\newcommand{\ol}{\overline}
\newcommand{\wt}{\widetilde}
\newcommand{\wh}{\widehat}

\newcommand{\tb}{\textbf}

\DeclareMathOperator{\Hom}{Hom}

\DeclareMathOperator{\Crit}{Crit}

\DeclareMathOperator{\im}{im}
\DeclareMathOperator{\id}{id}

\DeclareMathOperator{\PSS}{PSS}

\DeclareMathOperator{\lcm}{lcm}

\DeclareMathOperator{\ind}{ind}

\DeclareMathOperator{\Ham}{Ham}
\DeclareMathOperator{\Symp}{Symp}

\DeclareMathOperator{\osc}{osc}

\DeclareMathOperator{\Spec}{Spec}

\DeclareMathOperator{\coker}{coker}

\DeclareMathOperator{\Pin}{Pin}

\DeclareMathOperator{\ddd}{d}

\def\ham{\mathrm{Ham}}
\def\symp#1{\mathrm{Symp}#1}

\begin{document}

\title{Spectral invariants for monotone Lagrangians}
\author{R\'emi Leclercq, Frol Zapolsky}
\date{\today}

\address{RL: Universit\'e Paris-Sud, D\'epartement de Math\'ematiques, Bat. 425, 91400 Orsay, France}
\email{remi.leclercq@math.u-psud.fr}

\address{FZ: Department of Mathematics, University of Haifa, Mount Carmel, Haifa 31905, Israel}
\email{frol.zapolsky@gmail.com}

\subjclass[2010]{Primary 57R17; Secondary 53D12 53D40} 
\keywords{Symplectic manifolds, monotone Lagrangian submanifolds, Lagrangian Floer homology, quantum homology, spectral invariants}

\begin{abstract}
Since spectral invariants were introduced in cotangent bundles via generating functions by Viterbo in the seminal paper \cite{Viterbo_gfqi}, they have been defined in various contexts, mainly via Floer homology theories, and then used in a great variety of applications.
In this paper we extend their definition to monotone Lagrangians, which is so far the most general case for which a ``classical'' Floer theory has been developed. 
Then, we gather and prove the properties satisfied by these invariants, and which are crucial for their applications.
Finally, as a demonstration, we apply these new invariants to symplectic rigidity of some specific monotone Lagrangians.\end{abstract}

\maketitle
\tableofcontents


\section{Introduction and main results}

Spectral invariants were introduced into symplectic topology by Viterbo \cite{Viterbo_gfqi}, and subsequently their theory was developed by Schwarz \cite{Schwarz_Action_sp_aspherical_mfds}, Oh \cite{Oh_Construction_sp_invts_Ham_paths_closed_symp_mfds} in the context of periodic orbit Floer homology, by Oh \cite{Oh_sympl_topology_action_fcnl_I}, Leclercq \cite{Leclercq_spectral_invariants_Lagr_FH}, Monzner--Vichery--Zapolsky \cite{Monzner_Vichery_Zapolsky_partial_qms_qss_cot_bundles} in the context of Lagrangian Floer homology, by Chaperon \cite{Chaperon_On_generating_families}, Bhupal \cite{Bhupal_partial_order_Rn}, Sandon \cite{Sandon_CH_capacity_nonsqueezing_gf}, and Zapolsky \cite{Zapolsky_Geometry_Cont_groups_ct_rigidity} for Legendrian submanifolds of contact manifolds and for contactomorphisms, and by Albers--Frauenfelder \cite{Albers_Frauenfelder_Sp_invts_RFH_global_Ham_pert} in Rabinowitz Floer homology.

They have been used in numerous deep applications, such as metrics on infinite-dimensional groups of symmetries \cite{Viterbo_gfqi, Schwarz_Action_sp_aspherical_mfds, Oh_Sp_invts_analysis_Floer_moduli_space_geom_Ham, Khanevsky_Hofer_metric_diameters, Leclercq_spectral_invariants_Lagr_FH, Sandon_integer_biinvt_metric_Rtwo_n_S_one, Monzner_Zapolsky_Comparison_symp_homogenization_Calabi_QSs, Colin_Sandon_discriminant_osc_lt_ct_Leg_isos, Zapolsky_Geometry_Cont_groups_ct_rigidity, zapolsky,   Seyfaddini_Unboundedness_Lag_Hofer_distance_ball}; the symplectic camel problem \cite{Viterbo_gfqi, theret}; quasi-morphisms on the Hamiltonian group \cite{Entov_Polterovich_Calabi_quasimorphism_quantum_homology, Ostrover_QMs_nonmonotone_symp_mfds, Usher_Deformed_Ham_HF_cap_estimates_Calabi_QMs, FOOO_Sp_invts_bulk_QMs_Lag_HF}; quasi-states and symplectic and contact rigidity \cite{Entov_Polterovich_Quasi_states_symplectic_intersections, Entov_Polterovich_rigid_subsets_sympl_mfds, Monzner_Vichery_Zapolsky_partial_qms_qss_cot_bundles, Zapolsky_Geometry_Cont_groups_ct_rigidity}; orderability and contact nonsqueezing \cite{Sandon_CH_capacity_nonsqueezing_gf, Albers_Merry_Orderability_nonsqueezing_RFH}; $C^0$-symplectic topology \cite{Seyfaddini_C0_limits_Ham_paths_sp_invts, Seyfaddini_Displaced_disk_problem_symp_top, Humiliere_Leclercq_Seyfaddini_Reduction_sympeo, Humiliere_Leclercq_Seyfaddini_Coisotropic_rigidity_C0_symp_geom, HLS12}; function theory on symplectic manifolds \cite{Entov_Polterovich_Zapolsky_qms_Poisson_bracket, Buhovsky_Entov_Polterovich_PB_symp_invts},  \cite{Polterovich_Rosen_Function_thry_symp_mfds} and the references therein; quantum measurements and noise \cite{Polterovich_Quantum_unsharpness_symp_rigidity, Polterovich_Symp_geom_quantum_noise}; surface dynamics \cite{Humiliere_Le_Roux_Seyfaddini_Dyn_interpretation_sp_invts}; and contact dynamics \cite{Zapolsky_Geometry_Cont_groups_ct_rigidity}. There are other applications, and it is not feasible to list all of them here, but the above sample should give the reader a feeling of the power of this wonderful tool of symplectic topology.

Spectral invariants have been defined in various contexts, such as periodic orbit Floer homology and Lagrangian Floer homology for weakly exact Lagrangians. The present paper extends their definition to the setting of Lagrangian Floer homology for monotone Lagrangian submanifolds. This is so far the most general case for which meaningful theory can be developed, while staying in the realm of ``classical'' Floer theory, that is avoiding the more advanced and complicated techniques such as virtual fundamental cycles and Kuranishi structures, used, for instance in \cite{FOOO_Lagr_intersection_Floer_thry_anomaly_obstr_I, FOOO_Lagr_intersection_Floer_thry_anomaly_obstr_II, FOOO_Sp_invts_bulk_QMs_Lag_HF}.

Future applications of spectral invariants for monotone Lagrangians include \emph{inter alia} new results on the Lagrangian Hofer metric \cite{Khanevsky_Zapolsky_Lagrangian_Hofer_metric}.

\subsection{Main result}\label{subsec:main_result}

Let us briefly review the setting in which we will be working. Throughout this paper we fix a closed\footnote{In this paper we limit ourselves to closed symplectic manifolds, however, using techniques developed by Frauenfelder--Schlenk \cite{frauenfelder-schlenk} and Lanzat \cite{Lanzat_QMs_QSs_convex_symp_mfds}, it is straightforward to generalize to the case of manifolds which are convex at infinity \cite{Eliashberg_Gromov_Convex_symp_mfds}.} connected symplectic manifold $(M,\omega)$ of dimension $2n$ and a closed connected Lagrangian submanifold $L \subset M$. We have the natural homomorphisms
$$\omega \fc \pi_2(M,L) \to \R\,, \qquad \text{ the \tb{symplectic area}, and}$$
$$\mu \fc \pi_2(M,L) \to \Z\,, \qquad \text{ the \tb{Maslov index}.}$$
We say that $L$ is \tb{monotone} if there is a positive constant $\tau$ such that
$$\omega(A) = \tau\mu(A)\quad \text{for every }A \in \pi_2(M,L)\,.$$
The \tb{minimal Maslov number} $N_L$ of $L$ is defined to be the positive generator of the subgroup $\mu\big(\pi_2(M,L)\big) \subset \Z$ if it is nontrivial, otherwise we set $N_L = \infty$. In this paper we assume that $L$ is \emph{monotone of minimal Maslov number at least two}. Usually we will assume that $N_L$ is finite, which then implies that the group of periods $\omega\big(\pi_2(M,L)\big) \subset \R$ is infinite cyclic; in this case we let $\sfA = \tau N_L\in \R_{>0}$ be its positive generator. We will leave to the reader the task of adapting the results presented in this paper to the case $N_L = \infty$, which is the case of weakly exact Lagrangians.

We proceed to formulate the main result of this paper. A few preliminaries are in order. Fix a ground ring $R$. If its characteristic is different from $2$, we require that the following assumption be satisfied.

\bigskip

\tb{Assumption (O):} The second Stiefel--Whitney class $w_2(TL)$ of $L$ vanishes on the image of the boundary homomorphism $\pi_3(M,L) \to \pi_2(L)$.

\bigskip

\begin{remark}\label{rem:assmpt_O}
Assumption \tb{(O)} is satisfied if $L$ is relatively $\Pin^\pm$, that is either one of the classes $w_2(TL)$ or $w_2(TL) + w_1^2(TL)$ is in the image of the restriction morphism $H^2(M;\Z_2) \to H^2(L;\Z_2)$. However it is a strictly weaker assumption, see \S \ref{subsec:discussion}.
\end{remark}

In \cite{Biran_Cornea_Lagr_QH, Biran_Cornea_Lagr_top_enumerative_invts, Zapolsky_Canonical_ors_HF} it is shown how to define the \tb{quantum homology} $QH_*(L)$ of $L$ over $R$. 
This is an algebra over $R$, where the product operation is
$$\star \fc QH_j(L) \otimes QH_k(L) \to QH_{j+k-n}(L)\,,$$
with a unit which is given by the fundamental class of $L$, $[L] \in QH_n(L)$. It follows that $QH_*(L) \neq 0$ if and only if $[L] \neq 0$. \emph{We assume that this is the case throughout the text.}

In \S \ref{sec:definition} below we define the \tb{Lagrangian spectral invariant} associated to $L$, which is a function
$$\ell \fc QH_*(L) \times C^0\big(M \times [0,1]\big) \to \R \cup \{-\infty\}\,.$$
We use the abbreviation
\begin{equation}\label{eq:ell_plus}
\ell_+ = \ell([L];\cdot)\,.
\end{equation}
For a smooth Hamiltonian $H$ on $M$ we let $\Spec(H:L) \subset \R$ be its action spectrum relative to $L$, that is the set of critical values of the associated action functional, see the definitions in \S \ref{sec:background}.

Given a continuous function $H \fc M \times [0,1] \rightarrow \bb R$, we denote by $\overline H$ the function defined by $\overline H_t(x) = -H_{1-t}(x)$. In case $H$ is smooth, we have $\phi_{\overline H}^t = \phi_H^{1-t}\phi_H^{-1}$ for all $t$. If $H,K \fc M \times [0,1] \to \R$ are continuous and $H_t = K_t = 0$ for $t$ close to $0,1$, we define $H \sharp K \fc M \times [0,2] \to \R$ to be their concatenation with respect to the time variable.\footnote{The definition of spectral invariants can be generalized in a straightforward manner to Hamiltonians parametrized by a time interval which is not necessarily $[0,1]$.}

Our first result concerns the main properties of the spectral invariant.

\begin{theo}\label{thm:main_properties_Intro}
Let $L$ be a closed connected monotone Lagrangian submanifold of $(M,\omega)$ with minimal Maslov number $N_L \geq 2$. The spectral invariant
$$\ell \fc QH_*(L) \times C^0\big(M\times [0,1]\big) \rightarrow \bb R \cup \{ -\infty \}$$
satisfies the following properties. 
\begin{Properties} 
   \item[Finiteness] $\ell(\alpha;H) = - \infty$ if and only if $\alpha = 0$.
   \item[Spectrality] For $H \in C^\infty\big(M \times [0,1]\big)$ and $\alpha \neq 0$, $\ell(\alpha;H) \in \Spec(H:L)$.
   \item[Ground ring action] For $r \in R$, $\ell(r \cdot \alpha;H) \leq \ell(\alpha;H)$. In particular, if $r$ is invertible, then $\ell(r\cdot \alpha;H) = \ell(\alpha;H)$.
   \item[Normalization] If $c$ is a function of time then $$\ell(\alpha;H+c)=\ell(\alpha;H) + \int_0^1 c(t) \,dt\,.$$ We have $\ell_+(0)=0$.
   \item[Continuity] For any Hamiltonians $H$, $K$, and $\alpha \neq 0$:
    $$\int_0^1 \min_M (K_t - H_t) \,dt \leq \ell(\alpha;K) - \ell(\alpha;H) \leq \int_0^1 \max_M (K_t - H_t) \,dt \,.$$ 
   \item[Monotonicity] If $H \leq K$, then $\ell(\alpha;H) \leq \ell(\alpha;K)$.
   \item[Triangle inequality] For all $\alpha$ and $\beta$:
   $$\ell(\alpha \star \beta; H \sharp K) \leq \ell(\beta;H) + \ell(\alpha;K)\,.$$

   \item[Lagrangian control] If for all $t$, $H_t|_L = c(t) \in \bb R$ (respectively $\leq$, $\geq$), then
   $$\ell_+(H)=\int_0^1 c(t) \,dt\quad\text{(respectively } \leq, \geq)\,.$$
   Thus for all $H$:
   $$\int_0^1 \min_L H_t \,dt \leq \ell_+(H) \leq \int_0^1 \max_L H_t \,dt \,.$$
   \item[Non-negativity] $\ell_+(H) + \ell_+(\overline{H}) \geq 0$.
\end{Properties}
\end{theo}
\noindent Theorem \ref{thm:main_properties_Intro} is proved in \S \ref{sec:main-properties} below as part of the more general Theorem \ref{thm:main_properties_Lagr_sp_invts}.

We say that $H$ is \tb{normalized} if $\int_M H_t\,\omega^n = 0$ for all $t$. We have the following, by now standard, observation.
\begin{prop}\label{prop:sp_invts_htpy_class_path_Intro}
If $H$ is normalized, then for $\alpha \in QH_*(L) \setminus \{0\}$ the spectral invariant $\ell(\alpha;H)$ only depends on the homotopy class of the path $\{\phi_H^t\}_{t \in [0,1]}$ relative to the endpoints.
\end{prop}
\noindent See \S \ref{sec:invariance-ell2} for a proof. Thus the spectral invariant descends to a function
$$\ell \fc QH_*(L) \times \wt \Ham(M,\omega) \to \R \cup \{-\infty\}\,.$$
This function satisfies a number of properties, analogous to those formulated in Theorem \ref{thm:main_properties_Intro}. See \S \ref{sec:SI-isotopies} for  details.

\subsection{Overview of the construction of spectral invariants}

We present here a brief sketch of the construction of $\ell$. The reader is referred to \S \ref{sec:definition} for full details.

Assume that $H \fc M \times [0,1] \to \R$ is such that $\phi_1(L)$ intersects $L$ transversely. It is then possible to choose a time-dependent family of $\omega$-compatible almost complex structures $J$ on $M$ for which the Floer complex
$$\big(CF_*(H:L),\partial_{H,J} \big)$$
is defined. This complex is filtered by the action functional of $H$, and we let
$$CF_*^a(H:L) \subset CF_*(H:L)$$
be the submodule generated by elements whose action is $< a$. From the definition of the Floer boundary operator $\partial_{H,J}$ it follows that this submodule is in fact a subcomplex. We let
$$HF_*(H,J:L) \quad \text{and} \quad HF_*^a(H,J:L)$$
be the homologies of $\big(CF_*(H:L),\partial_{H,J} \big)$ and $\big ( CF_*^a(H:L),\partial_{H,J} \big)$, respectively, and we let
$$i_*^a \fc HF_*^a(H,J:L) \to HF_*(H,J:L)$$
be the morphism induced on homology by the inclusion.

There is a canonical map, the so-called PSS isomorphism
$$\PSS^{H,J} \fc QH_*(L) \to HF_*(H,J:L)\,.$$
For $\alpha \in QH_*(L)$ we then define
$$\ell(\alpha;H) = \inf \{a \in \R \,|\, \PSS^{H,J}(\alpha) \in \im i_*^a\}\,.$$
It is then shown that $\ell$ is continuous in $H$ in the sense of Theorem \ref{thm:main_properties_Intro}, which implies that is possesses a unique extension to the set of continuous Hamiltonians.

\subsection{Additional properties of spectral invariants}

Theorem \ref{thm:main_properties_Intro} collects the more basic properties of the Lagrangian spectral invariants. The latter satisfy a number of other properties, of which we mention three in this introduction. For the full list the reader is referred to Theorem \ref{thm:main_properties_Lagr_sp_invts} in \S \ref{sec:main-properties}.

The quantum homology of $M$ acts on $QH_*(L)$ via the so-called quantum module action \cite{Biran_Cornea_Lagr_QH, Zapolsky_Canonical_ors_HF}
$$\bullet\fc QH_j(M) \otimes QH_k(L) \to QH_{j+k-2n}(L)\,.$$
We have the corresponding \textsc{Module structure} property, where
$$c \fc QH_*(M) \times C^0(M\times S^1) \to \R$$
is the Hamiltonian spectral invariant \cite{Oh_Construction_sp_invts_Ham_paths_closed_symp_mfds}, see also \S \ref{sec:non-degen-hamilt}:
\begin{prop}\label{prop:module_struct_Intro}
Let $H^1$, $H^2$ be Hamiltonians, with $H^2$ being time-periodic. Then for $a \in QH_*(M)$ and $\alpha \in QH_*(L)$ we have
$$\ell(a\bullet \alpha; H^1 \sharp H^2) \leq c(a;H^2) + \ell(\alpha;H^1)\,.$$
\end{prop}
\noindent This is proved as part of the properties of spectral invariants, Theorem \ref{thm:main_properties_Lagr_sp_invts}.

Next we show that Lagrangian spectral invariants generalize Hamiltonian spectral invariants. Let $\Delta \subset M \times M$ be the Lagrangian diagonal where the symplectic structure is $\omega \oplus (-\omega)$. In the following theorem $QH_*(\Delta)$ denotes the Lagrangian quantum homology of $\Delta$.
\begin{theo}\label{thm:spectral_invts_diagonal_Intro}
There is a canonical algebra isomorphism $QH_*(\Delta) = QH_*(M)$. For any class $\alpha \in QH_*(\Delta) = QH_*(M)$ and any time-periodic Hamiltonian $H$ on $M$ we have
$$c(\alpha;H) = \ell(\alpha;H \oplus 0)\,.$$
\end{theo}
\noindent See \S \ref{subsubsec:equiv_Ham_sp_invt_Lagr_diagonal} for a proof.

The last property we wish to mention is a relation between the spectral invariants of two monotone Lagrangians $L_i \subset (M_i, \omega_i)$ such that the product $L_1 \times L_2$ is monotone of minimal Maslov at least two. In this case both $L_1$, $L_2$ have minimal Maslov number at least $2$. In \S \ref{subsubsec:sp_invts_products} below we prove the following property:
$$\ell(\alpha_1 \otimes \alpha_2;H^1 \oplus H^2) \leq \ell(\alpha_1;H^1) + \ell(\alpha_2;H^2)$$
where $\alpha_i \in QH_*(L_i)$, $H^i$ is a Hamiltonian on $M_i$, and $\alpha_1 \otimes \alpha_2 \in QH_*(L_1 \times L_2)$ denotes the element canonically associated to $\alpha_1$, $\alpha_2$. In case the ground ring $R$ is a field, we have an equality.

  \subsection{Hofer bounds}
  \label{sec:Hofer-bound}

We now explain a relation between Lagrangian spectral invariants associated to two Lagrangian submanifolds $L$, $L' = \varphi(L)$ for some $\varphi \in \Ham(M,\omega)$, and the Hofer norm of $\varphi$, as well as the Lagrangian Hofer distance between $L$ and $L'$. 

First, as we shall see in Theorem \ref{thm:main_properties_Lagr_sp_invts}, Lagrangian spectral invariants satisfy a \textsc{Symplectic invariance} property, expressed as follows. Let $\psi \in \symp(M,\omega)$ and denote by $\ell'$ the spectral invariant associated to $L'=\psi(L)$. Then for any $\alpha' \in QH_*(L')$ and any Hamiltonian $H$ we have 
$$\ell'(\alpha' ; H) = \ell(\psi^{-1}_*(\alpha') ; H \circ \psi)\,.$$
For a fixed $\alpha' \in QH_*(L')$, this property allows us to compare $\ell'(\alpha'; \,\cdot\,)$ and $ \ell(\psi^{-1}_*(\alpha') ; \,\cdot\,)$ as functions defined on $C^0(M \times [0,1])$ or on $\wt\Ham(M,\omega)$.

The following proposition states that when $\psi$ is Hamiltonian, the difference between these two functions is bounded from above by the \tb{Hofer norm} of $\psi$, defined as
\begin{align*}
  \| \psi \|  = \inf \big\{ \textstyle \int_0^1 \osc_M(H_t) \,dt \, | \, H \co M \times [0,1] \rightarrow \bb R \text{ with } \phi_H = \psi  \big\} \,.
\end{align*}

\begin{prop}\label{theo:lagrangian-nature}
    Let $L$ and $L'$ be Hamiltonian isotopic Lagrangians. Let $\alpha \in QH_*(L)$. For all $\varphi \in \ham(M,\omega)$ such that $\varphi (L)=L'$ we have
    \begin{align*}  
      | \ell (\alpha ; H) - \ell'(\varphi_*(\alpha) ; H) | \leq  \| \varphi \| \quad \text{and} \quad | \ell (\alpha ; \wt\chi) - \ell'(\varphi_*(\alpha) ; \wt\chi) | \leq  \| \varphi \|\,,
    \end{align*}
where $H \in C^0(M \times [0,1])$ and $\wt\chi \in \wt\Ham(M,\omega)$ are arbitrary elements.
\end{prop}

\noindent This property is somewhat surprising since, in the monotone case, one might have expected a bound in terms of Hofer's geometry of the universal cover of the Hamiltonian diffeomorphism group and not of the group itself (compare to the weakly exact case \cite{Leclercq_spectral_invariants_Lagr_FH} and \cite{Monzner_Vichery_Zapolsky_partial_qms_qss_cot_bundles} where the Lagrangian spectral invariants only depend on time--$1$ objects).

In a similar vein we formulate a relation with the Lagrangian Hofer distance (see \cite{Chekanov_Invariant_Finsler_metrics_Lagrangians}). Here we use a variant of this distance defined on the universal cover of the space of Lagrangians which are Hamiltonian isotopic to a given one $L$. It is defined as follows. Recall that given a Hamiltonian $H$, the associated path of Lagrangians $L_t = \phi_H^t(L)$ is exact, meaning the tangent vector to this path at every point $L_t$ is given by an exact $1$-form on $L_t$. Conversely, any exact path of Lagrangians starting at $L$ can be represented in such a way. Given a homotopy class $\wt L$ of exact paths of Lagrangians from $L$ to $L'$ we define the \tb{Hofer length} of $\wt L$ to be
$$\|\wt L\| = \inf \big\{ \textstyle \int_0^1 \osc_M(H_t)\,dt\,|\, [\{\phi_H^t(L)\}_{t \in [0,1]}] = \wt L \big\}\,.$$
Such a homotopy class induces a natural isomorphism $QH_*(L) \simeq QH_*(L')$. Indeed, assume $L_t = \phi_H^t(L)$ is an exact path of Lagrangians. The symplectomorphism $\phi$ induces an isomorphism $(\phi_H)_* \fc QH_*(L) \to QH_*(L')$, see \S\ref{subsec:symplectomorphism_group}. It is not hard to show that this isomorphism only depends on the homotopy class of the path $\{L_t\}_t$ relative to the endpoints.
\begin{prop}\label{prop:relation_Lagr_Hofer_distance}
Let $\wt L$ be a homotopy class of exact paths of Lagrangians connecting $L$ and $L'$, and let $\alpha \in QH_*(L)$, $\alpha' \in QH_*(L')$ be quantum homology classes corresponding to each other by the above natural isomorphism induced by $\wt L$. Then
$$| \ell (\alpha ; H) - \ell'(\alpha' ; H) | \leq  \|\wt L\| \quad \text{and} \quad | \ell (\alpha ; \wt\chi) - \ell'(\alpha' ; \wt\chi) | \leq  \|\wt L\|$$
for any $H \in C^0(M \times [0,1])$ and $\wt\chi \in \wt\Ham(M,\omega)$.
\end{prop}

These properties are proved in \S\ref{sec:proof-Hofer-bound}.

\begin{remark}
Let $\wt\cL$ be the universal cover of the space of Lagrangians Hamiltonian isotopic to $L$, that is an element of $\wt\cL$ is a homotopy class of exact paths of Lagrangians starting at $L$, relative to the endpoints. Over $\wt\cL$ we have the natural local system of algebras $\cQ\cH \to \wt\cL$ whose fiber at a homotopy class $\wt L$ with endoint $L'$ is given by $QH_*(L')$. The above isomorphisms provide a canonical trivialization of this local system, that is $\cQ\cH = QH_*(L) \times \wt \cL$. The above discussion allows us, given $\alpha \in QH_*(L)$, to view the associated spectral invariant as a function on the space
$$\wt \cL \times \wt\Ham(M,\omega)\,.$$
Proposition \ref{prop:relation_Lagr_Hofer_distance}, together with the \textsc{Continuity} property, implies that every such function is Lipschitz with respect to the natural Hofer (pseudo-)metric on this space.
\end{remark}

\subsection{An application to symplectic rigidity}

We now present a sample application of this theory to symplectic rigidity.

In \cite{Entov_Polterovich_Quasi_states_symplectic_intersections} Entov--Polterovich introduced the notion of a symplectic quasi-state. In the following definition $\{\cdot,\cdot\}$ stands for the Poisson bracket and $\| \cdot \|_{C^0}$ for the supremum norm.
\begin{defin}
A quasi-state on $M$ is a functional $\zeta \fc C^0(M) \to \R$ satisfying
\begin{Properties}
\item[Normalization] $\zeta(1) = 1$.
\item[Quasi-linearity] For $F,G \in C^\infty(M)$ with $\{F,G\} = 0$ we have $\zeta(F+G) = \zeta(F) + \zeta(G)$.
\item[Continuity] For $F,G \in C^0(M)$ we have $|\zeta(F) - \zeta(G)\| \leq \|F-G\|_{C^0}$.
\end{Properties}
\end{defin}

They developed a construction of (nonlinear) symplectic quasi-states on $M$ using idempotents in $QH_*(M)$. We refer the reader to \cite{Entov_Polterovich_Symp_QSs_semisimplicity_QH} for details. Briefly, if $QH_*(M) \simeq \cF \oplus \cQ$ as an algebra where $\cF$ is a field, then the spectral invariant $c(e;\cdot)$, where $e \in \cF$ is the unit, has the property that the functional
$$F \mapsto \zeta_e(F) = \lim_{k \to \infty}\frac{c(e;kF)}{k}$$
is a symplectic quasi-state.

They also defined a class of rigid subsets of $M$, namely superheavy subsets with respect to a quasi-state:
\begin{defin}
Let $\zeta \fc C^0(M) \to \R$ be a symplectic quasi-state on $M$. A closed subset $X \subset M$ is called \tb{$\zeta$-superheavy} if for each $c \in \R$
$$\zeta(F) = c \quad \text{for any }F \in C^0(M) \text{ with }F|_X = c\,.$$
If $e \in QH_*(M)$ is an idempotent as above and $\zeta = \zeta_e$, we also say that $X$ is $e$-superheavy.
\end{defin}

In this paper we use Lagrangian spectral invariants to prove the superheaviness of certain Lagrangian submanifolds. Namely, we have the following result, proved in \S \ref{sec:Superheavy} below:
\begin{prop}\label{prop:superheavy_if_not_killed_idem}
Let $e \in QH_*(M)$ be an idempotent as above. If $e\bullet [L]  \neq 0$, then $L$ is $e$-superheavy.
\end{prop}

\begin{remark}This proposition is analogous to, and in certain cases follows from, Proposition 8.1 of \cite{Entov_Polterovich_rigid_subsets_sympl_mfds}. \footnote{We thank Leonid Polterovich for pointing this out to us.} This is essentially because the quantum inclusion $i_L$ of Biran--Cornea \cite{Biran_Cornea_Lagr_QH}, appearing in the formulation of the cited proposition in the guise of the map $j^q$, is dual to the quantum action $QH(M) \ni e \mapsto e\bullet [L] \in QH(L)$.

\end{remark}

We apply this proposition to two examples of monotone Lagrangians. First we consider the monotone product $S^2 \times S^2$ in which there is a monotone Lagrangian torus $L_{S^2\times S^2}$ defined as follows. View $S^2 \subset \R^3$ as the set of unit vectors. Then
$$L_{S^2 \times S^2} = \big\{(x,y) \in S^2 \times S^2\,|\, x\cdot y = -\tfrac 1 2\,, x_3 + y_3 = 0\big\}\,,$$
where $x\cdot y$ is the Euclidean scalar product. See \cite{Eliashberg_Polterovich_Symp_QSs_quadric_Lag_submfds} and references therein.

The other example is the Chekanov monotone torus in $L_{\C P^2} \subset \C P^2$ \cite{Chekanov_Schlenk_Notes_mon_twist_tori}. We refer the reader to the paper by Oakley--Usher \cite{Oakley_Usher_Lagr_submfds}, in which they review various constructions of monotone tori in $S^2 \times S^2$ and in $\C P^2$, and prove that they are all Hamiltonian equivalent. Also see the paper by Gadbled \cite{Gadbled_On_exotic_mon_Lagr_tori}. One possible description of $L_{\C P^2}$ is as follows. Consider the degree $2$ polarization of $\C P^2$ by a conic. This conic is a complex (and thus symplectic) sphere, in which there is a monotone Lagrangian circle, the equator. The Lagrangian circle bundle construction \cite{Biran_Cieliebak_Symp_top_subcritical_mfds, Biran_Lag_non_intersections} in this situation yields $L_{\C P^2}$.
\begin{theo}\label{thm:superheavy_tori}
\begin{itemize}
 \item The Chekanov torus $L_{\C P^2}$ is superheavy with respect to the fundamental class $[\C P^2]$ taken over the ground ring $R = \C$.
 \item Over the ground ring $R = \C$, the quantum homology $QH_*(S^2 \times S^2)$ contains two idempotents $e_\pm$. The torus $L_{S^2 \times S^2}$ is $e_+$-superheavy.
\end{itemize}
\end{theo}
\noindent We prove this in \S \ref{sec:Superheavy}.

\subsection{Relation with previous results}

The properties of the Lagrangian spectral invariants listed in Theorem \ref{thm:main_properties_Intro} have previously appeared in other contexts, and consequently they are more or less standard, see \cite{Viterbo_gfqi, Oh_Construction_sp_invts_Ham_paths_closed_symp_mfds, Monzner_Vichery_Zapolsky_partial_qms_qss_cot_bundles}. The major exception here is the \textsc{Module property}, see Proposition \ref{prop:module_struct_Intro}. An early version of this result was establised in \cite{Monzner_Vichery_Zapolsky_partial_qms_qss_cot_bundles} for the zero section of a cotangent bundle, in case $a = [M]$ and $\alpha = [L]$. An inequality identical to ours in the setting of weakly exact Lagrangians was proved in \cite{Duretic_Katic_Milinkovic}. \textsc{Triangle inequality} for Floer--homological spectral invariants was first proved by Oh \cite{Oh_sympl_topology_action_fcnl_II}, although the inequality proved there included an error term. In the context of Hamiltonian Floer homology the sharp inequality was proved by Schwarz \cite{Schwarz_Action_sp_aspherical_mfds}. A sharp triangle inequality for Lagrangian spectral invariants was proved in \cite{Monzner_Vichery_Zapolsky_partial_qms_qss_cot_bundles} based on the trick due to Abbondandolo--Schwarz \cite{AS10}.

Theorem \ref{thm:spectral_invts_diagonal_Intro} was proved by Leclercq \cite{Leclercq_spectral_invariants_Lagr_FH} for weakly exact $M$. The canonical isomorphism $QH_*(\Delta) = QH_*(M)$ appears, for example, in \cite{Biran_Polterovich_Salamon_Propagation_Ham_dyn_rel_symp_H}.

Various versions of Proposition \ref{prop:superheavy_if_not_killed_idem} have appeared in the literature, see for example \cite{Entov_Polterovich_rigid_subsets_sympl_mfds}, as well as \cite{Biran_Cornea_Lagr_QH}; these results use early ideas appearing in \cite{Albers_On_the_extrinsic_top_Lag_sumbfds}; see also \cite{FOOO_Sp_invts_bulk_QMs_Lag_HF}. Let us point out that the estimates on spectral invariants in these papers use quantities such as depth and height of a Hamiltonian relative to $L$, which in our opinion are somewhat artificial, and they do not have a clear ``symplectic'' meaning. In contrast, Lagrangian spectral invariants do possess such a meaning, and therefore the comparison of Hamiltonian and Lagrangian spectral invariants, used in the proof of Proposition \ref{prop:superheavy_if_not_killed_idem} seems to be a more natural approach.

The superheaviness of the Chekanov torus $L_{\C P^2} \subset \C P^2$ was proved by Wu \cite{Weiwei_Wu_On_an_exotic_Lag_torus_CP_two}. We would like to point out, however, that our proof does not rely on the machinery of toric degenerations. The superheaviness of the torus $L_{S^2 \times S^2}$ was established by Eliashberg--Polterovich \cite{Eliashberg_Polterovich_Symp_QSs_quadric_Lag_submfds}. Again, our proof is based on a more natural approach to estimating Hamiltonian spectral invariants via Lagrangian ones.

\subsection{Discussion}\label{subsec:discussion}

We use quantum homology over arbitrary rings, therefore the issue of orientations plays a crucial role. This was resolved in the context of periodic orbit Floer homology in \cite{Floer_Hofer_Coherent_orientations}, which gave rise to the notion of coherent orientations, which has been used ever since. Coherent orientations for Lagrangian Floer homology are constructed in \cite{FOOO_Lagr_intersection_Floer_thry_anomaly_obstr_II}, and for Lagrangian quantum homology in \cite{Biran_Cornea_Lagr_top_enumerative_invts}; see also \cite{Hu_Lalonde_Relative_Seidel_morphism_Albers_map, Wehrheim_Woodward_Orientations_pseudoholo_quilts}. We use here the notion of canonical orientations, which also appears in \cite{Seidel_The_Book, Welschinger_Open_strings_Lag_conductors_Floer_fctr, Abouzaid_Symp_H_Viterbo_thm}. This has the advantage that the structures inherent to the theory are more transparent. The drawback, of course, is a higher level of abstraction, and that computations are less straightforward.

Assumption \tb{(O)} is the minimal assumption under which the so-called canonical quantum homology can be defined \cite{Zapolsky_Canonical_ors_HF}. The corresponding chain complex distinguishes homotopy classes of disks, rather than their homology classes or symplectic area, which are the more usual equivalence relations appearing in Floer theories. The advantage of the canonical complex is that there are no further assumptions or choices required for its definition, for instance $L$ is not required to be orientable or relatively (S)Pin.

The drawback of the canonical complex is that the resulting homology is oftentimes too small and does not contain the desired classes, for example, it may happen that not all singular homology classes of $L$ appear in it. To remedy this, there is the possibility of forming quotient complexes, where cappings having the same area or the same homology class are identified, in which case the homology of the complex is better behaved. In order to form such quotient complexes, further assumptions and choices need to be made. Even though precise assumptions can be formulated in terms of various obstruction classes, we point out that any sort of quotient complex can be formed once $L$ is assumed to be relatively $\Pin^\pm$ and a relative $\Pin^\pm$ structure has been chosen.

\begin{defin}\label{def:relatively_Pin}
We say that $L$ is \tb{relatively} $\Pin^\pm$ if the class $w^\pm(L)$ is in the image of the restriction morphism $H^2(M;\Z_2) \to H^2(L;\Z_2)$, where $w^+(L) = w_2(TL)$ and $w^-(L) = w_2(TL) + w_1^2(TL)$.
\end{defin}
\noindent In \cite{Zapolsky_Canonical_ors_HF} the notion of a relative $\Pin^\pm$ structure is defined and it is shown how to produce coherent orientations for disks in $M$ with boundary on $L$ and how to use them to define the aforementioned quotient complexes.

Finally let us point out that, as we mentioned in Remark \ref{rem:superheavy_set_enough_one_side_ineq}, assumption \tb{(O)} is strictly weaker than $L$ being relatively $\Pin^\pm$. An example is furnished by the Lagrangian $\R P^5 \subset \C P^5$. Indeed, assumption \tb{(O)} is satisfied since $\pi_2(\R P^5) = 0$. On the other hand $\R P^5$ is not relatively $\Pin^\pm$ since the restriction map $H^2(\C P^5;\Z_2) \to H^2(\R P^5;\Z_2)$ vanishes while the classes $w_2(\R P^5) = w_2(\R P^5) + w_1^2(\R P^5)$ do not. \footnote{We wish to thank Jean-Yves Welschinger for this example.}

\subsection{Overview of the paper}

In \S \ref{sec:background} we review the construction of Lagrangian and symplectic Floer and quantum homology with canonical orientations, as well as the relevant Piunikhin--Salamon--Schwarz isomorphisms. Various algebraic structures and operations appearing in these theories are also described. In \S \ref{sec:definition} we define the Lagrangian spectral invariants and prove basic properties thereof. \S \ref{sec:main-properties} is devoted to the proof of main properties of spectral invariants. In \S\ref{sec:proof-Hofer-bound} we prove the Hofer bounds on spectral invariants formulated in \S\ref{sec:Hofer-bound}. Finally \S \ref{sec:Superheavy} contains the proofs regarding superheavy Lagrangians.

\subsection{Acknowledgements}

RL is partially supported by ANR Grant ANR-13-JS01-0008-01. FZ is partially supported by grant number 1281 from the GIF, the German--Israeli Foundation for Scientific Research and Development, by grant number 1825/14 from the Israel Science Foundation, and wishes to thank RL and Universit\'e Paris--Sud for the opportunity to visit the maths department as a professeur invit\'e, during which substantial portions of the present paper, as well as of \cite{Zapolsky_Canonical_ors_HF}, were written; the research and personal atmosphere there were excellent, and the cafeteria was nice. The authors are grateful to Leonid Polterovich whose suggestion lead them to Proposition \ref{theo:lagrangian-nature}.

\section{Background: Floer homology, quantum homology, and PSS}\label{sec:background}

Here we collect the necessary preliminaries concerning the Floer theories that we need in order to define and study Lagrangian spectral invariants, and establish notation. Floer homology was defined by Floer \cite{Floer_Symp_fixed_pts_holo_spheres, Floer_Morse_thry_Lagr_intersections, Floer_unregularized_grad_flow_symp_action, Floer_Witten_cx_inf_dim_Morse_thry}, and extended by Hofer--Salamon \cite{Hofer_Salamon_HF_Nov_rings}. The Lagrangian case was handled, for instance, by Floer (\emph{loc. cit.}), Oh \cite{Oh_FH_Lagr_intersections_hol_disks_I}, Biran--Cornea \cite{Biran_Cornea_Lagr_QH, Biran_Cornea_Lagr_top_enumerative_invts}, Seidel \cite{Seidel_The_Book}, Fukaya--Ohta--Oh--Ono \cite{FOOO_Lagr_intersection_Floer_thry_anomaly_obstr_I, FOOO_Lagr_intersection_Floer_thry_anomaly_obstr_II}.

The principal reference for all the material in this section is \cite{Zapolsky_Canonical_ors_HF}, where the reader will find a detailed description of all the concepts not defined here; the style of \cite{Seidel_The_Book} will undoubtedly be felt here, and indeed this book served as inspiration for many of the constructions presented herein.

\subsection{Morse theory} We briefly recall the basic notions of Morse theory, mostly in order to establish notation. For a Morse function $f$ on a manifold $Q$ we denote by $\Crit f$ its set of critical points. The index of $q \in \Crit f$ is denoted by $\ind_f q$, or just $\ind q$ if $f$ is clear from the context. For a Riemannian metric $\rho$ on $Q$ we let $\cS_f(q)$, $\cU_f(q)$ be the stable and the unstable manifolds of $f$ at $q$ with respect to the negative gradient flow of $f$ relative to $\rho$; mostly we omit the subscript $_f$. We say that the pair $(f,\rho)$ is \tb{Morse--Smale} if every unstable manifold intersects every stable manifold transversely.

\subsection{Floer homology}

This section is concerned with Lagrangian and periodic orbit Floer homology.

\subsubsection{Lagrangian Floer homology}\label{subsubsec:Lagr_HF}

We let
$$\Omega_L = \big\{\gamma \fc [0,1] \to M\,|\, \gamma(0),\gamma(1) \in L\,,[\gamma] = 0 \in \pi_1(M,L)\big\}\,.$$
Let $D^2 \subset \C$ be the closed unit disk with the standard conformal structure. We consider the punctured disk $\dot D^2 = D^2 \setminus \{1\}$ with the induced conformal structure and endow it with the standard positive end $\epsilon \fc [0,\infty) \times [0,1] \to \dot D^2$ defined by
\begin{equation}\label{eq:std_end}
\epsilon(z) = \frac{e^{\pi z} - i}{e^{\pi z} + i}\,.
\end{equation}
A \tb{capping} of an arc $\gamma \in \Omega_L$ is a smooth map $\wh \gamma \fc (\dot D^2,\partial \dot D^2) \to (M,L)$ which is b-smooth\footnote{Roughly speaking, being b-smooth means that $\wh \gamma \circ \epsilon$ converges to $\gamma$ sufficiently fast together with all derivatives as $s \to \infty$. See also Schwarz's thesis \cite{Schwarz_PhD_thesis} and his book on Morse homology \cite{Schwarz_Morse_H_book}, where this notion appears under a different name.} in the sense of \cite{Zapolsky_Canonical_ors_HF}, and satisfies
$$\wh \gamma \big(\epsilon(s,t) \big) \xrightarrow[s \to \infty]{} \gamma(t)\,.$$
This condition in particular means that if we compactify $\dot D^2$ by adding an interval at infinity according the the end $\epsilon$ to obtain a surface with corners, then $\wh\gamma$ extends to a smooth map defined on this surface. Sometimes we will also use this extension, and by abuse of notation we will denote it by $\wh\gamma$ as well.

We let $\dot D^2_- = D^2 \setminus \{-1\}$ with the induced conformal structure. For a capping $\wh \gamma$ of $\gamma$ we define its opposite as the map
$$-\wh\gamma \fc \dot D^2_- \to M \quad \text{given by}\quad -\wh\gamma(\sigma,\tau) = \gamma(-\sigma,\tau)\,.$$
We call two cappings $\wh\gamma,\wh\gamma'$ equivalent, and denote $\wh\gamma \sim \wh\gamma'$, if the preglued map $\wh\gamma \sharp (- \wh\gamma')$ defines the trivial element in $\pi_2(M,L)$ (see \cite{Zapolsky_Canonical_ors_HF} for details on pregluing). Two pairs $(\gamma,\wh\gamma),(\gamma',\wh\gamma')$ are equivalent if $\gamma = \gamma'$ and $\wh\gamma,\wh\gamma'$ are equivalent cappings. We let the class of $(\gamma,\wh\gamma)$ be denoted by $[\gamma,\wh\gamma]$ and define
$$\wt\Omega_L = \big\{ [\gamma,\wh\gamma]\,|\,\gamma \in \Omega_L\,,\wh\gamma \text{ is a capping of }\gamma \big\}\,,$$
which together with the obvious projection $p \fc \wt\Omega_L \to \Omega_L$ forms a covering space of $\Omega_L$. We usually denote elements of $\wt\Omega_L$ by $\wt\gamma$ and it is understood that it is the class $[\gamma,\wh\gamma]$.

For a time-dependent Hamiltonian $H$ on $M$ we have the action functional
$$\cA_{H:L} \fc \wt\Omega_L \to \R\,, \quad \cA_{H:L}\big(\wt\gamma = [\gamma,\wh\gamma] \big) = \int_0^1 H_t \big(\gamma(t) \big)\, dt - \int_{\dot D^2}\wh\gamma^*\omega\,.$$
Its critical point set $\Crit \cA_{H:L}$ consists of the classes $[\gamma,\wh\gamma]$ for which $\gamma$ is a Hamiltonian arc of $H$, that is solves the ODE $\dot\gamma = X_H(\gamma)$.

We call the Hamiltonian $H$ \tb{nondegenerate} if for every critical point $\wt\gamma \in \Crit \cA_{H:L}$ the linearized map $\phi_{H,*} \fc T_{\gamma(0)}M \to T_{\gamma(1)}M$ maps $T_{\gamma(0)}L$ to a subspace transverse to $T_{\gamma(1)}L$. For such a Hamiltonian there is a well-defined function
$$m_{H:L} \fc \Crit \cA_{H:L} \to \Z$$
called the Conley--Zehnder index.\footnote{We normalize it so that if $H$ is the pullback of a $C^2$-small Morse function $f \fc L \to \R$ to a Weinstein neighborhood of $L$, then for $q \in \Crit f$ and $\wh q$ being the constant capping we have $m_{H:L} \big([q,\wh q]\big) = \ind_f q$, and in general $m_{H:L} \big([\gamma,A\sharp \wh \gamma] \big) = m_{H:L} \big([\gamma,\wh\gamma] \big) - \mu(A)$ for $A \in \pi_2(M,L,\gamma_0)$.}

Fix now a time-dependent $\omega$-compatible almost complex structure $J$ on $M$. In \cite{Zapolsky_Canonical_ors_HF} it is explained how to construct, for $p > 2$, a Fredholm operator
$$D_{\wh\gamma} \fc W^{1,p} \big( \dot D^2,\partial \dot D^2; \wh\gamma^*TM, (\wh\gamma|_{\partial \dot D^2})^*TL \big) \to L^p(\dot D^2;\Omega^{0,1}\otimes \wh\gamma^*TM)$$
given a capping $\wh\gamma$. This operator has index $\ind D_{\wh\gamma} = n - m_{H:L}(\wt\gamma)$. \emph{Ibid.}, it is shown how to assemble these into a Fredholm bundle $D_{\wt\gamma}$ over the space $\wt\gamma$ of cappings\footnote{The set $\wt\gamma = \big\{(\gamma,\wh\gamma)\,|\,\wh\gamma \text{ is a capping of }\gamma \big\}$ is of course a topological space.} in a given equivalence class.\footnote{Strictly speaking, the family $D_{\wt\gamma}$ is parametrized by a space which forms a fibration over $\wt\gamma$ with contractible fibers corresponding to auxiliary choices such as connections and extensions of almost complex structures, but such details are irrelevant here.}

Recall \cite{Zinger_Det_line_bundle_Fredholm_ops} that the determinant line of a Fredholm operator $D$ is defined to be
$$\ddd(D) = \ddd(\ker D) \otimes \ddd(\coker D)^\vee\,,$$
where for a finite-dimensional real vector space $V$ we denote by $\ddd(V)$ its top exterior power.

The determinant line bundle $\ddd(D_{\wt\gamma})$ is well-defined and we have the following foundational lemma, proved in \cite{Zapolsky_Canonical_ors_HF}.
\begin{lemma}
The line bundle $\ddd(D_{\wt\gamma})$ is orientable if and only assumption \tb{(O)} holds. \qed
\end{lemma}
\noindent This allows us to define the rank $1$ free $\Z$-module $C(\wt\gamma)$ whose two generators are the two possible orientations\footnote{Note that $\wt\gamma$ is a connected topological space.} of the line bundle $\ddd(D_{\wt\gamma})$. The Lagrangian Floer complex as a $\Z$-module is
$$CF_*(H:L) = \bigoplus_{\wt\gamma \in \Crit \cA_{H:L}}C(\wt\gamma)\,.$$
It is graded by $m_{H:L}$ and its generators are assigned actions using $\cA_{H:L}$.

For $\wt\gamma_\pm \in \Crit \cA_{H:L}$ we let
\begin{multline*}
\wt\cM(H,J;\wt\gamma_-,\wt\gamma_+) = \big\{u \fc \R \times [0,1] \to M\,|\, \ol\partial_{H,J}u = 0\,, u \big(\R \times \{0,1\} \big) \subset L\,,\\ u(\pm\infty,\cdot) = \gamma_\pm\,,\wh\gamma_-\sharp u \sim \wh\gamma_+ \big\}\,,
\end{multline*}
where
$$\ol\partial_{H,J}u = \partial_s u + J_t(u) \big(\partial_t u - X_H(u)\big)\,.$$
\begin{remark}\label{rema:Floer_eq_grad_flow}
It is well-known that the Floer equation $\ol\partial_{H,J} u = 0$ can be considered as the negative gradient flow equation for the action functional. To see this, define a scalar product on
$$T_{\wt\gamma}\wt\Omega_L = C^\infty \big([0,1],\{0,1\};\gamma^*TM,(\gamma|_{\{0,1\}})^*TL \big)$$
via
$$\langle \xi,\eta \rangle = \int_0^1 \omega \big(\xi(t),J_t\eta(t) \big)\,dt\,.$$
Then the gradient of $\cA_{H:L}$ at $\wt\gamma$ writes
$$\nabla_{\wt\gamma}\cA_{H:L} = J(\gamma) \big(\dot \gamma - X_H(\gamma) \big)\,.$$
\end{remark}

For a fixed nondegenerate $H$ there is a residual subset of almost complex structures $J$ for which the linearized operator $D_u$ is surjective for any element $u \in \wt\cM(H,J;\wt\gamma_-,\wt\gamma_+)$. We call the pair $(H,J)$ a \tb{regular Floer datum} for $L$ if this is the case. Then it follows that $\wt\cM(H,J;\wt\gamma_-,\wt\gamma_+)$ is a smooth manifold of dimension $m_{H:L}(\wt\gamma_-) - m_{H:L}(\wt\gamma_+)$.

The group $\R$ acts on $\wt\cM(H,J;\wt\gamma_-,\wt\gamma_+)$ by shifts in the $\R$ variable; we let $\cM(H,J;\wt\gamma_-,\wt\gamma_+)$ be the quotient. This is a smooth manifold of dimension $m_{H:L}(\wt\gamma_-) - m_{H:L}(\wt\gamma_+) - 1$. If this dimension is $0$, then it is a compact manifold, meaning it is a finite set of points.

When $m_{H:L}(\wt\gamma_-) = m_{H:L}(\wt\gamma_+) + 1$ and $u \in \wt\cM(H,J;\wt\gamma_-,\wt\gamma_+)$, the operator $D_u$ has index $1$, is surjective, and its kernel is canonically isomorphic to $\R$ via the $\R$-action by translations. In particular, since $\ddd(D_u) = \ddd(\ker D_u) \otimes \ddd(\coker D_u)^\vee = \ddd(\R)$, we see that $D_u$ has a canonical orientation, corresponding to the positive orientation of $\R$; we denote it by $\partial_u \in \ddd(D_u)$.

We define the boundary operator
$$\partial_{H,J} \fc CF_k(H:L) \to CF_{k-1}(H:L)$$
as the unique $\Z$-linear map whose matrix element $C(\wt\gamma_-) \to C(\wt\gamma_+)$ is the map
$$\sum_{[u] \in \cM(H,J;\wt\gamma_-,\wt\gamma_+)}C(u)$$
where $C(u) \fc C(\wt\gamma_-) \to C(\wt\gamma_+)$ is an isomorphism of $\Z$-modules, defined as follows. Recall that $C(\wt\gamma_\pm)$ are generated by orientations of the line bundles $\ddd(D_{\wt\gamma_\pm})$. Consider the following string of isomorphisms
$$\ddd(D_{\wh\gamma_-}) \simeq \ddd(D_u) \otimes \ddd(D_{\wh\gamma_-}) \simeq \ddd(D_{\wh\gamma_-} \oplus D_u) \simeq \ddd(D_{\wh\gamma_+})\,,$$
where the first isomorphism maps $\mfo_- \mapsto \partial_u \otimes \mfo_-$, the second one is the direct sum isomorphism, while the third one is the composition of linear gluing and deformation isomorphisms. All these are defined in detail in \cite{Zapolsky_Canonical_ors_HF}. The resulting isomorphism of lines $\ddd(D_{\wh\gamma_-}) \simeq \ddd(D_{\wh\gamma_+})$ gives rise to the desired isomorphism of $\Z$-modules $C(\wt\gamma_-) \simeq C(\wt\gamma_+)$, which is $C(u)$. It can be seen that this isomorphism only depends on the class $[u] \in \cM(H,J;\wt\gamma_-,\wt\gamma_+)$. We have therefore defined the boundary operator $\partial_{H,J}$.
\begin{theo}[\cite{Zapolsky_Canonical_ors_HF}]
$\partial_{H,J}^2 = 0$. \qed
\end{theo}
\noindent This allows us to define the \tb{Lagrangian Floer homology} $HF_*(H,J:L)$ as the homology of the complex $\big(CF_*(H:L),\partial_{H,J} \big)$.

There are \tb{continuation maps} in Floer homology, defined as follows. Let $(H^i,J^i)$, $i=0,1$ be regular Floer data for $L$. Consider a homotopy of Floer data $(H^s,J^s)_{s\in \R}$, that is just a smooth $s$-dependent family of Floer data, which satisfies $(H^s,J^s) = (H^0,J^0)$ for $s \leq 0$, $(H^s,J^s) = (H^1,J^1)$ for $s \geq 1$. Given $\wt\gamma_i \in \Crit \cA_{H^i:L}$, $i=0,1$, we define the solution space
\begin{multline}\label{eq:parametrized_Floer_eq}
\cM\big((H^s,J^s)_s;\wt\gamma_0,\wt\gamma_1\big) = \big \{u \fc \R \times [0,1] \to M\,|\, \ol\partial_{(H^s,J^s)_s}u = 0\,,\\ u \big(\R \times \{0,1\} \big) \subset L\,, u(-\infty,\cdot) = \gamma_0\,,u(\infty,\cdot) = \gamma_1\,, \wh\gamma_0\sharp u \sim \wh\gamma_1 \big\}\,,
\end{multline}
where
\begin{equation}\label{eq:paramd_Floer_operator}
\ol\partial_{(H^s,J^s)_s}u = \partial_s u + J^s(u) \big(\partial_t u - X_{H^s}(u) \big)\,.
\end{equation}
Solutions of $\ol\partial_{(H^s,J^s)_s}u = 0$ can alternatively be considered as the negative gradient flow lines of the $s$-dependent functional $\cA_{H^s:L}$ (see Remark \ref{rema:Floer_eq_grad_flow}).

We call the homotopy $(H^s,J^s)_s$ \tb{regular} if for all $\wt\gamma_i \in \Crit \cA_{H^i:L}$ and $u \in \cM \big((H^s,J^s)_s;\wt\gamma_0,\wt\gamma_1 \big)$, the linearized operator $D_u$ is onto. In this case the solution space \eqref{eq:parametrized_Floer_eq} is a smooth manifold of dimension $m_{H^0:L}(\wt\gamma_0) - m_{H^1:L}(\wt\gamma_1)$. When this dimension vanishes, the solution space is compact, that is a finite set of points.

The continuation map
$$\Phi_{(H^s,J^s)_s} \fc CF_*(H^0:L) \to CF_*(H^1:L)$$
is defined as the linear map with matrix coefficients $C(\wt\gamma_0) \to C(\wt\gamma_1)$, where $m_{H^0:L}(\wt\gamma_0) - m_{H^1:L}(\wt\gamma_1) = 0$. This matrix coefficient equals
$$\sum_{u \in \cM ((H^s,J^s)_s;\wt\gamma_0,\wt\gamma_1 )} C(u)\,,$$
where $C(u) \fc C(\wt\gamma_0) \to C(\wt\gamma_1)$ is an isomorphism, defined as follows. The linearized operator $D_u$ is onto and has index $0$, therefore it is an isomorphism; we let $\mfo_u \in \ddd(D_u)$ be the canonical positive orientation. The isomorphism $C(u)$ is then defined via the composition of the following string of isomorphisms
$$\ddd(D_{\wt\gamma_0}) \simeq \ddd(D_u) \otimes \ddd(D_{\wt\gamma_0}) \simeq \ddd(D_{\wt\gamma_0} \oplus D_u) \simeq \ddd(D_{\wt\gamma_1})\,,$$
the first isomorphism being $\mfo_0 \mapsto \mfo_u \otimes \mfo_0$, the second being the direct sum, and the third the composition of the linear gluing and deformation isomorphisms, similarly to the definition of the boundary operator above.

It is then proved that $\Phi_{(H^s,J^s)_s}$ is a chain map, and, considering homotopies of homotopies, that the resulting map on homology is independent of the chosen homotopy of Floer data, thereby giving rise to continuation morphisms
$$\Phi_{H^0,J^0}^{H^1,J^1} \fc HF_*(H^0,J^0:L) \to HF_*(H^1,J^1:L)\,,$$
which satisfy
$$\Phi_{H^1,J^1}^{H^2,J^2}\Phi_{H^0,J^0}^{H^1,J^1} = \Phi_{H^0,J^0}^{H^2,J^2} \qquad \mbox{and} \qquad \Phi_{H,J}^{H,J} = \id\,.$$
It follows that the continuation morphisms are isomorphisms and we thus have a direct system of graded modules $\big( HF_*(H,J:L) \big)_{(H,J)}$ indexed by regular Floer data, connected by the continuation isomorphisms. The direct limit of this system is the \tb{abstract Lagrangian Floer homology} $HF_*(L)$.

This carries a product. In order to define it, we will review the Floer PDE defined on punctured Riemann surfaces.

Recall \cite{Seidel_The_Book} that a punctured Riemann surface $\Sigma$ is defined as the complement, in a compact Riemann surface $\wh\Sigma$ with boundary, of a finite set $\Theta \subset \wh \Sigma$. Elements of $\Theta$ are called the punctures. There is the notion of an end associated to a puncture. If $\theta$ denotes a boundary puncture of $\Sigma$, an \tb{end associated to it} is a proper conformal embedding
$$\epsilon_\theta \fc [0,\infty) \times [0,1] \to \Sigma \quad \text{or} \quad \epsilon_\theta \fc (-\infty,0] \times [0,1] \to \Sigma$$
such that $\epsilon_\theta^{-1}(\partial \Sigma)$ equals $\R_\pm \times \{0,1\}$. The punctures of $\Sigma$ are given a sign and the end is defined on the positive or on the negative half-strip according to the sign. A \tb{perturbation datum} on $\Sigma$ is by definition a pair $(K,I)$, where $K$ is a $1$-form on $\Sigma$ with values in $C^\infty(M)$, for which $K|_{\partial \Sigma}$ vanishes along $L$, and $I$ is a family of compatible almost complex structures on $M$ parametrized by points of $\Sigma$. Assume we have chosen a regular Floer datum $(H^\theta,J^\theta)$ associated to every puncture $\theta$ of $\Sigma$. Then $(K,I)$ is said to be \emph{compatible with the Floer data} if for every $\theta$
$$\epsilon_\theta^*K = H^\theta_t\,dt \quad \text{and} \quad I \big(\epsilon_\theta(s,t) \big) = J^\theta_t\,.$$

Given a perturbation datum $(K,I)$ on $\Sigma$, we can define the corresponding Floer PDE for smooth maps $u \fc (\Sigma,\partial \Sigma) \to (M,L)$:
\begin{equation}\label{eq:FloerPDE}
\ol\partial_{K,I}u := (du - X_K)^{0,1} = 0\,,
\end{equation}
where $X_K$ is the $1$-form on $\Sigma$ with values in Hamiltonian vector fields on $M$ defined via $\omega \big(X_K(\xi),\cdot \big) = - dK(\xi)$ for a tangent vector $\xi \in T\Sigma$. The symbol $^{0,1}$ denotes the complex antilinear part of a $1$-form on $\Sigma$ with values in a complex vector bundle over it, in this case $u^*TM$.

Assume now $\Sigma$ has a unique positive puncture $\theta$, negative punctures $\theta_i$, and genus $0$. Pick regular Floer data $(H,J)$ and $(H^i,J^i)$ associated to $\theta,\theta_i$, respectively, and choose $\{\wt\gamma_i \in \Crit \cA_{H^i:L}\}_i$ and $\wt\gamma \in \Crit \cA_{H:L}$. We can define
\begin{multline*}
\cM \big(K,I;\{\wt\gamma_i\}_i,\wt\gamma \big) = \big\{u \fc (\Sigma,\partial\Sigma) \to (M,L)\,|\,\ol\partial_{K,I}u=0\,, \\ u\text{ asymptotic to } \gamma_i,\gamma \text{ and } \wh\gamma_1\sharp\dots\sharp u\sim \wh\gamma \big\}\,.
\end{multline*}
We call the perturbation datum $(K,I)$ \tb{regular} if for every choice of $\wt\gamma_i,\wt\gamma$ and every solution $u \in \cM \big(K,I; \{\wt\gamma_i\}_i,\wt\gamma \big)$ the linearized operator $D_u$ is surjective. This is the case for a residual set of compatible perturbation data.

With these preliminaries in place we can define the product on $HF_*(L)$. Let $\Sigma_\star$ be a punctured Riemann surface obtained from $D^2$ by removing three boundary points $\theta^i$, $i=0,1,2$, where the punctures are ordered in accordance with the boundary orientation of $\partial D^2$. We assign $\theta^0,\theta^1$ the negative sign and $\theta^2$ the positive sign. We let $(H^i,J^i)$ be regular Floer data associated to $\theta^i$. We pick a compatible regular perturbation datum $(K,I)$. For $\wt\gamma_i \in \Crit \cA_{H^i:L}$ we consider the space $\cM \big(K,I;\{\wt\gamma_i\}_i \big)$. By assumption, this is a smooth manifold, and its dimension equals $m_{H^0:L}(\wt\gamma_0) + m_{H^1:L}(\wt\gamma_1) - m_{H^2:L}(\wt\gamma_2) - n$. If this dimension is zero, this manifold is compact and therefore consists of a finite number of points.

Assume now that the $\wt\gamma_i$ satisfy $m_{H^0:L}(\wt\gamma_0) + m_{H^1:L}(\wt\gamma_1) - m_{H^2:L}(\wt\gamma_2) = n$. For $u \in \cM \big( K,I;\{\wt\gamma_i\}_i \big)$ the operator $D_u$ is surjective and has index $0$, therefore it is an isomorphism, and as such $\ddd(D_u)$ possesses the canonical positive orientation, which we denote $\mfo_u$.

Similarly to the definition of the matrix elements of the boundary operator above, we now define the matrix element
$$\sum_{u \in \cM (K,I;\{\wt\gamma\}_i )} C(u) \fc C(\wt\gamma_0) \otimes C(\wt\gamma_1) \to C(\wt\gamma_2)\,,$$
where $C(u) \fc C(\wt\gamma_0) \otimes C(\wt\gamma_1) \to C(\wt\gamma_2)$ is an isomorphism defined as follows. Consider the string of isomorphisms
$$\ddd(D_{\wt\gamma_0}) \otimes \ddd(D_{\wt\gamma_1}) \simeq \ddd(D_u) \otimes \ddd(D_{\wt\gamma_0}) \otimes \ddd(D_{\wt\gamma_1}) \simeq \ddd(D_{\wt\gamma_0} \oplus D_{\wt\gamma_1}\oplus D_u) \simeq \ddd(D_{\wt\gamma_2})$$
where the first one maps $\mfo_0\otimes \mfo_1 \mapsto \mfo_u \otimes \mfo_0 \otimes \mfo_1$, the second one is the direct sum isomorphism, and the third one is the composition of linear gluing and deformation isomorphisms. The resulting isomorphism $\ddd(D_{\wt\gamma_0}) \otimes \ddd(D_{\wt\gamma_1}) \simeq \ddd(D_{\wt\gamma_2})$ yields the isomorphism $C(u)$.

We have therefore defined a bilinear map
$$\star_{K,I} \fc CF_k(H^0:L) \otimes CF_l(H^1:L) \to CF_{k+l-n}(H^2:L)\,.$$
In \cite{Zapolsky_Canonical_ors_HF} it is shown that this map is a chain map, therefore it induces a map on homology. Using cobordism arguments, it is shown that this induced map on homology is independent of the conformal structure on $\Sigma_\star$ and of the choice of perturbation datum $(K,I)$, which means that we have a correctly defined bilinear map on homology
$$\star \fc HF_k(H^0,J^0:L) \otimes HF_l(H^1,J^1:L) \to HF_{k+l-n}(H^2,J^2:L)\,.$$
Furthermore, this map respects continuation maps, in the sense that
$$\star \circ \Phi \otimes \Phi = \Phi \circ \star\,,$$
therefore it induces a product $\star$ on the abstract Floer homology $HF_*(L)$.

It is further shown \emph{ibid.} that this product operation is associative and has a unit, thus $HF_*(L)$ is an associative unital algebra.

\subsubsection{Periodic orbit Floer homology}\label{subsec:Ham_HF}

Let
$$\Omega = \big\{x \fc S^1 \to M\,|\, [x] = 0 \in \pi_1(M) \big\}$$
be the contractible loop space of $M$. We view $S^2$ as the Riemann sphere $\C \cup \{\infty\}$, and let $\dot S^2 = S^2 \setminus \{1\}$ be the sphere with one puncture. We endow it with the standard positive end $\epsilon \fc [0,\infty) \times S^1  \to \dot S^2$ defined by
$$\epsilon(z) = \frac{e^{2\pi z} - i}{e^{2\pi z} + i}\,,$$
where we identify $S^1 = \R/\Z$. A capping of $x \in \Omega$ is a b-smooth map $\wh x \fc \dot S^2 \to M$ with $\wh x \big(\epsilon(s,t) \big) \to x(t)$ as $s\to \infty$. Similarly to \S \ref{subsubsec:Lagr_HF} we have the notion of equivalence of two cappings $\wh x,\wh x'$ of $x$, meaning that their pregluing $\wh x \sharp (- \wh x')$ is a contractible sphere; we denote this by $\wh x \sim \wh x'$. We call two pairs $(x,\wh x),(x',\wh x')$ equivalent if $x = x'$ and $\wh x \sim \wh x'$. The class of $(x,\wh x)$ is denoted by $[x,\wh x]$. Let
$$\wt\Omega = \big\{[x,\wh x]\,|\,x \in \Omega\,,\wh x\text{ a capping for }x \big\}\,.$$
This comes with the obvious projection $\wt \Omega \to \Omega$ making it into a covering space. We denote elements of $\wt \Omega$ via $\wt x$ and usually it will stand for $[x,\wh x]$.

For a Hamiltonian $H \fc S^1 \times M \to \R$ we have the action functional
$$\cA_H \fc \wt\Omega \to \R\,,\quad \cA_H \big(\wt x = [x,\wh x] \big) = \int_{S^1} H_t \big( x(t) \big)\, dt - \int_{\dot S^2}\wh x^*\omega\,.$$
Its critical point set $\Crit \cA_H$ is the set of classes $[x,\wh x]$ for which $x$ is a periodic orbit of $H$, $\dot x = X_H(x)$.

We call $H$ nondegenerate if for each $\wt x \in \Crit \cA_H$ the linearized flow $\phi_{H*} \fc T_{x(0)}M \to T_{x(0)}M$ does not have $1$ as an eigenvalue. In this case we have a well-defined Conley--Zehnder index\footnote{Similarly to the Lagrangian case we normalize it to equal the Morse index of a critical point of $H$ in case it is a $C^2$-small Morse function, and to satisfy $m_H \big([x,A\sharp \wh x] \big) = m_H \big([x,\wh x] \big) - 2c_1(A)$ for $A \in \pi_2(M,x(0))$.}
$$m_H \fc \Crit \cA_H \to \Z\,.$$
Given a time-periodic compatible almost complex structure $J$ on $M$, we can define the operator
$$D_{\wh x} \fc W^{1,p}(\dot S^2;\wh x^*TM) \to L^p(\dot S^2;\Omega^{0,1}\otimes \wh x^*TM)\,,$$
and assemble these into a bundle $D_{\wt x}$ of Fredholm operators over the space of cappings in class $\wt x$, similarly to the Lagrangian case. We have
\begin{lemma}[\cite{Zapolsky_Canonical_ors_HF}]
The determinant line bundle $\ddd(D_{\wt x})$ is orientable. \qed
\end{lemma}
We let $C(\wt x)$ be the free $\Z$-module of rank $1$ whose two generators are the two possible orientations of the line bundle $\ddd(D_{\wt x})$. The Floer complex of $H$ as a $\Z$-module is
$$CF_*(H) = \bigoplus_{\wt x \in \Crit \cA_H}C(\wt x)\,.$$
This is graded by $m_H$ and its generators are assigned actions given by $\cA_H$.

We can similarly define the space of solutions of the Floer PDE,
$$\wt\cM(H,J;\wt x_-,\wt x_+)$$
for $\wt x_\pm \in \Crit \cA_H$; we put $\cM(H,J;\wt x_-,\wt x_+) = \wt\cM(H,J;\wt x_-,\wt x_+)/\R$. We call $(H,J)$ a regular Floer datum if for all $\wt x_\pm$ and every solution $u$ in this space the linearized operator $D_u$ is onto. Then this solution space is a smooth manifold of dimension $m_H(\wt x_-) - m_H(\wt x_+)$. The quotient $\cM(H,J;\wt x_-,\wt x_+)$ is a compact manifold whenever it is zero-dimensional, and therefore it is a finite set of points.

Assume that $m_H(\wt x_-) - m_H(\wt x_+) = 1$ and that $u \in \wt \cM(H,J;\wt x_-,\wt x_+)$. Then $D_u$ is surjective and has index $1$. Its kernel carries a canonical orientation and therefore $D_u$ is canonically oriented; we let this orientation be denoted via $\partial_u$.

Similarly to the Lagrangian case, we have the isomorphism
$$C(u) \fc C(\wt x_-) \to C(\wt x_+)$$
for such $u$ given by applying the direct sum, linear gluing and deformation isomorphisms. The matrix element $C(\wt x_-) \to C(\wt x_+)$ of the boundary operator $\partial_{H,J} \fc CF_k(H) \to CF_{k-1}(H)$ is the sum
$$\sum_{[u] \in \cM(H,J;\wt x_-,\wt x_+)} C(u) \fc C(\wt x_-) \to C(\wt x_+)\,.$$
It is proved in \cite{Zapolsky_Canonical_ors_HF} that $\partial_{H,J}^2 = 0$ and we let $HF_*(H,J)$ be the homology of the resulting chain complex $ \big(CF_*(H),\partial_{H,J} \big)$.

Continuation maps
$$\Phi_{H^0,J^0}^{H^1,J^1} \fc HF_*(H^0,J^0) \to HF_*(H^1,J^1)$$
are defined similarly to the Lagrangian case, and it is shown that they behave well with respect to composition and that $\Phi_{H,J}^{H,J} = \id$, implying that they are isomorphisms. We let $HF_*(M)$ be the abstract Floer homology, which is the limit of the direct system of graded modules $\big( HF_*(H,J) \big)_{(H,J)}$ indexed by regular Floer data, and connected by the continuation isomorphisms.

The product
$$* \fc HF_k(H^0,J^0) \otimes HF_l(H^1,J^1) \to HF_{k+l-2n}(H^2,J^2)$$
is defined using the moduli space of solutions of the Floer PDE \eqref{eq:FloerPDE} on the surface $\Sigma_*$ obtained from $S^2$ by removing three points. It is shown to be independent of the conformal structure on $\Sigma_*$ and the perturbation datum used to define it, and moreover to behave well with respect to continuation isomorphisms, that is $* \circ \Phi \otimes \Phi = \Phi \circ *$. This means that we have a well-defined product $*$ on the abstract Floer homology $HF_*(M)$. It is further shown to be supercommutative, associative, and unital.

There is an additional algebraic structure, namely the quantum module action of $HF_*(M)$ on $HF_*(L)$. To describe it, consider the surface $\Sigma_\bullet$ obtained from $D^2$ by removing two boundary and one interior point, where the resulting interior puncture and one of the boundary ones are given the negative sign, while the remaining boundary puncture is given the positive sign. Using solutions of the Floer PDE for $\Sigma_\bullet$, we can define a bilinear map
$$\bullet \fc HF_k(H,J) \otimes HF_l(H^0,J^0:L) \to HF_{k+l-2n}(H^1,J^1:L)\,,$$
which behaves well with respect to continuation isomorphisms, thereby defining a bilinear operation $\bullet \fc  HF_*(M) \otimes HF_*(L) \to HF_*(L)$. This is shown to make $HF_*(L)$ into an algebra over the algebra $HF_*(M)$, meaning we have the structure of an $HF_*(M)$-module on $HF_*(L)$ and in addition the latter is a superalgebra over the former.\footnote{This means that we have $a \bullet (\alpha \star \beta) = (a \bullet \alpha) \star \beta = (-1)^{\text{sign}} \alpha \star (a \bullet \beta)$ with a suitable Koszul sign.}

\subsection{Quantum homology}

Here we briefly describe Lagrangian and symplectic quantum homology. It is advantageous to think of quantum homology as Morse--Bott Floer homology for the zero Hamiltonian and a time-independent almost complex structure. Lagrangian quantum homology was defined by Biran--Cornea \cite{Biran_Cornea_Lagr_QH, Biran_Cornea_Lagr_top_enumerative_invts}. We give here a different formulation.

\subsubsection{Lagrangian quantum homology}\label{subsec:Lagr_QH}

A \tb{quantum datum} for $L$ is a triple $\cD = (f,\rho,J)$ where $f \fc L \to \R$ is a Morse function, $\rho$ is a Riemannian metric on $L$ such that the pair $(f,\rho)$ is Morse--Smale, and $J$ is an $\omega$-compatible almost complex structure. The corresponding quantum complex is the $\Z$-module
$$QC_*(\cD:L) = \bigoplus_{q \in \Crit f} \bigoplus_{A \in \pi_2(M,L,q)}C(q,A)\,.$$
Here $C(q,A)$ is the rank $1$ free $\Z$-module generated by the orientations of the determinant line bundle $\ddd \big(D_A \sharp T\cS(q) \big)$, where the operator family $D_A \sharp T\cS(q)$ is defined as follows. Denote
$$C^\infty_A = \big\{u \in C^\infty(D^2,\partial D^2,1;M,L,q)\,|\,[u] = A \big\}\,.$$
Let $u \in C^\infty_A$ and let
$$D_u \fc W^{1,p} \big(D^2,S^1;u^*TM, (u|_{S^1})^*TL \big) \to L^p(D^2;\Omega^{0,1}\otimes u^*TM)$$
be the formal linearization of the Cauchy--Riemann operator at $u$. See \cite{Zapolsky_Canonical_ors_HF} for a more precise definition. The operators $D_u$ for different $u$ assemble into a Fredholm family $D_A$ over $C^\infty_A$. We let $$D_u \sharp T\cS(q)$$
be the restriction of $D_u$ to the subspace
$$\big\{\xi \in W^{1,p}(D^2,S^1;u^*TM,(u|_{S^1})^*TL)\,|\, \xi(1) \in T_q\cS(q)\big\}\,.$$
Similarly the operators $D_u \sharp T\cS(q)$ assemble into a family $D_A \sharp T\cS(q)$. We have
\begin{lemma}[\cite{Zapolsky_Canonical_ors_HF}]
The determinant line bundle $\ddd \big(D_A \sharp T\cS(q) \big)$ is orientable if and only if assumption \tb{(O)} holds. \qed
\end{lemma}
\noindent Thus we have defined the module $C(q,A)$ and therefore also the module $QC_*(\cD:L)$. The grading on the latter module is determined by requiring the elements of $C(q,A)$ to have degree $\ind_fq - \mu(A)$.

The boundary operator $\partial_\cD \fc QC_*(\cD:L) \to QC_{*-1}(\cD:L)$ is defined using the moduli spaces of pearls defined by Biran--Cornea \cite{Biran_Cornea_Lagr_QH}. When suitable genericity conditions are satisfies (see \emph{ibid.} for their precise formulation), we call the datum $\cD$ regular. See \cite{Zapolsky_Canonical_ors_HF} for a detailed description of $\partial_\cD$ and for a proof that $\partial_\cD^2 = 0$ in the present context. We denote by $QH_*(\cD:L)$ the homology of $\big( QC_*(\cD:L),\partial_{\cD} \big)$. It is the quantum homology associated to the datum $\cD$.

There exist continuation isomorphisms relating the different $QH_*(\cD:L)$, which we define below in \S \ref{subsec:Lagr_PSS} and we let $QH_*(L)$ be the limit of the corresponding direct system of graded modules. This is the \tb{quantum homology} of $L$. We can also define an associative product $\star$ on $QH_*(L)$. It has degree $-n$.

This product admits a unit, constructed as follows. Assume $f$ has a unique maximum $q$. The module $C(q,0)$ in this case is canonically isomorphic to $\Z$. We denote by $1_q \in C(q,0)$ the element corresponding under this isomorphism to $1 \in \Z$. It can be shown that this is a cycle and that it acts as the unit for $\star$ on the chain level. It follows that its class $[L] \in QH_*(L)$ is independent of $\cD$ and that it is the unit for the product.

\subsubsection{Quantum homology of $M$}

This is defined similarly to the Lagrangian case. A quantum datum on $M$ is a triple $\cD = (f,\rho,J)$ where $(f,\rho)$ is a Morse--Smale pair on $M$ and $J$ is a compatible almost complex structure. The quantum complex as a module is
$$QC_*(\cD) = \bigoplus_{q\in \Crit f} \bigoplus_{A \in \pi_2(M,q)}C(q,A)$$
where $C(q,A)$ is the rank $1$ free $\Z$-module generated by orientations of the operator family $D_A \sharp T\cS(q)$. Here $D_A$ is the Fredholm family composed of linearized operators $D_u$ for smooth maps $u \fc S^2 \to M$ in class $A$. This family is always orientable. The family $D_A \sharp T\cS(q)$ has as representatives the operators $D_u \sharp T\cS(q)$, where $D_u \sharp T\cS(q)$ is the restriction of $D_u$ to the subspace of sections in $W^{1,p}(S^2,u^*TM)$ taking values in $T_q\cS_f(q) \subset T_qM$ at $1 \in S^2$. The family $D_A \sharp T\cS(q)$ is orientable as well and therefore we have defined the modules $C(q,A)$. Its elements by definition have degree $\ind_fq - 2c_1(A)$.

The boundary operator $\partial_\cD$ on $QC_*(\cD)$ is in fact just the Morse boundary operator, enriched by the local system $\pi_2(M)$. See \cite{Zapolsky_Canonical_ors_HF} for more details. We thus have the quantum homology $QH_*(\cD)$ as the homology of the resulting chain complex. This is shown to be independent of $\cD$ using continuation isomorphisms in Morse homology; thus we obtain the abstract quantum homology $QH_*(M)$.

It carries a supercommutative associative product $*$ with unit. There is also the bilinear map
$$\bullet \fc QH_*(M) \otimes QH_*(L) \to QH_*(L)$$
of degree $-2n$, which makes $QH_*(L)$ into a superalgebra over $QH_*(M)$ similarly to the case of Floer homology, see \S \ref{subsec:Ham_HF}.

\subsection{Piunikhin--Salamon--Schwarz isomorphisms}

These were first defined in \cite{PSS}, and consequently extended to various other settings \cite{KaticMilinkovic, Albers, Biran_Cornea_Lagr_QH, Zapolsky_Canonical_ors_HF}. They are commonly abbreviated to ``PSS'', and they relate Floer and quantum homologies.

\subsubsection{Lagrangian PSS}\label{subsec:Lagr_PSS}

Fix a regular Floer datum $(H,J)$ and a regular quantum datum $\cD$ for $L$. There are well-defined maps
$$\PSS_\cD^{H,J} \fc QH_*(\cD:L) \to HF_*(H,J:L)\,,$$
$$\PSS_{H,J}^\cD \fc HF_*(H,J:L) \to QH_*(\cD:L)\,,$$
which behave well with respect to continuation morphisms in Floer homology. In fact, the continuation morphisms in quantum homology can be defined in one of two ways. One is direct, as in \cite{Biran_Cornea_Lagr_QH} and the other is using the PSS maps. For instance if $\cD'$ is another regular quantum datum for $L$, we can define the continuation morphism
$$\Phi_\cD^{\cD'} := \PSS_{H,J}^{\cD'} \PSS_\cD^{H,J} \fc QH_*(\cD:L) \to QH_*(\cD':L)\,,$$
and since the PSS maps respect Floer continuation maps, this is independent of the Floer datum $(H,J)$. We will use the following naturally induced isomorphism:
\begin{equation}\label{eq:PSS_isomorphism}
\PSS^{H,J} \fc QH_*(L) \to HF_*(H,J:L)\,.
\end{equation}

We can also show that the PSS maps are unital algebra isomorphisms. In particular, if $(H^i,J^i)$, $i=0,1,2$, are regular Floer data, then the following diagram commutes:
$$\xymatrix{QH_*(L) \ar[r]^-{\star} \ar[d]^{\PSS^{H^0,J^0} \otimes \PSS^{H^1,J^1}} \otimes QH_*(L) & QH_*(L) \ar[d]^{\PSS^{H^2,J^2}} \\ HF_*(H^0,J^0:L) \otimes HF_*(H^1,J^1:L) \ar[r]^-{\star} & HF_*(H^2,J^2:L)}$$

\subsubsection{PSS isomorphism between periodic orbit Floer homology and quantum homology}

Analogously we can define the PSS isomorphisms between the periodic orbit Floer homology $HF_*(H,J)$ for a regular Floer datum $(H,J)$ and the quantum homology $QH_*(\cD)$ for a regular quantum datum $\cD$:
$$\PSS \fc QH_*(\cD) \to HF_*(H,J)\quad \text{and} \quad \PSS \fc HF_*(H,J) \to QH_*(\cD)\,.$$
These similarly respect the Floer continuation maps, and therefore induce the isomorphisms
$$\PSS^{H,J} \fc QH_*(M) \to HF_*(H,J)\quad \text{and} \quad \PSS_{H,J} \fc HF_*(H,J) \to QH_*(M)\,.$$
The induced isomorphism $QH_*(M) \simeq HF_*(M)$ is a graded algebra isomorphism. Moreover, the superalgebra structures of $QH_*(L)$ over $QH_*(M)$ and of $HF_*(L)$ over $HF_*(M)$ are intertwined by the PSS isomorphisms.

\subsection{Duality}\label{subsec:duality}

We first describe duality in Lagrangian Floer homology. To this end we need to introduce Lagrangian Floer cohomology and twisted coefficients. Let $(H,J)$ be a regular Floer datum for $L$. We let the corresponding Lagrangian Floer cochain complex be
$$CF^*(H,J:L) = \bigoplus_{\wt \gamma \in \Crit \cA_{H:L}} C(\wt\gamma)^\vee$$
where for a $\Z$-module $Q$ we let $Q^\vee = \Hom_\Z(Q,\Z)$ be the dual module. Note that if $Q$ is free and of rank $1$, there is a natural isomorphism $Q^\vee = Q$ given by sending a generator $q^\vee$ of $Q^\vee$ to the generator $q$ of $Q$ with $q^\vee(q) = 1$.

The differential $\partial^\vee_{H,J}$ is defined as the operator whose matrix element
$$C(\wt\gamma_+)^\vee \to C(\wt\gamma_-)^\vee$$
is dual to the matrix element $C(\wt\gamma_-) \to C(\wt\gamma_+)$ of $\partial_{H,J}$. We let the differential
$$\delta_{H,J} \fc CF^k(H,J:L) \to CF^{k+1}(H,J:L)$$
be defined as $\delta_{H,J} = (-1)^{k-1}\partial_{H,J}^\vee$.

Clearly we have $\delta_{H,J}^2 = 0$ and we let $HF^*(H,J:L)$ be the cohomology of the resulting cochain complex $\big(CF^*(H,J:L),\delta_{H,J}\big)$.

Another modification we need is twisted coefficients. Assume $\cL$ is a $\Z$-bundle over $\wt\Omega_L$ and for a path $u$ in $\wt\Omega_L$ from $\wt\gamma_-$ to $\wt\gamma_+$ we let $\cP_u \fc \cL_{\wt\gamma_-} \to \cL_{\wt\gamma_+}$ be the parallel transport map. Let
$$CF_*(H,J:L;\cL) = \bigoplus_{\wt\gamma \in \Crit \cA_{H:L}}C(\wt\gamma) \otimes \cL_{\wt\gamma}\,.$$
For $\wt\gamma_\pm \in \Crit \cA_{H:L}$ with $|\wt\gamma_-| = |\wt\gamma_+| - 1$ we define the homomorphism
$$\sum_{[u] \in \cM(H,J;\wt\gamma_-,\wt\gamma_+)}C(u) \otimes \cP_u \fc C(\wt\gamma_-) \to C(\wt\gamma_+)$$
where we use the lift of $u$ to $\wt\Omega_L$ in the parallel transport map $\cP_u$. We let $\partial_{H,J} \otimes \cP$ be the operator whose matrix elements are these homomorphisms. It is straightforward to show that $(\partial_{H,J}\otimes \cP)^2=0$ and we let $HF_*(H,J:L;\cL)$ be the resulting homology. Continuation maps carry over to yield the abstract twisted Lagrangian Floer homology $HF_*(L;\cL)$. Dualizing, we can define the twisted Floer cochain complex $CF^*(H:L;\cL)$ and the corresponding twisted cohomology $HF^*(H,J:L;\cL)$.

For the remainder of this section we let $\cL$ be the $\Z$-bundle over $\wt\Omega_L$ with fiber $|\ddd(T_{\gamma(1)}L)|$ over $\wt\gamma \in \wt\Omega_L$, where for a real line $V$ we let $|V|$ be the free $\Z$-module of rank $1$ whose two generators are the two orientations of $V$.

We now define the dual datum $(\ol H, \ol J)$ as follows:
$$\ol H_t(x) = -H_{1-t}(x)\,, \quad \ol J_t(x) = J_{1-t}(x)\,.$$
It is also regular. There is a bijection
$$\Crit \cA_{H:L} \leftrightarrow \Crit \cA_{\ol H:L}\,,\quad \wt\gamma = [\gamma,\wh \gamma] \mapsto \ol{\wt\gamma} = [\ol \gamma,\ol{\wh\gamma}]\,,$$
where $\ol\gamma(t) = \gamma(1-t)$ and $\ol{\wh\gamma}(\sigma,\tau) = \wh\gamma(\sigma,-\tau)$. We have
$$\cA_{H:L}(\wt\gamma) = -\cA_{\ol H:L}(\ol{\wt\gamma})\,\quad \text{and} \quad m_{H:L}(\wt\gamma) = n - m_{\ol H:L}(\ol{\wt\gamma})\,.$$
Likewise there are natural diffeomorphisms
$$\wt\cM(H,J;\wt\gamma_-,\wt\gamma_+) \leftrightarrow \wt\cM(\ol H,\ol J;\ol{\wt\gamma}_+,\ol{\wt\gamma}_-)\,,\quad u \mapsto \ol u(s,t) = u(-s, 1-t)\,.$$
It is proved in \cite{Zapolsky_Canonical_ors_HF} that for any $\wt\gamma \in \Crit \cA_{H:L}$ there is a natural isomorphism
$$C(\ol{\wt\gamma}) = [C(\wt\gamma) \otimes \cL_{\wt\gamma}]^\vee\,,$$
leading to graded module isomorphisms
\begin{equation}\label{eq:Dual_module_isomorphism}
CF_*(\ol H:L) = CF^{n-*}(H,L;\cL)\,.
\end{equation}
\emph{Ibid.}, it is proved that the following diagram commutes for every $\wt\gamma_\pm \in \Crit \cA_{H:L}$ of index difference $1$ and any $u \in \wt\cM(H,J;\wt\gamma_-,\wt\gamma_+)$:
$$\xymatrix{C(\ol{\wt\gamma}_+)  \ar@{=}[r] \ar[d]^{C(\ol u)} &  [C(\wt\gamma_+) \otimes \cL_{\wt\gamma_+}]^\vee \ar[d]^{(-1)^{n-m_{H:L}(\wt\gamma_-)}(C(u) \otimes \cP_u)^\vee}  \\  C(\ol{\wt\gamma}_-) \ar@{=}[r] &  [C(\wt\gamma_-) \otimes \cL_{\wt\gamma_-}]^\vee } $$
It follows that the identification of modules \eqref{eq:Dual_module_isomorphism} is in fact an identification of complexes, leading to a canonical isomorphism
$$HF_*(\ol H,\ol J:L) = HF^{n-*}(H,J:L;\cL)\,,$$
which we refer to as \tb{the duality isomorphism}. Continuation isomorphisms are intertwined by this identification and therefore we have a canonical isomorphism of abstract homologies $HF_*(L) = HF^{n-*}(L;\cL)$.

Duality in periodic orbit Floer homology is a similarly defined canonical identification $HF_*(M) = HF^{2n-*}(M)$. Here there is no need to twist by the orientation bundle of $M$ because it carries the canonical orientation coming from the symplectic form.

We also have duality isomorphisms in quantum homologies of $L$ and $M$, which take the form of canonical identifications
\begin{equation}\label{eq:duality_isomorphism_QH}
QH_*(L) = QH^{n-*}(L;\cL)\quad \text{and}\quad QH_*(M) = QH^{2n-*}(M)\,.
\end{equation}

The PSS maps can be defined between quantum and Floer cohomologies and homologies with twisted coefficients, and they intertwine these canonical identifications.

Finally we note that there are obvious degree-preserving homomorphisms
$$HF^*(H,J:L) \to HF_*(H,J:L)^\vee\,\quad \text{and} \quad QH^*(L) \to QH_*(L)^\vee\,,$$
coming from the duality pairing; similar homomorphisms exist for the twisted versions. This observation, together with the duality isomorphism above leads to the following Poincar\'e duality pairing:
\begin{equation}\label{eq:Poincare_duality}
\langle \cdot, \cdot \rangle \fc QH_*(L) \otimes QH_{n-*}(L;\cL) \to \Z\,.
\end{equation}

\subsection{Change of coefficients, quotient complexes, Novikov rings}

Here we describe various algebraic operations on the aforementioned complexes.

\subsubsection{Coefficients in a ring of characteristic $2$}

Regardless of whether the assumption \tb{(O)} is satisfied, we can define the Lagrangian Floer and quantum homology of $L$ with coefficients in an arbitrary ring $R$ of characteristic $2$. If $(H,J)$ is a regular Floer datum for $L$, we put
$$CF_*(H:L;R) = \bigoplus_{\wt\gamma \in \Crit \cA_{H:L}}R\cdot \wt\gamma\,,$$
and the boundary operator has matrix elements
$$\#_2 \cM(H,J;\wt\gamma_-,\wt\gamma_+) \fc R\cdot \wt\gamma_- \to R\cdot \wt\gamma_+$$
for $\wt\gamma_\pm \in \Crit \cA_{H:L}$ of index difference $1$. Quantum homology is similarly defined, and all of the above constructions go through. The exception is the duality isomorphism, which takes the simplified form
$$HF_*(L;R) = HF^{n-*}(L;R)\,,$$
since of course there is no need to twist by the orientation bundle of $L$.

\subsubsection{Change of coefficients}

The above complexes are all defined over $\Z$. One can tensor them with an arbitrary ring $R$ to obtain complexes over it. Another possibility is to use twisted coefficients. Since Floer homology is a kind of Morse homology on the path spaces, we can take a flat bundle of $R$-modules (which is just a local system of $R$-modules, see \S \ref{subsubsec:qt_complexes} below) over the suitable path space and use its parallel transport maps to twist the Floer boundary operator. We already used it above in the definition of Floer (co)homology twisted by the orientation bundle of $L$.

\subsubsection{Quotient complexes}\label{subsubsec:qt_complexes}

The reader will have noticed that our Floer and quantum complexes all distinguish between homotopy classes of cappings rather than the more usual equivalence relations in Floer theory, such as equivalence in homology or having equal area. The reason for this is that the complexes can then be defined without any additional choices beyond the Floer or quantum data. In applications, however, it is often necessary to work with weaker equivalence relations for cappings, and this leads to the need of forming quotient complexes. This process is summarized in the current subsection. Even though we use $\Z$ coefficients below, everything goes through over any ring.

We first describe the case of periodic orbit Floer homology and the quantum homology of $M$. Quotient complexes are formed using local systems. Recall that a local system $G$ on a topological space $X$ is a functor from the fundamental groupoid of $X$ to the category of groups. In less abstract terms, we are given a group $G_x$ for every point $x \in X$, and for every homotopy class of paths from $x$ to $y$, a group isomorphism $G_x \simeq G_y$. A subsystem $G' < G$ is given by a subgroup $G'_x < G_x$ for every $x$, such that the isomorphisms $G_x \simeq G_y$ map $G_x'$ to $G_y'$.

The fundamental local system we need to use is $\pi_2(M)$, whose value at $q \in M$ is $\pi_2(M,q)$ and where for $q,q' \in M$ and a homotopy class of paths from $q$ to $q'$ the isomorphism $\pi_2(M,q) \simeq \pi_2(M,q')$ is given by the standard action of paths on homotopy groups \cite{Hatcher_AG}. It acts on $\wt\Omega$, in the following sense. Fix $q \in M$ and let $\wt\Omega_q = \{\wt x\,|\, x(0) = q\}$. Then $\pi_2(M,q)$ acts on $\wt\Omega_q$ by appending spheres to cappings.

In order to form a quotient complex for which spectral invariants still make sense, we need to use the local system $\pi_2^0(M)$, which is the subsystem of $\pi_2(M)$ whose value at $q$ is the group $\pi_2^0(M,q) := \ker \omega$. Fix a subsystem $G < \pi_2^0(M)$. The action of $\pi_2(M)$ on $\wt \Omega$ restricts to an action of $G$. We let $\wt\Omega/G$ be the quotient space. It forms a covering space $\wt\Omega /G \to \Omega$. Note that the fiber of this projection over $x \in \Omega$ consists of the set of equivalence classes cappings for $x$, where two cappings are equivalent if and only if their difference defines an element of $G_{x(0)}$.

The action functional of a Hamiltonian $H$ on $M$, $\cA_H \fc \wt\Omega \to \R$, descends to $\cA_H^G \fc \wt\Omega/G \to \R$. We see that $\Crit \cA_H^G = \Crit\cA_H/G$ in an obvious sense. In \cite{Zapolsky_Canonical_ors_HF} it is described how the quotient map $\Crit \cA_H \to \Crit \cA_H^G$ leads to a map of complexes
$$\big(CF_*(H),\partial_{H,J} \big) \to \big(CF_*^G(H),\partial_{H,J}^G\big)\,,$$
where
$$CF_*^G(H) = \bigoplus_{\wt x^G \in \Crit \cA_H^G}C \big(\wt x^G \big)\,,$$
with $C \big( \wt x^G \big)$ being the limit of the direct system of modules $C(\wt x)$ for $\wt x$ ranging in the class $\wt x^G \in \wt\Omega/G$, and connected by canonical isomorphisms (\emph{ibid.}). The boundary operator $\partial_{H,J}^G$ is well-defined by the requirement that the above quotient map be a chain map. Thus there is an induced graded morphism $HF_*(H,J) \to HF_*^G(H,J)$. These respect continuation morphisms and thus lead to a canonical morphism $HF_*(M) \to HF_*^G(M)$.

It is also shown \emph{ibid.} that the product structure on the Floer complexes descends to a product structure on the quotient modules. The morphism $HF_*(M) \to HF_*^G(M)$ is a morphism of algebras.

We next discuss how the Novikov ring appears in this formulation. In general the local system $\pi_2(M)$ is not constant, meaning the fundamental group $\pi_1(M,q)$ acts on $\pi_2(M,q)$ nontrivially, and this prohibits the existence of a module structure over ``the group ring'' $\Z[\pi_2(M)]$, since $\pi_2(M)$ is not a group, but rather just a groupoid. However in certain cases there is such a structure. Namely, in case $F$ is a constant local system over $M$ and $\phi \fc \pi_2(M) \to F$ is a surjective morphism of local systems such that $G = \ker \phi$, the quotient complex $CF_*^G(H)$ inherits the structure of a module over the group ring $\Z[F]$, and the corresponding homology $HF_*^G(M)$ the structure of an algebra over $\Z[F]$.

Two basic examples of this are as follows. If $F = \Z$ and $\phi(A) = c_1(A)/N_M$, where $N_M$ is the minimal Chern number of $M$, then $G = \pi_2^0(M)$ and the complex $CF_*^{\pi_2^0(M)}(H)$ inherits the structure of a module over the group ring $\Z[\Z] = \Z[t,t^{-1}]$, where $t$ is a formal variable of degree $-2N_M$. The reader will identify this as the Novikov ring, familiar in Floer homology.

Another example is $F = H_2^S(M) = \im \big(\pi_2(M) \to H_2(M;\Z) \big)$ with $\phi$ being the Hurewicz morphism. In this case $G$ is its kernel, and the complex $CF_*^G(H)$ carries the structure of a module over the group ring $\Z[H_2^S(M)]$.

In \cite{Zapolsky_Canonical_ors_HF} there is a description of a completely analogous construction of quotient complexes in quantum homology. The result is the canonical quantum homology $QH_*^G(M)$. Taking quotient complexes respects the PSS isomorphisms, therefore we have canonically $QH_*^G(M) = HF_*^G(M)$, which is an algebra isomorphism; in case $F$ is a constant local system and $G$ is the kernel of a surjective morphism of local systems $\phi \fc \pi_2(M) \to F$, it intertwines the resulting structures of algebras over $\Z[F]$.

We have a similar situation concerning the Lagrangian case, but there are necessary additional structures and assumptions due to the lack of canonical orientations. We briefly describe this, referring the reader to \cite{Zapolsky_Canonical_ors_HF} for details.

Analogously to the absolute case, we need to use local systems on $L$. We have the local system $\pi_2(M,L)$ and its subsystem $\pi_2^0(M,L) = \ker \omega$. There is an action of $\pi_2(M,L)$ on $\wt\Omega_L$. Given a subsystem $G < \pi_2^0(M,L)$, we let $\wt\Omega_L/G$ be the quotient of the induced action of $G$.

Assume $(H,J)$ is a regular Floer datum for $L$. The action functional $\cA_{H:L} \fc \wt\Omega_L \to \R$ descends to $\cA_{H:L}^G \fc \wt\Omega_L/G \to \R$ and we have the obvious quotient map $\Crit \cA_{H:L} \to \Crit \cA_{H:L}^G$. However, this does not immediately lead to the formation of a quotient complex, as we have to identify the modules $C(\wt\gamma)$ for $\wt \gamma$ ranging in a class $\wt\gamma^G \in \Crit \cA_{H:L}^G$, which in general is impossible.

A sufficient condition for this to be possible is to assume that $L$ is relatively $\Pin^\pm$ (Definition \ref{def:relatively_Pin}). If this holds, one can define the notion of a relative $\Pin^\pm$ structure on $L$ \cite{Zapolsky_Canonical_ors_HF}. A choice of such a relative $\Pin^\pm$ structure on $L$ determines a system of orientations on the determinant line bundles\footnote{The operator family $D_A \sharp 0$ is defined as follows. Its representative over $u \in C^\infty_A$ is the restriction of $D_u$ to the subspace of $W^{1,p}\big( D^2,S^1;u^*TM,(u|_{S^1})^*TL \big)$ consisting of sections which vanish at $1$.} $\ddd(D_A\sharp 0)$, which is sufficient to make the necessary identifications of the modules $C(\wt\gamma)$ for $\wt\gamma$ running over a class $\wt\gamma^G \in \wt\Omega_L^G$, see \emph{ibid.}

Fixing a relative $\Pin^\pm$ structure on $L$ we obtain the quotient complex $\big( CF_*^G(H:L),\partial_{H,J}^G \big)$ and the quotient chain map
$$\big(CF_*(H:L),\partial_{H,J} \big) \to \big(CF_*^G(H:L),\partial_{H,J}^G \big)\,.$$
This quotient map determines a unique bilinear map giving the abstract homology $HF_*^G(L)$ the structure of an associative unital algebra for which the natural map $HF_*(L) \to HF_*^G(L)$ is an algebra morphism.

The choice of a relative $\Pin^\pm$ structure yields canonical isomorphisms
$$\ddd(D_u\sharp 0) \simeq \ddd \big(T_{u(1)}L \big)^{\otimes \mu(u)}$$
for smooth maps $u \fc (D^2,S^1) \to (M,L)$. Therefore in order to be able to orient the whole collection of operators $D_u\sharp 0$ with varying $u$, we need to restrict ourselves to the space of smooth maps $u$ with $w_1(u) = 0$, where $w_1 \fc \pi_2(M,L) \to \Z_2$ is the natural morphism of local systems defined by pulling back the first Stiefel--Whitney class $w_1(TL)$ by the boundary homomorphism $\pi_2(M,L) \to \pi_1(L)$. 

A module structure over a Novikov ring is similar to the absolute case. If $F$ is a constant local system on $L$ and $\phi \fc \ker w_1 \to F$ is a surjective morphism of local systems with $G = \ker \phi$, using the relative $\Pin^\pm$ structure one can endow $CF_*^G(H,L)$ with the structure of a module over the group ring $\Z[F]$. Again, there are two basic examples, $\phi \fc  \ker w_1 \to \Z$, $\phi = \mu/\lcm(2,N_L)$, in which case one recovers the structure of a module over the Novikov ring $\Z[\Z] = \Z[t,t^{-1}]$. The other example is the relative Hurewicz morphism $\phi \fc \ker w_1 \to H_2^{D,\text{or}}(M,L)$, where $H_2^{D,\text{or}}(M,L) = \ker \big(H_2^D(M,L) \xrightarrow{w_1} \Z_2\big)$ and $H_2^D(M,L) = \im \big( \pi_2(M,L) \to H_2(M,L;\Z) \big)$ is the image of the relative Hurewicz morphism. In this case $CF_*^G(H:L)$ acquires the structure of a module over $\Z[H_2^{D,\text{or}}(M,L)]$. Note that in general only the ``orientable'' part of $H_2^D(M,L)$ can act on the quotient Floer complex.

Finally we note that the quantum module structure
$$QH_*(M) \otimes QH_*(L) \to QH_*(L)$$
descends to a well-defined module operation
$$QH_*^{G_M}(M) \otimes QH_*^{G_L}(L) \to QH_*^{G_L}(L)$$
provided $G_M < \pi_2^0(M)$ and $G_L < \pi_2^0(M,L)$ are local systems such that $QH_*^{G_L}(L)$ is well-defined (by the above, this is the case if a relative $\Pin^\pm$ structure on $L$ has been chosen), and such that for any $q \in L$ the natural morphism $\pi_2^0(M,q) \to \pi_2^0(M,L,q)$ maps $G_{M,q}$ into $G_{L,q}$. In this case the PSS morphisms intertwine this module structure with the module structure on the quotient Floer homologies. Moreover, if $G_M,G_L$ are kernels of surjective local system morphisms onto constant systems $F_M,F_L$ where the following diagram commutes:
$$\xymatrix{\pi_2^0(M) \ar[r] \ar[d] & \pi_2^0(M,L) \ar[d] \\ F_M \ar[r] & F_L}$$
then by the above $QH_*^{G_M}(M)$ inherits the structure of a $\Z[F_M]$-module, $QH_*^{G_L}(L)$ the structure of a $\Z[F_L]$-module and the induced structures of a $\Z[F_M]$-module on $QH_*^{G_L}(L)$ coincide.

\subsection{The Lagrangian diagonal in $\big( M\times M,\omega \oplus (-\omega) \big)$}\label{subsec:Lagr_diagonal}

Recall that the diagonal $\Delta \subset M \times M$ is a Lagrangian submanifold. It is monotone if and only if $M$ is a monotone symplectic manifold and, in this case, the minimal Maslov number of $\Delta$ coincides with twice the minimal Chern number of $M$. In this subsection we establish a relation between Lagrangian quantum and Floer homology of $\Delta$ and the quantum and Floer homology of $M$.

Let $(H,J)$ be a time-periodic regular Floer datum on $M$. We define a Floer datum $(\wh H,\wh J)$ for $\Delta \subset M \times M$ as follows. First define
$$H^1_t = H_t\,, \quad H^2_t = H_{1-t}\,,\quad J^1_t = J_t\,,\quad J^2_t = -J_{1-t}\quad \text{for }t\in [0,\tfrac 1 2]\,,$$
and put
$$\wh H_t(x,y) = H^1_t(x) + H^2_t(y)\,,\quad \wh J_t(x,y) = J^1_t(x) \oplus J^2_t(y)\,.$$
Then $(\wh H,\wh J)$ is a regular Floer datum\footnote{Note that this datum is defined for $t \in [0,\tfrac 1 2]$. We leave it to the reader to make the necessary adjustments of the Floer theories.} for the diagonal $\Delta$.

We claim that there is a canonical isomorphism of chain complexes
\begin{equation}\label{eq:isomorphism_HF_diagonal}
\big( CF_*(H),\partial_{H,J} \big) = \big(CF_*(\wh H:\Delta),\partial_{\wh H, \wh J} \big)\,,
\end{equation}
preserving the grading and the actions of the generators.

We start with the Hamiltonian orbits. A periodic orbit $\gamma$ of $H$ on $M$ gives rise to a Hamiltonian arc $\Gamma \fc [0,\tfrac 1 2] \to M \times M$ of $\wh H$ with endpoints on $\Delta$, namely
$$\Gamma(t) = \big(\gamma_1(t),\gamma_2(t) \big)\quad \text{where} \quad \gamma_1(t) = \gamma(t)\,,\quad \gamma_2(t) = \gamma(1-t) \,.$$
This defines a bijection between the periodic orbits of $H$ and arcs of $\wh H$ with endpoints on $\Delta$.

If $\wh \gamma$ is a capping of an orbit $\gamma$ of $H$, we can define a capping $\wh \Gamma$ of $\Gamma$, as follows. The capping $\wh \gamma \fc \dot S^2 \to M$ is defined on the punctured Riemann sphere $\dot S^2 = \wh \C \setminus \{1\}$ where $\wh\C = \C \cup \{\infty\}$. There is a conformal embedding $\R \times \R/\Z \to \dot S^2$ given by $z = s+it \mapsto \frac{e^{2\pi z} - i}{e^{2\pi z} + i}$; it is onto $S^2 \setminus \{\pm1\}$. Relative to the coordinates $(s,t)$ induced by this embedding we can define the maps
$$\wh\gamma_1(s,t) = \wh\gamma(s,t)\,,\quad \wh\gamma_2(s,t) = \wh\gamma(s,1-t) \quad \text{for }t \in [0,\tfrac 1 2]\,.$$
The same conformal embedding restricted to $\R \times [0,\tfrac 1 2]$ is onto the punctured disk $D^2 \setminus \{\pm1\}$. Composing the inverse of this embedding with the maps $\wh\gamma_{1,2}$ we get maps $\wh\gamma_{1,2} \fc \dot D^2 \to M$, where we abuse notation slightly. Putting
$$\wh\Gamma(z) = \big(\wh\gamma_1(z),\wh\gamma_2(z) \big)$$
we see that $\wh\Gamma$ is asymptotic to $\Gamma$ at the puncture and that it maps the boundary to $\Delta$. The map thus defined
$$\Crit \cA_{H} \to \Crit \cA_{\wh H:\Delta}\,,\quad \wt\gamma = [\gamma,\wh\gamma] \mapsto \wt\Gamma = [\Gamma,\wh\Gamma]$$
is a bijection, which moreover preserves the action values and the grading.

Recall that
$$CF_*(H) = \bigoplus_{\wt\gamma \in \Crit \cA_H}C(\wt\gamma) \quad \text{and} \quad CF_*(\wh H:\Delta) = \bigoplus_{\wt\Gamma \in \Crit \cA_{\wh H:\Delta}}C(\wt\Gamma)\,.$$
We will now establish a canonical isomorphism $C(\wt\gamma) = C(\wt\Gamma)$ where $\wt\gamma,\wt\Gamma$ correspond to each other by the above bijection. Consider a capping $\wh\gamma$ of $\gamma$ and the corresponding Fredholm operator
$$D_{\wh\gamma} \fc W^{1,p}(\dot S^2,\wh\gamma^*TM) \to L^p(\dot S^2, \Omega^{0,1}\otimes \wh\gamma^*TM)\,.$$
We have seen that $\wh\gamma$ gives rise to two maps $\wh\gamma_{1,2} \fc \dot D^2 \to M$. Similarly, restricting sections of the bundles $\wh\gamma^*TM$ and $\Omega^{0,1} \otimes \wh\gamma^*TM$ to the disks produces pairs of sections which agree on the boundary. This gives rise to natural isomorphisms
$$W^{1,p}(\dot S^2,\wh\gamma^*TM) \simeq W^{1,p} \big(\dot D^2, \partial \dot D^2; \wh\Gamma^*T(M \times M),(\wh\Gamma|_{\partial \dot D^2})^*T\Delta \big)$$
$$L^p(\dot S^2, \Omega^{0,1}\otimes \wh\gamma^*TM) \simeq L^p \big(\dot D^2,\Omega^{0,1} \otimes \wh\Gamma^*T(M\times M) \big)\,.$$
These isomorphisms intertwine the operators $D_{\wh\gamma}$ and $D_{\wh\Gamma}$, thereby inducing a Fredholm isomorphism between them (see \S 2 of \cite{Zapolsky_Canonical_ors_HF}). It follows that the determinant lines of these two operators are canonically isomorphic, which means that we have obtained the desired canonical isomorphism $C(\wt\gamma) = C(\wt\Gamma)$, and consequently the isomorphism \eqref{eq:isomorphism_HF_diagonal} as modules.

We now must show that this module isomorphism is a chain isomorphism. Given a Floer cylinder $u \fc \R \times S^1 \to M$ of $(H,J)$ asymptotic to orbits $\wt\gamma_\pm = \lim_{s \to \pm \infty} u(s,\cdot)$ we can define two maps $u_i \fc \R \times [0,\tfrac 1 2] \to M$, $i=1,2$, via
$$u_1(s,t) = u(s,t)\,,\quad u_2(s,t) = u(s,1-t)\,.$$
Putting $U(z) = \big(u_1(z),u_2(z) \big)$, we obtain a map $U \fc \R\times [0,\tfrac 1 2] \to M \times M$ sending the boundary to $\Delta$. It can be checked that $\partial_{\wh H,\wh J}U = 0$ and that $U$ is asymptotic to $\Gamma_\pm$. This defines a map
$$\wt\cM(H,J;\wt\gamma_-,\wt\gamma_+) \to \wt\cM(\wh H,\wh J;\wt\Gamma_-,\wt\Gamma_+)\,.$$
It can be seen that this map is a bijection. Indeed, given a Floer strip $U$ for $(\wh H, \wh J)$, we can define a continuous map $u \fc \R \times [0,1] \to M$ by gluing the two components of $U$; the resulting map solves the Floer PDE for $(H,J)$ away from the seams. It can also be seen that $u$ is in fact $C^1$ along the seams. It then follows that it is smooth by elliptic regularity.

It is then straightforward to show that the following diagram commutes:
$$\xymatrix{C(\wt\gamma_-) \ar@{=}[d] \ar[r]^{C(u)} & C(\wt\gamma_+) \ar@{=}[d] \\ C(\wt\Gamma_-) \ar[r]^{C(U)} & C(\wt\Gamma_+)}$$
using the gluing and deformation isomorphisms. This shows that the isomorphism \eqref{eq:isomorphism_HF_diagonal} is indeed an isomorphism of chain complexes.

It can also be shown that the above procedure can be adapted to produce a canonical isomorphism $QH_*(M) = QH_*(\Delta)$. Moreover, the PSS maps on both sides are respected by these isomorphisms, in particular we have the following commutative diagram:
$$\xymatrix{QH_*(M) \ar@{=}[r] \ar[d]^{\PSS} & QH_*(\Delta) \ar[d]^{\PSS} \\ HF_*(H,J) \ar@{=}[r] & HF_*(\wh H, \wh J:\Delta)}$$

\subsection{Lagrangian Floer and quantum homology of products}\label{subsec:Lagr_HF_QH_products}

Here we establish a relation between the Floer and quantum homologies of a pair of Lagrangians $L_1,L_2$ with their product.

Let $(M_i,\omega_i)$ be symplectic manifolds and $L_i \subset M_i$ Lagrangian submanifolds, $i=1,2$, so that the product $L_1 \times L_2$ is a monotone Lagrangian of minimal Maslov number at least $2$. This implies that both $L_1$ and $L_2$ are monotone and have minimal Maslov numbers at least $2$. Pick regular Floer data $(H^i,J^i)_i$ for $L_i$. We claim that there is a canonical chain isomorphism
\begin{equation}\label{eq:isomorphism_products}
CF_*(H^1:L_1) \otimes CF_*(H^2:L_2) = CF_*(H^1 \oplus H^2 : L_1 \times L_2)\,.
\end{equation}
It is clear that there is a natural bijection
$$\Crit \cA_{H^1:L_1} \times \Crit \cA_{H^2:L_2} = \Crit \cA_{H^1 \oplus H^2:L_1 \times L_2}$$
preserving actions and the grading. We first need to establish a canonical isomorphism $C(\wt\gamma_1) \otimes C(\wt\gamma_2) = C(\wt\gamma)$ when $\wt\gamma_i \in \Crit \cA_{H^i:L_i}$ and $\wt\gamma$ corresponds to $\wt\gamma_1 \times \wt\gamma_2$ via the bijection. This is done as follows. There is an obvious isomorphism of vector bundles
$$\wh\gamma^*T(M_1 \times M_2) = \wh\gamma_1^*TM_1 \oplus \wh\gamma_2^*TM_2$$
which induces an isomorphism of the corresponding Sobolev spaces of sections, thereby producing a Fredholm isomorphism
$$D_{\wh\gamma} = D_{\wh\gamma_1} \oplus D_{\wh\gamma_2}\,,$$
which implies, using the direct sum isomorphism, that
$$\ddd(D_{\wh\gamma}) = \ddd(D_{\wh\gamma_1}) \otimes \ddd(D_{\wh\gamma_2})\,,$$
therefore that we have the desired isomorphism $C(\wt\gamma) = C(\wt\gamma_1) \otimes C(\wt\gamma_2)$.

We also claim that $(H^1 \oplus H^2,J^1 \oplus J^2)$ is a regular Floer datum for $L_1\times L_2$. Moreover, if $u \fc \R \times [0,1] \to M_1 \times M_2$ is a solution of the Floer PDE for this datum, and it has index $1$, and we let $u_i \fc \R \times [0,1] \to M_i$ be its components, then since $D_u = D_{u_1} \oplus D_{u_2}$ and all the operators are surjective, it follows that precisely one of the operators $D_{u_i}$ is an isomorphism, while the other has index $1$. Therefore Floer trajectories of $(H^1 \oplus H^2,J^1 \oplus J^2)$ of index $1$ have the form $u = (u_1,u_2)$ where one of the $u_i$ is $s$-independent, while the other is an index $1$ Floer trajectory for its datum. This, together with the Koszul signs arising from the direct sum isomorphisms as in \cite{Zapolsky_Canonical_ors_HF}, implies that the boundary operator $\partial_{H^1 \oplus H^2,J^1 \oplus J^2}$ is indeed the tensor product of the operators $\partial_{H^i,J^i}$, in the graded sense. This shows that the above isomorphism of modules \eqref{eq:isomorphism_products} is a chain isomorphism, and so we have a canonical map of Floer homologies
$$HF_*(H^1,J^1:L_1) \otimes HF_*(H^2,J^2:L_2) \to HF_*(H^1 \oplus H^2,J^1 \oplus J^2:L_1\times L_2)\,,$$
which is injective in case the ground ring $R$ is a PID, and an isomorphism if it is a field.

We similarly have a map of quantum homologies
$$QH_*(L_1) \otimes QH_*(L_2) \to QH_*(L_1 \times L_2)\,,$$
and the PSS isomorphisms are respected by these, so that we have a commutative diagram

\noindent \resizebox{\textwidth}{!}{
$$\xymatrix{QH_*(L_1) \otimes QH_*(L_2) \ar[r] \ar[d]_{\PSS \otimes \PSS} & QH_*(L_1 \times L_2) \ar[d]^{\PSS} \\ HF_*(H^1,J^1:L_1) \otimes HF_*(H^2,J^2:L_2) \ar[r] & HF_*(H^1 \oplus H^2,J^1 \oplus J^2:L_1\times L_2)}$$
}

\subsection{Action of the symplectomorphism group}\label{subsec:symplectomorphism_group}

If $\psi \in \Symp(M,\omega)$, let $L' = \psi(L)$. In this subsection we construct a canonical chain isomorphism
$$\psi_* \fc \big(CF_*(H:L),\partial_{H,J} \big) \to \big( CF_*(H^\psi:L'),\partial_{H^\psi,J^\psi} \big)\,,$$
where $(H,J)$ is a regular Floer datum for $L$, $(H^\psi,J^\psi)$ is the Floer datum $H^\psi = H \circ \psi^{-1}$, $J^\psi = \psi_*J$, which is regular for $L'$. Likewise, we construct a chain isomorphism
$$\psi \fc QC_*(\cD:L) \to QC_*(\cD^\psi:L')$$
of quantum complexes where $\cD = (f,\rho,I)$ is a regular quantum datum for $L$ and $\cD^\psi = (f^\psi,\rho^\psi,I^\psi)$ is the quantum datum with $f^\psi = f\circ \psi^{-1}$, $\rho^\psi = \psi_*\rho$, $I^\psi = \psi_*I$, which is regular for $L'$.

First, note that there is an obvious bijection
$$\psi_* \fc \Crit \cA_{H:L} \to \Crit \cA_{H^\psi:L'}$$
given by $\wt\gamma = [\gamma,\wh \gamma] \mapsto \wt\gamma^\psi = [\gamma^\psi = \psi \circ \gamma,\wh\gamma^\psi = \psi\circ\wh\gamma]$. It clearly preserves the actions and Conley--Zehnder indices. Then there is a canonical isomorphism of $\Z$-modules, $C(\wt\gamma) \simeq C(\wt\gamma^\psi)$, constructed in the following manner. The symplectomorphism $\psi$ induces a bundle pair isomorphism
$$\psi_* \fc \big(\wh\gamma^*TM,(\wh\gamma|_{\partial \dot D^2})^*TL \big) \to \big((\wh\gamma^\psi)^*TM,(\wh\gamma^\psi|_{\partial \dot D^2})^*TL\big)$$
which in turn induces the commutative diagram of Fredholm operators and Banach isomorphisms
$$\xymatrix{W^{1,p} \big(\dot D^2,\partial\dot D^2;\wh\gamma^*TM,(\wh\gamma|_{\partial \dot D^2})^*TL\big) \ar[r]^-{D_{\wh\gamma}} \ar[d]^{\psi_*} & L^p(\dot D^2;\Omega^{0,1} \otimes \wh\gamma^*TM) \ar[d]^{\psi_*} \\
W^{1,p} \big(\dot D^2,\partial\dot D^2;(\wh\gamma^\psi)^*TM,(\wh\gamma^\psi|_{\partial \dot D^2})^*TL \big) \ar[r]^-{D_{\wh\gamma^\psi}} & L^p \big(\dot D^2;\Omega^{0,1} \otimes (\wh\gamma^\psi)^*TM \big)}$$
It follows that the operators $D_{\wh\gamma}$, $D_{\wh\gamma^\psi}$ are canonically isomorphic and therefore so are their determinant bundles, whose orientations are the generators of $C(\wt\gamma)$, $C(\wt\gamma^\psi)$. The ensuing isomorphism of $\Z$-modules
$$CF_*(H:L) \simeq CF_*(H^\psi:L')$$
is a chain isomorphism. Indeed, let $\wt\gamma_\pm \in \Crit\cA_{H:L}$ and define
$$\wt\cM(H,J;\wt\gamma_-,\wt\gamma_+) \to \wt\cM(H^\psi,J^\psi;\wt\gamma_-^\psi,\wt\gamma_+^\psi)\,,\quad u \mapsto u^\psi = \psi \circ u\,.$$
This map is easily shown to be a diffeomorphism, which also preserves the canonical orientations in case $\wt\gamma_\pm$ have index difference $1$. It is clear that in the latter case the following diagram commutes:
$$\xymatrix{C(\wt\gamma_-) \ar[r]^{C(u)} \ar[d]^{\psi_*} & C(\wt\gamma_+) \ar[d]^{\psi_*} \\ C(\wt\gamma_-^\psi)\ar[r]^{C(u^\psi)} & C(\wt\gamma_+^\psi)}$$
which implies our claim.

The isomorphism in quantum homology is similarly constructed; the details are left to the reader. It is also obvious that $\psi_*$ intertwines the PSS isomorphisms on homology, giving the commutative diagram
$$\xymatrix{HF_*(H,J:L) \ar[r]^{\psi_*} \ar[d]^{\PSS} & HF_*(H^\psi,J^\psi:L') \ar[d]^{\PSS} \\ QH_*(L) \ar[r]^{\psi_*} & QH_*(L')}$$

\section{Lagrangian spectral invariants: definition} \label{sec:definition}

Recall that in order to be able to define Lagrangian quantum and Floer homology of $L$, we need to fix a ground ring $R$, a local subsystem $G < \pi_2^0(M,L)$, and in case $G$ is nontrivial and $R$ has characteristic different from $2$, a relative $\Pin^\pm$ structure, in order to be able to form quotient complexes as described in \S \ref{subsubsec:qt_complexes}. We may also wish to twist the complex by a flat bundle of $R$-modules. We fix these choices once and for all and omit them from notation throughout \S\S \ref{sec:definition}-\ref{sec:main-properties}.

In this section, we define Lagrangian spectral invariants for continuous time-dependent Hamiltonians on $M$, as well as for elements of $\wt\Ham(M,\omega)$. This is done using the following steps, each one performed in its own subsection.
\begin{itemize}
\item[\S \ref{sec:non-degen-hamilt}:] To any regular Floer datum $(H,J)$ for $L$ we associate a function
  \begin{align*}
    \ell(\, \cdot\;; H,J) : QH_*(L) \setminus \{ 0 \} \rightarrow \bb R \,.
  \end{align*}
\item[\S \ref{sec:cont-property}:] We show that for any non-zero $\alpha \in QH_*(L)$ and any regular Floer data $(H^i,J^i)$ for $L$, $i=0,1$ we have
\begin{equation}\label{eq:continuity_Lagr_sp_invts}
\int_0^1 \min_M(H^1 - H^0)\,dt \leq \ell(\alpha;H^1,J^1) - \ell(\alpha;H^0,J^0) \leq \int_0^1 \max_M(H^1 - H^0)\,dt\,.
\end{equation}
\item[\S \ref{sec:invariance-ell}:] It follows that $\ell(\alpha; H,J)$ only depends on $H$, and we let the common value be denoted by $\ell(\alpha; H)$. Also it follows that $\ell$ can be uniquely extended to a function 
  \begin{align*}
    \ell : QH_*(L) \setminus \{ 0 \} \times C^0 \big(M\times \bb [0,1] \big) \rightarrow \bb R \,.
  \end{align*}
\item[\S \ref{sec:invariance-ell2}:] Finally we show that $\ell$ only depends on the homotopy class of the path $\{ \phi_H^t \}_{t\in[0,1]}$ with fixed endpoints, provided $H$ is normalized, so that $\ell$ descends to a function
\begin{align*}
  \ell : QH_*(L) \setminus \{ 0 \} \times \wt \Ham(M,\omega) \rightarrow \bb R \,.
\end{align*}
\end{itemize}

      \subsection{Nondegenerate Hamiltonians}\label{sec:non-degen-hamilt}

Pick a regular Floer datum $(H,J)$ for $L$ and recall that in particular $H$ is supposed to be nondegenerate. There exists a canonical PSS isomorphism \eqref{eq:PSS_isomorphism}
$$ \PSS^{H,J}: QH_*(L) \rightarrow HF_*(H,J:L) \,.$$

The Floer complex can be naturally filtered by the action. Namely, for $a \in \bb R$ define
\begin{align*}
  CF^a_*(H:L) = \bigoplus_{\substack{\wt\gamma \in \Crit \cA_{H:L}\\ \cA_{H:L}(\wt\gamma) < a}}C(\wt\gamma)\,.  
\end{align*}

If the matrix element $C(\wt\gamma_-) \to C(\wt\gamma_+)$ of the boundary operator $\partial_{H,J}$ is non-zero, then there exists a Floer trajectory $u \in \wt \cM(H,J;\wt\gamma_-,\wt\gamma_+)$. Therefore, since the Floer equation is the negative gradient equation for $\cA_{H:L}$ (see Remark \ref{rema:Floer_eq_grad_flow}), we obtain
$$\cA_{H:L}(\wt\gamma_-) - \cA_{H:L}(\wt\gamma_+) > 0\,,$$
because $u$, which connects critical points of $\cA_{H:L}$ of index difference one, necessarily depends on $s$. Thus the action decreases along Floer trajectories, the boundary operator preserves the above filtration, and $CF^a_*(H:L) \subset CF_*(H:L)$ is a subcomplex. We denote by $HF_*^a(H,J:L)$ the homology of $\big( CF_*^a(H:L),\partial_{H,J} \big)$ and let
$$i_*^a \fc HF_*^a(H,J:L) \to HF_*(H,J:L)$$
be the map induced on homology by the inclusion. The Lagrangian spectral invariant corresponding to a non-zero class $\alpha \in QH_*(L)$ is then defined as
\begin{align}
  \label{eq:definition-SI}
  \ell(\alpha;H,J) = \inf \big\{ a \in \bb R : \PSS^{H,J}(\alpha) \in \im(i_*^a) \big\} \,.
\end{align}

Finally, recall from \S \ref{subsec:Lagr_QH} that there exists a canonical class $[L]\neq 0$ in $QH_*(L)$. The spectral invariant associated to this specific class will be of particular interest to us and we will denote it 
\begin{align}  \label{eq:definition-ellplus}
  \ell_+(H,J)=\ell([L];H,J) \,.
\end{align}

We will also need spectral invariants coming from cohomology. We refer the reader to \S \ref{subsec:duality} for relevant definitions. We note that
$$CF^*_a(H:L) := CF^*(H:L) \bigg / \bigoplus_{\substack{\wt\gamma \in \Crit \cA_{H:L} \\ \cA_{\wt\gamma}> a}}C(\wt\gamma)^\vee$$
is a quotient cochain complex. We let $HF^*_a(H,J:L)$ be the corresponding cohomology and we denote by $i_a^* \fc HF^*(H,J:L) \to HF^*_a(H,J:L)$ the map induced on cohomology by the quotient map $CF^*(H:L) \to CF^*_a(H:L)$. Then for $\alpha^\vee \in QH^*(L)$ we can define
\begin{equation}\label{eq:cohomological_sp_invts}
\ell(\alpha^\vee;H) = \sup \big\{a\in \R\,|\, \PSS^{H,J}(\alpha^\vee) \in \ker i^*_a \big\}\,,
\end{equation}
where $\PSS^{H,J} \fc QH^*(L) \to HF^*(H,J:L)$ is the PSS isomorphism on cohomology.

Finally we will use Hamiltonian spectral invariants. If $(H,J)$ is a time-periodic regular Floer datum on $M$, the Floer complex $CF_*(H)$ can similarly be filtered by action to yield a subcomplex $CF_*^a(H)$ generated by critical points of action $< a$. Letting $i_*^a$ be the map induced on homology by the inclusion $CF_*^a(H) \hookrightarrow CF_*(H)$, we define for any non-zero class $b \in QH_*(M)$
$$c(b;H) = \inf \big\{a \in \R\,|\, \PSS^{H,J}(b) \in \im i_*^a \big\}\,.$$

\subsection{Continuity property of $\ell$}  \label{sec:cont-property}

We consider regular Floer data $(H^i,J^i)$, $i=0,1$. In this section we prove the property \eqref{eq:continuity_Lagr_sp_invts} of the invariants defined by \eqref{eq:definition-SI}.
In order to do so, it is enough to prove that, given $\epsilon > 0$, we can pick a regular homotopy of Floer data $(H^s,J^s)_{s\in \R}$ between $(H^0,J^0)$ and $(H^1,J^1)$, which is stationary for $s \notin (0,1)$, such that for any $a \in \R$, the corresponding continuation morphism
$$\Phi_{(H^s,J^s)_s} \fc CF_*(H^0:L) \to CF_*(H^1:L)$$
maps $CF_*^a(H^0:L)$ into $CF_*^{a+b}(H^1:L)$ where $b = \int_0^1 \max_M(H^1_t-H^0_t)\,dt + \epsilon$. Indeed, in this case we obtain the following commutative diagram
$$\xymatrix{HF_*^a(H^0,J^0:L) \ar[d]^{\Phi} \ar[r]^{i_*^a} & HF_*(H^0,J^0:L) \ar[d]^{\Phi} & & \ar[ll]_-{\PSS^{H^0,J^0}} \ar[lld]^-{\PSS^{H^1,J^1}} QH_*(L) \\ HF_*^{a+b}(H^1,J^1:L) \ar[r]^-{i_*^{a+b}} & HF_*(H^1,J^1:L) }$$
which implies
$$\ell(\alpha;H^1,J^1) - \ell(\alpha;H^0,J^0) \leq \int_0^1 \max_M (H^1_t - H^0_t)\,dt + \epsilon\,.$$
Since this is true for any $\epsilon > 0$, we obtain the right inequality in \eqref{eq:continuity_Lagr_sp_invts}. The other inequality is obtained by exchanging the roles of $H^0$ and $H^1$.

To prove that there is such a regular homotopy of Floer data, consider the special homotopy of Hamiltonians given by
$$K^s_t(x) = H^0_t(x) + \beta(s) \big(H^1_t(x) - H^0_t(x)\big)\,,$$
where $\beta \fc \R \to [0,1]$ is a smooth nondecreasing function which satisfies $\beta(s) = 0$ for $s \leq 0$ and $\beta(s) = 1$ for $s \geq 1$. There is a regular homotopy of Floer data $(H^s,J^s)_s$, stationary for $s \notin (0,1)$, such that
$$\max_{(x,t)\in M \times [0,1]} \left(\frac{\partial H^s_t}{\partial s}(x) - \frac{\partial K^s_t}{\partial s}(x)\right) \leq \epsilon\,.$$
Let $\wt\gamma_i \in \Crit \cA_{H^i:L}$ have the same index and let $v \in \cM \big((H^s,J^s)_s;\wt\gamma_0,\wt\gamma_1 \big)$. Recall that $v$ satisfies the parametrized Floer equation $\ol\partial_{(H^s,J^s)_s}u = 0$, see \eqref{eq:paramd_Floer_operator}. There is a natural lift of $v$ to a path in $\wt\Omega_L$ running from $\wt\gamma_0$ to $\wt\gamma_1$; by abuse of notation we denote this path by $v \fc \R \to \wt\Omega_L$. The parametrized Floer equation can equivalently be written as $v'(s) = -\nabla\cA_{H^s} \big(v(s) \big)$. We therefore have

\begin{align*}
\cA_{H^1:L}(\wt\gamma_1) - \cA_{H^0:L}(\wt\gamma_0) &= \int_{-\infty}^\infty \frac{d}{ds} \cA_{H^s:L} \big(v(s)\big)\,ds \\
&= \int_{-\infty}^\infty \left(d_{v(s)} \cA_{H^s:L} \big(v'(s) \big) + \frac{\partial \cA_{H^s:L}}{\partial s}\big(v(s) \big)\right)\,ds \\
&= - \int_{-\infty}^\infty \|\nabla_{v(s)} \cA_{H^s:L}\|^2\,ds \\
&\qquad\qquad + \int_{\R \times [0,1]} \frac{\partial H^s_t}{\partial s}\big(v(s,t)\big)\,ds\,dt \\
&\leq  \int_{\R \times [0,1]} \frac{\partial H^s_t}{\partial s}\big(v(s,t) \big)\,ds\,dt\,.
\end{align*}

\noindent The last term equals
$$\int_{\R \times [0,1]} \frac{\partial K^s_t}{\partial s} \big(v(s,t) \big)\,ds\,dt + \int_{\R \times [0,1]} \left(\frac{\partial H^s_t}{\partial s}\big(v(s,t) \big) - \frac{\partial K^s_t}{\partial s} \big(v(s,t) \big)\right) \,ds\,dt\,.$$
Here in the first integral we have
\begin{align*}
\int_{\R \times [0,1]} \frac{\partial K^s_t}{\partial s} \big(v(s,t) \big)\,ds\,dt &= \int_{\R \times [0,1]} \beta'(s) (H^1_t - H^0_t) \circ v\,ds\,dt\\
&\leq \int_0^1 \max_M (H^1_t - H^0_t)\,dt \int_{-\infty}^\infty \beta'(s)\,ds \\
&= \int_0^1 \max_M (H^1_t - H^0_t)\,dt\,.
\end{align*}
By the choice of $H^s$ we have
$$\int_{\R \times [0,1]} \left(\frac{\partial H^s_t}{\partial s} - \frac{\partial K^s_t}{\partial s}\right)\circ v \,ds\,dt = \int_{[0,1] \times [0,1]} \left(\frac{\partial H^s_t}{\partial s} - \frac{\partial K^s_t}{\partial s}\right)\circ v \,ds\,dt \leq \epsilon\,.$$
All these estimates together imply
$$\cA_{H^1:L}(\wt\gamma_1) - \cA_{H^0:L}(\wt\gamma_0) \leq \int_0^1\max_M (H^1_t - H^0_t)\,dt + \epsilon\,,$$
which means that the homotopy of Floer data $(H^s,J^s)_s$ has the property claimed above.

\subsection{Invariance of $\ell$ (part 1) and arbitrary Hamiltonians}      \label{sec:invariance-ell}

First, by putting $H^0=H^1$ in \eqref{eq:continuity_Lagr_sp_invts}, we immediately deduce that $\ell(\alpha;H,J)$ does not depend on the choice of the almost complex structure $J$. We will therefore denote the resulting common value by $\ell(\alpha;H)$ from now on. Similarly, we use the notation $\ell_+(H)$ for $\ell_+(H,J)$.
      
Again, by \eqref{eq:continuity_Lagr_sp_invts}, we can define the spectral numbers for an arbitrary continuous function $H : M \times [0,1] \rightarrow \bb R$, as follows. Pick a sequence of smooth nondegenerate Hamiltonians $\{H_n\}_n$ satisfying
$$\lim_{n \to \infty} \|H_n - H\|_{C^0(M \times [0,1])} = 0$$
and define $\ell(\alpha;H)$ as the limit $\lim_{n\to\infty}\ell(\alpha;H_n)$. Equation \eqref{eq:continuity_Lagr_sp_invts} ensures that this limit exists and does not depend on the choice of the sequence $\{H_n\}_n$.

\subsection{Invariance of $\ell$ (part 2)}      \label{sec:invariance-ell2}

We now show that $\ell$ only depends on homotopy classes of Hamiltonian isotopies with fixed endpoints, thereby proving Proposition \ref{prop:sp_invts_htpy_class_path_Intro}. Recall that we call a Hamiltonian $H$ normalized if for all $t\in [0,1]$, $\int_M H_t \, \omega^n=0$.

\begin{defin}\label{defi:equivalent-norm-Hamiltonians}
  Two normalized Hamiltonians $H^0$ and $H^1$ are called equivalent if there is a homotopy $(H^s)_{s \in [0,1]}$ between them so that for all $s$, $H_s$ is normalized and $\phi_{H_s}^1=\phi$.
\end{defin}
The universal cover of $\Ham(M,\omega)$ coincides with the set of equivalence classes of normalized Hamiltonians.

\begin{remark}\label{rema:time-rep-SI}
Before proving the invariance property of $\ell$, recall that when dealing with spectral invariants of the Hamiltonian $H_t$ in question, we may assume that it vanishes for $t$ near $0,1$, see for instance \cite{Polterovich_geometry_of_Ham, Monzner_Vichery_Zapolsky_partial_qms_qss_cot_bundles}. From now on all Hamiltonians are assumed to have this property.
\end{remark}

For a smooth Hamiltonian $H$ we let $\Spec(H:L) = \cA_{H:L}(\Crit \cA_{H:L})$ be its \tb{action spectrum}. The following can be proved similarly to the Hamiltonian case (see \cite{Oh_Construction_sp_invts_Ham_paths_closed_symp_mfds}):
\begin{lemma}\label{lem:spectrum_nowheredense}
The action spectrum $\Spec (H:L) \subset \R$ is a closed nowhere dense subset. \qed
\end{lemma}

The proof of the fact that $\ell$ only depends on the equivalence class of a normalized Hamiltonian boils down to the following lemma.
\begin{lemma}\label{lem:invariance-of-spec}
  Let $H$ be normalized. Then $\Spec(H:L)$ only depends on the equivalence class of $H$.
\end{lemma}
\noindent We denote $\Spec(\wt\phi:L)=\Spec(H:L)$ for any normalized representative $H$ of $\wt \phi \in \wt \Ham(M,\omega)$ and call it the action spectrum of $\wt \phi$.

First let us show that the lemma implies the desired invariance property of $\ell$. Indeed, if $(H^s)_s$ is a normalized homotopy so that for all $s \in [0,1]$, $\phi_{H^s}^1 = \phi$, by Lemma \ref{lem:invariance-of-spec} the spectrum of $\m A_{H^s:L}$ does not depend on $s$. This, combined with the fact that the spectrum is nowhere dense (Lemma \ref{lem:spectrum_nowheredense}) and with the \textsc{Spectrality} and \textsc{Continuity} properties of spectral invariants (see Theorem \ref{thm:main_properties_Lagr_sp_invts} below), shows that for every non-zero $\alpha \in QH_*(L)$, $\ell(\alpha;H^0) = \ell(\alpha;H^1)$.

\begin{proof}[Proof of Lemma \ref{lem:invariance-of-spec}]
Let $(H^s)_{s\in[0,1]}$ be a normalized homotopy so that $\phi_{H^s}^1 = \phi$ for all $s$. We need to show that $\Spec(H^0:L)=\Spec(H^1:L)$. Let us define a map
$$\Crit \cA_{H^0:L} \to \Crit \cA_{H^1:L}$$
as follows. Let $\wt\gamma_0 = [\gamma_0,\wh\gamma_0] \in \Crit \cA_{H^0:L}$ and let $u(s,t) = \phi_{H^s}^t \big(\gamma_0(0) \big)$ for $(s,t) \in [0,1]^2$. Define the glued map $\wh\gamma_1 = \wh\gamma_0 \sharp u$, which we view as a capping of the Hamiltonian arc $\gamma_1 = u(1,\cdot)$ of $H^1$. The above bijection then maps $\wt\gamma_0$ to $\wt\gamma_1 = [\gamma_1,\wh\gamma_1]$ by definition.

The proof will be complete once we show that $\cA_{H^0:L}(\wt\gamma_0) = \cA_{H^1:L}(\wt\gamma_1)$.
This is done in two steps.

\bigskip

\noindent \textit{Step 1.} First we show that
\begin{align} \label{eq:step1-invarinace-of-spec}
  \m A_{H^1:L}(\wt\gamma_1) - \m A_{H^0:L}(\wt\gamma_0) =  \int_{I^2} (\del_s H_t^s) \circ u \,ds\,dt 
\end{align}
with $I=[0,1]$. Let $\gamma_s = u(s,\cdot)$, $\wh\gamma_s = \wh\gamma_0 \sharp u|_{[0,s]\times [0,1]}$, and $\wt\gamma_s = [\gamma_s,\wh\gamma_s]$. We have
\begin{align*}
\m A_{H^1:L}(\wt\gamma_1) - \m A_{H^0:L}(\wt\gamma_0) &= \int_0^1 \frac{d}{ds} \cA_{H^s:L}(\wt\gamma_s)\,ds \\
&= \int_0^1 d_{\wt\gamma_s}\cA_{H^s:L}(\partial_s\wt\gamma_s)\,ds + \int_{I^2}(\partial_s H_t^s) \circ u\,ds\,dt\,.
\end{align*}
Since by construction $\wt\gamma_s$ is a critical point of $\cA_{H^s:L}$, the first summand in the last expression vanishes, and we have proved \eqref{eq:step1-invarinace-of-spec}.\\
 
\noindent\textit{Step 2.} We prove that
\begin{equation}\label{eq:integral_H_over_u_p}
\int_{I^2} (\del_s H_t^s) \big(\phi^t_{H^s}(p) \big) \,ds\,dt
\end{equation}
does not depend on $p$.

Consider the time-periodic Hamiltonian $K^s = H^s \sharp \ol{H^0}$. For $p \in M$ let $u_p(s,t) = \phi_{H^s}^t(p)$. Define $\delta_s^p = u_p(s,\cdot) \sharp \ol{u_p(0,\cdot)}$. Then $\delta_s^p$ is a periodic orbit of $K^s$ and the map $\wh\delta_s^p = u_p|_{[0,s]\times[0,1]}$ can be viewed as a capping of $\delta_s^p$. We let $\wt\delta_s^p = [\delta_s^p,\wh\delta_s^p]$ and observe that this is a critical point of $\cA_{K^s}$.
Noting that $\cA_{K^0}(\wt\delta_0^p) = 0$, we have:
\begin{align*}
\cA_{K^1}(\wt\delta_1^p) &= \int_0^1 \frac{d}{ds} \cA_{K^s}(\wt\delta_s^p)\,ds \\
&= \int_0^1 d_{\wt\delta_s^p}\cA_{K^s}(\partial_s\wt\delta_s^p)\,ds + \int_{I^2}(\partial_s H^s_t) \circ u_p \,ds\,dt\,.
\end{align*}
Since $\wt\delta_s^p$ is a critical point of $\cA_{K^s}$, this first summand in the last expression vanishes, and so we obtain that the sought-for integral  \eqref{eq:integral_H_over_u_p} in fact equals the value of the action functional $\cA_{K^1}$ at the critical point $\wt\delta_1^p$.

Note that $M \ni p \mapsto \wt\delta_1^p$ is a smooth embedding of $M$ into the set of critical points of $\cA_{K^1}$. Since a functional is constant on a connected submanifold of the set of its critical points, we conclude that our integral \eqref{eq:integral_H_over_u_p} is indeed independent of $p$.

\bigskip

\noindent\textit{End of the proof.} Using the normalization condition, the fact that $\phi_{H^s}^t$ is a symplectomorphism for all $s,t$, and Steps 1, 2, we obtain finally:
\begin{align*}
0 &= \int_0^1 dt \int_M\big(H_t^1(p) - H_t^0(p)\big) \, \omega^n_p \\
&= \int_{I^2} \bigg(\int_M\partial_sH^s_t(p)\,\omega^n_p \bigg) \, ds \, dt \\
&= \int_{I^2} \bigg(\int_M\partial_sH^s_t \big(\phi_{H^s}^t(p) \big)\,\omega^n_p \bigg) \, ds \, dt \\
&= \int_{I^2} \partial_sH^s_t \big(\phi_{H^s}^t(p) \big)\,ds\,dt \cdot \int_M \omega^n\\
&= \big( \cA_{H^1:L}(\wt\gamma_1) - \cA_{H^0:L}(\wt\gamma_0) \big) \int_M \omega^n\,,
\end{align*}
thereby finishing the proof of Lemma \ref{lem:invariance-of-spec}.
\end{proof}

This concludes the proof of invariance. In view of this property, we have actually defined 
\begin{align*}
  \ell : QH_*(L) \setminus \{ 0 \} \times \wt \Ham(M,\omega) \rightarrow \bb R
\end{align*}
with $\ell(\alpha;\wt\phi) = \ell(\alpha;H)$ for any normalized Hamiltonian $H$ representing $\wt\phi$. We denote $\ell_+(\wt\phi) = \ell([L];\wt\phi)$.

   \section{Spectral invariants: main properties}   \label{sec:main-properties}

We maintain the choices made at the beginning of \S \ref{sec:definition}.

\subsection{Quantum valuation} \label{sec:quantum-valuation}

Here we introduce the valuation in Lagrangian quantum homology. It is analogous to the valuation introduced in \cite{Entov_Polterovich_Calabi_quasimorphism_quantum_homology}, and serves a similar purpose.

Fix a regular quantum datum $\cD$ for $L$. The quantum valuation
$$\nu \fc QH_*(\cD:L) \to \R \cup \{-\infty\}$$
is defined as follows. If $\alpha \in QH_*(\cD:L)$ is non-zero, let $C \in QC_*(\cD:L)$ be a chain representing it. We define the valuation of $C$ to be 
$$\nu(C)=\tfrac{1}{\sfA} \max \big\{ -\omega(A) \,|\, \text{the component of }C \text{ in }C(q,A) \text{ is not }0 \big\}\,.$$
Then we define
$$\nu(\alpha) = \inf \{ \nu(C) \,|\, [C] = \alpha \}$$
and we put $\nu(0)=-\infty$. Proposition \ref{prop:valuation_sp_invts} below will show that $\nu$ equals the spectral invariant of the zero Hamiltonian, which implies that $\nu$ in fact descends to a well-defined function
$$\nu \fc QH_*(L) \to \R \cup \{-\infty\}\,.$$

\begin{remark}
When $N_L = \infty$, it follows from the definitions that $\nu(\alpha) = 0$ for any nonzero $\alpha$.
\end{remark}

\begin{lemma}
If $\alpha \neq 0$, then $\nu(\alpha) \in \R$.
\end{lemma}
\begin{proof}
Just like in \cite{Entov_Polterovich_Calabi_quasimorphism_quantum_homology}, we can show that $\nu(\alpha)$ is the maximum of $\nu(\alpha_k)$ where $\alpha_k \in QH_k(\cD:L)$ is the degree-$k$ component of $\alpha$. Therefore it suffices to show that $\nu(\alpha)$ is finite for homogeneous $\alpha$. In this case, due to the monotonicity of $L$, the direct summands $C(q,A)$ of $QC_k(\cD:L)$ have the property that there are only finitely many values that $-\omega(A)/\sfA$ can attain. Therefore any nonzero cycle $C \in QC_k(\cD:L)$ has valuation $\nu(C)$ which is at least the minimal such value. It follows that $\nu(\alpha)$ is also at least this minimal value.
\end{proof}

There is a more natural way to define $\nu$ in the context of this paper. Define the function
$$\cA_0 \fc \big\{(q,A)\,|\, q\in \Crit f\,,A \in \pi_2(M,L,q)\big\} \to \R \quad \text{by} \quad \cA_0(q,A) = -\omega(A)\,.$$
Morally speaking, since quantum homology is the Morse--Bott Floer homology of the zero Hamiltonian, this is the corresponding action functional.

Using $\cA_0$ we can now define a subcomplex $QC_*^a(\cD:L)$ spanned by those generators of $QC_*(\cD:L)$ with $\cA_0 < a$. This is indeed a subcomplex since it follows from the definition of the quantum boundary operator that if its matrix element $C(q,A) \to C(q',A')$ is non-zero, then necessarily $\omega(A') \geq \omega(A)$. We let $QH_*^a(\cD:L)$ be the homology of this subcomplex and let
$$i_Q^a \fc QH_*^a(\cD:L) \to QH_*(\cD:L)$$
be the map induced by the inclusion $QC_*^a(\cD:L) \hookrightarrow QC_*(\cD:L)$. It is then obvious from the definitions that
\begin{align*}
  \nu(\alpha) = \tfrac 1 {\sfA}\inf \big\{ a \in \bb R \,|\, \alpha \in \im (i_Q^a) \big\} \,.
\end{align*}
Intuitively, this means that $\nu(\alpha)\sfA$ is the spectral invariant of the zero Hamiltonian. We now show that this analogy is in fact precise.
\begin{prop}\label{prop:valuation_sp_invts}
We have
$$\int_0^1 \min_M H_t \,dt \leq \ell(\alpha;H) - \nu(\alpha)\sfA \leq \int_0^1 \max_M H_t \,dt \,.$$
In particular $\nu(\alpha)\sfA = \ell(\alpha;0)$ and therefore $\nu$ is independent of the quantum datum $\cD$ used to define it.
\end{prop}
\begin{proof}
Fix a regular Floer datum $(H,J)$. We claim that for any given $\epsilon > 0$ we can choose a suitable perturbation datum needed in the defintion of the PSS morphism
$$QC_*(\cD:L) \to CF_*(H:L)$$
so that for any $a \in \R$ it restricts to a map
$$QC_*^a(\cD:L) \to CF_*^{a + b}(H:L)$$
where $b = \int_0^1 \max_M H_t \,dt + \epsilon$. Let us first show how this implies the right inequality of the proposition. It follows that there is a commutative diagram
$$\xymatrix{QH_*^a(\cD:L) \ar[r] \ar[d]^{i_Q^a} & HF_*^{a+b}(H,J:L) \ar[d]^{i_*^{a+b}}\\ QH_*(\cD:L) \ar[r]^-{\PSS_\cD^{H,J}} & HF_*(H,J:L)}$$
The definition of spectral invariants and the valuation then imply that
$$\ell(\alpha;H) \leq \sfA \cdot \nu(\alpha) + \int_0^1 \max_M H_t \,dt + \epsilon\,.$$
Since this holds for any $\epsilon$, we obtain the right inequality above. The other inequality is similarly proved using the inverse PSS morphism.

To prove the assertion, we need to use the definition of the PSS isomorphism described in \cite{Zapolsky_Canonical_ors_HF}. We do not need the full-detail version of the definition, but rather enough of it in order to prove action estimates. We fix $\epsilon > 0$.

Let $q \in \Crit f$, $A \in \pi_2(M,L,q)$, $\wt\gamma \in \Crit \cA_{H:L}$, and assume that the matrix element $C(q,A) \to C(\wt\gamma)$ of the PSS morphism does not vanish. This means that there are maps
$$u_i \fc (D^2,S^1) \to (M,L)\,,\quad i=1,\dots,k\,,\quad \text{and}\quad u_0 \fc (\dot D^2,\partial \dot D^2) \to (M,L)\,,$$
with the following properties. Each $u_i$ for $i>0$ is $I$-holomorphic, there is a gradient trajectory of $f$ running from $q$ to $u_1(-1)$, and gradient trajectories of $f$ running from $u_i(1)$ to $u_{i+1}(-1)$ for $i < k$, and from $u_k(1)$ to $u_0(-1)$. The map $u_0$ satisfies the Floer PDE on $\dot D^2$ where the perturbation datum is chosen as follows. There is a conformal identification $\R \times [0,1] \simeq \dot D^2 \setminus \{-1\}$ given by the same formula as \eqref{eq:std_end} and with respect to the coordinates $(s,t)$ induced on $\dot D^2 \setminus \{-1\}$ via this identification, the perturbation datum takes the form of a homotopy of Floer data $(H^s,J^s)_{s \in \R}$, which is stationary for $s \notin (0,1)$, and such that $(H^s,J^s) = (0,I)$ for $s \leq 0$, $(H^s,J^s) = (H,J)$ for $s \geq 1$, and where moreover there exists a nondecreasing smooth function $\beta \fc \R \to [0,1]$ with $\beta(s) = 0$ for $s \leq 0$ and $\beta(s) = 1$ for $s \geq 1$ such that
\begin{equation}\label{eq:cond_Floer_htpy_valuation}
\max_{(x,t) \in M \times [0,1]}\left(\frac{\partial_sH^s_t}{\partial s}(x) - \beta'(s)H_t(x)\right) \leq \epsilon\,.
\end{equation}
Moreover, by definition the capping $\wh \gamma$ in $\wt\gamma$ has the property that $\wh\gamma \sim A \sharp u_1 \sharp \dots \sharp u_k \sharp u_0$ ($\sim$ being the equivalence relation on the set of cappings, see \S \ref{subsubsec:Lagr_HF}), where the concatenation is performed along the pieces of gradient trajectories of $f$ mentioned above.
We have
\begin{align*}
\cA_{H:L}(\wt\gamma) &= \int_0^1 H_t \big(\gamma(t) \big)\,dt - \int \wh\gamma^*\omega\\
&= \int_0^1 H_t \big(\gamma(t) \big)\,dt - \omega(A) - \sum_{i=1}^k \int u_i^*\omega - \int u_0^*\omega\\
&\leq -\omega(A) + \int_0^1 H_t \big(\gamma(t) \big)\,dt - \int u_0^*\omega\,,
\end{align*}
since the $u_i$ are holomorphic. We claim that
$$\int_0^1 H_t \big(\gamma(t) \big)\,dt - \int u_0^*\omega \leq \int_0^1 \max_M H_t\,dt + \epsilon\,.$$
This follows formally using the same argument as in \S \ref{sec:cont-property}, as we will now demonstrate. Observe that $u_0$, written in the coordinates $(s,t)$ on $\dot D^2 \setminus \{-1\}$, satisfies the parametrized Floer equation for the homotopy of Floer data $(H^s,J^s)_s$. The point $u_0(-1)$ can be considered as a Hamiltonian arc of the zero Hamiltonian, and we can cap it using the constant map at the same point. Since the homotopy $(H^s)_s$ satisfies the condition \eqref{eq:cond_Floer_htpy_valuation}, we have, using the argument in \S \ref{sec:cont-property}:
\begin{multline*}
\int_0^1 H_t \big(\gamma(t) \big)\,dt - \int u_0^*\omega = \cA_{H:L}\big([\gamma,u_0] \big) - \cA_{0:L}\big([u_0(-1),u_0(-1)]\big) \\ \leq \int_0^1 \max_M (H_t - 0) + \epsilon\,,
\end{multline*}
which is precisely what we wanted to show.
\end{proof}

\subsection{Spectral invariants of continuous Hamiltonians}
\label{sec:SI-hamiltonian-funct}

In this section we prove the following theorem, which generalizes Theorem \ref{thm:main_properties_Intro} of the introduction and includes Proposition \ref{prop:module_struct_Intro}.

\begin{theo}\label{thm:main_properties_Lagr_sp_invts}
  Let $L$ be a closed monotone Lagrangian of $(M,\omega)$ with minimal Maslov number $N_L \geq 2$. The function
$$\ell \fc QH_*(L) \times C^0 \big(M\times [0,1] \big) \rightarrow \bb R \cup \{ -\infty \}$$
constructed in \S \ref{sec:definition} satisfies the following properties. 
\begin{Properties} 
   \item[Finiteness] $\ell(\alpha;H) = - \infty$ if and only if $\alpha = 0$.
   \item[Spectrality] For $H \in C^\infty \big(M \times [0,1]\big)$ and $\alpha \neq 0$, $\ell(\alpha;H) \in \Spec(H:L)$.
   \item[Ground ring action] For $r \in R$, $\ell(r \cdot \alpha;H) \leq \ell(\alpha;H)$. In particular, if $r$ is invertible, then $\ell(r \cdot \alpha;H) = \ell(\alpha;H)$.
   \item[Symplectic invariance] Let $\psi \in \Symp(M,\omega)$ and $L'=\psi(L)$. Let
   $$\ell' \fc QH_*(L') \times C^0\big(M\times [0,1] \big) \to \R \cup \{-\infty\}$$
   be the corresponding spectral invariant. Then $\ell(\alpha;H) = \ell'(\psi_*(\alpha);H \circ \psi^{-1})$. 
   \item[Normalization] If $c$ is a function of time then $$\ell(\alpha;H+c)=\ell(\alpha;H) + \int_0^1 c(t) \,dt\,.$$ We have $\ell(\alpha;0) = \nu(\alpha) \sfA$ and $\ell_+(0)=0$.
   \item[Continuity] For any $H$ and $K$, and $\alpha \neq 0$:
   $$\int_0^1 \min_M (K_t - H_t) \,dt \leq \ell(\alpha;K) - \ell(\alpha;H) \leq \int_0^1 \max_M (K_t - H_t) \,dt \,.$$ 
   \item[Monotonicity] If $H \leq K$, then $\ell(\alpha;H) \leq \ell(\alpha;K)$.
   \item[Triangle inequality] For all $\alpha$ and $\beta$, $\ell(\alpha \star \beta; H \sharp K) \leq \ell(\beta;H) + \ell(\alpha;K)$.
   \item[Module structure] Assume that $H^2$ is $1$-periodic. For all $a \in QH_*(M)$ and $\alpha \in QH_*(L)$, $\ell(a \bullet \alpha;H^1 \sharp H^2) \leq c(a;H^2) + \ell(\alpha;H^1)$.   
   \item[Duality] For $\alpha \in QH_*(L)$ let $\alpha^\vee \in QH^{n-*}(L;\cL)$ be the element corresponding to $\alpha$ under the duality isomorphism \eqref{eq:duality_isomorphism_QH}. Then we have
   $$-\ell(\alpha;\ol H) = \ell(\alpha^\vee;H) \leq \inf \big\{\ell(\beta;H)\,|\, \beta \in QH_{n-*}(L;\cL) : \langle \alpha^\vee, \beta\rangle \neq 0 \big\}\,.$$
   In case the ground ring $R$ is a field and the Floer complexes of nondegenerate Hamiltonians are finite-dimensional in every degree, the last inequality is an equality.
\end{Properties}
~\\Recall from \S \ref{subsubsec:qt_complexes} that if $G$ is the kernel of a surjective local system morphism $\ker w_1 \to F$, where $F$ is a group, then $QH_*^G(L)$ is a module over the group ring $R[F]$. We have

\begin{Properties}
    \item[Novikov action] For $A \in F$, we have $\ell(A \cdot \alpha;H) = \ell(\alpha;H) - \omega(A)$. 
    \item[Lagrangian control] If for all $t$, $H_t|_L = c(t) \in \bb R$ (respectively $\leq$, $\geq$), then $$\ell(\alpha;H)=\int_0^1 c(t) \,dt + \nu(\alpha) \sfA\quad  (\text{respectively }\leq, \geq)\,.$$ So that, for all $H$:
$$\int_0^1 \min_L H_t \,dt \leq \ell(\alpha;H) - \nu(\alpha) \sfA \leq \int_0^1 \max_L H_t \,dt \,.$$
~\\Recall that $\ell([L];H)$ is denoted  $\ell_+(H)$ and that we assume $[L] \neq 0$.

   \item[Non-negativity] $\ell_+(H) + \ell_+(\overline{H}) \geq 0$.
   \item[Maximum] $\ell(\alpha;H) \leq \ell_+(H)+\nu(\alpha)\sfA$. 
\end{Properties}
\end{theo}

\subsubsection{Proof of Theorem \ref{thm:main_properties_Lagr_sp_invts}}
\label{sec:proofs}

First, notice that the fact that $\ell(\alpha;H) \in \bb R$ for any $\alpha \neq 0$ and any $H$ easily comes from \textsc{Normalization} and \textsc{Continuity}.\footnote{Alternatively, the finiteness of $\ell(\alpha;H)$ for nondegenerate $H$ can be seen as follows. Just like in the proof of Proposition \ref{prop:valuation_sp_invts}, it is enough to consider the case of homogeneous $\alpha$. The set of actions of critical points of $\cA_{H:L}$ of given Conley--Zehnder index is finite due to the monotonicity of $L$, in particular $\ell(\alpha;H)$ is at least the minimum such action. The assertion follows.} Moreover, we set $\ell(0;H) = -\infty$ for all $H$ so that \textsc{Finiteness} is proven. Notice also that \textsc{Continuity} was proven above for generic smooth Hamiltonians, see \S \ref{sec:cont-property}. It obviously extends to continuous functions, and immediately implies \textsc{Monotonicity}. Finally, notice that \textsc{Ground ring action} is obvious by definition.

\bigskip

We now prove the remaining properties, where we only deal with the nontrivial case of non-zero quantum classes.

\bigskip

   \noindent \textsc{$\bullet$ Spectrality:} The proof is similar to the Hamiltonian case, for instance see \cite{Oh_Construction_sp_invts_Ham_paths_closed_symp_mfds}. Recall that $\omega \big(\pi_2(M,L) \big) = \Z \sfA $ is a discrete subset of $\bb R$. Since we work in a slightly more general situation where we distinguish between homotopy classes of cappings, we include a proof for the sake of completeness.

When $H$ is smooth and nondegenerate, this property follows from the fact that the complement of the spectrum of $H$ is open and dense, since we work in monotone manifolds, and from the fact that the image of the map $i_*^a$ is unchanged as long as $a$ stays in the complement of the spectrum.

When $H$ is degenerate, choose a sequence of nondegenerate Hamiltonians $\{ H_n \}_n$ $C^2$-converging to $H$. By our definition of spectral invariants, $\ell(\alpha;H) = \lim_{n \rightarrow +\infty} \ell(\alpha;H_n)$. By the nondegenerate case, we can choose a sequence $\{(\gamma_n,\wh \gamma_n)\}_n$ so that for all $n$, $\wt \gamma_n = [\gamma_n, \wh\gamma_n] \in \Crit \cA_{H_n:L}$ and $\ell(\alpha;H_n) =  \m A_{H_n}(\wt\gamma_n)$. By $C^2$-convergence and compactness of $M$, the Arzela--Ascoli Theorem ensures the existence of a subsequence of $\{ \gamma_n \}_n$ $C^1$-converging to a Hamiltonian chord $\gamma$ of $H$.

Fix a Riemannian metric on $M$ with respect to which $L$ is totally geodesic and choose a number $\epsilon$ less than its injectivity radius. For $n$ large enough define a strip $u_n \fc [0,1]^2 \to M$ 
such that for fixed $t$, $s \mapsto u_n(s,t)$ is the geodesic arc of length at most $\epsilon$ running from $\gamma_n(t)$ to $\gamma(t)$. Similarly we define the strip $u_{n,n'}$ connecting $\gamma_n$ to $\gamma_{n'}$ for large enough $n,n'$. Now fix $n_0$ large enough and define $\wh\gamma = \wh\gamma_{n_0} \sharp u_{n_0}$. We claim that
$$\cA_{H:L} \big(\wt\gamma = [\gamma,\wh\gamma] \big) = \ell(\alpha;H)\,.$$

Note that $\int u_n^*\omega = \int u_{n'}^*\omega + \int u_{n,n'}^*\omega$. Indeed, by the Stokes formula and due to the fact that $L$ is Lagrangian, we see that the difference between the two terms equals the integral of $d\omega = 0$ over the set realizing a homotopy between the three strips, and therefore it vanishes. Next, we see that $\int u_n^*\omega \to 0$ as $n \to \infty$, as well as $\int u_{n,n'}^*\omega \to 0$ when $n,n' \to \infty$. Also note that $\int \wh\gamma_n^*\omega + \int u_{n,n'}^*\omega = \int \wh\gamma_{n'}^*\omega$. Indeed, by construction the difference between the two sides must be the area of an element of $\pi_2(M,L)$, but on the other hand since $\big\{\int \wh\gamma_n^*\omega\big\}_n$ is Cauchy and by the previous observation, this difference tends to zero, therefore it must be identically zero, owing to the fact that the group of periods of $\omega$ is discrete.

We have
$$\cA_{H:L}(\wt\gamma) - \cA_{H_n:L}(\wt\gamma_n) = \int_0^1 \big(H_t \big(\gamma(t)\big) - H_{n,t}\big(\gamma_n(t) \big)\big)\,dt - \int (\wh\gamma^*\omega - \wh\gamma_n^*\omega)\,.$$
By construction, the first integral tends to zero when $n \to \infty$. In the second integral we have
$$\int (\wh \gamma^* \omega - \wh\gamma_n^* \omega) = \int (\wh\gamma_{n_0}^*\omega - \wh\gamma_n^*\omega) + \int u_{n_0}^*\omega = \int u_{n,n_0}^*\omega + \int u_{n_0}^*\omega = \int u_n^*\omega$$
which tends to zero when $n \to \infty$. Thus we see that
$$\cA_{H:L}(\wt\gamma) = \lim_{n \to \infty}\cA_{H_n:L}(\wt\gamma_n) = \lim_{n\to\infty}\ell(\alpha;H_n) = \ell(\alpha;H)\,.$$

   \noindent \textsc{$\bullet$ Symplectic invariance} First notice that, for $\psi \in \Symp(M,\omega)$, $L'=\psi(L)$ shares the same properties as $L$ so that $\ell'$ is well-defined. It is enough to prove the assertion for nondegenerate $H$. Let therefore $(H,J)$ be a regular Floer datum. In \S \ref{subsec:symplectomorphism_group} it is shown that there exists a canonical chain isomorphism
   $$\psi_* \fc \big(CF_*(H:L),\partial_{H,J} \big) \to \big(CF_*(H^\psi,J^\psi:L'),\partial_{H^\psi,J^\psi} \big)\,,$$
   where $H^\psi = H \circ \psi^{-1}$ and $J^\psi = \psi_*J$. Since $\psi_*$ preserves the actions of the critical points of the respective action functionals, we obtain, for each $a \in \R$, the commutative diagram, where the vertical arrows are inclusions:
   $$\xymatrix{CF_*^a(H:L) \ar[r]^{\psi_*} \ar[d] & CF_*^a(H^\psi:L') \ar[d] \\ CF_*(H:L) \ar[r]^{\psi_*} & CF_*(H^\psi:L') }$$
   It is also shown in \S \ref{subsec:symplectomorphism_group} that $\psi_*$ induces an isomorphism on homology and an isomorphism $\psi_* \fc QH_*(L) \to QH_*(L')$. These are featured in the following diagram:
   $$\xymatrix{HF_*^a(H,J:L) \ar[r]^{i_*^a} \ar[d]^{\psi_*} & HF_*(H,J:L) \ar[d]^{\psi_*} & & QH_*(L) \ar[d]^{\psi_*} \ar[ll]_-{\PSS^{H,J}} \\  HF_*^a(H^\psi,J^\psi:L') \ar[r]^{i_*^a}  & HF_*(H^\psi,J^\psi:L') & & QH_*(L') \ar[ll]_-{\PSS^{H^\psi,J^\psi}}}$$
    from which it is clear that $\ell(\alpha;H) = \ell'(\psi_*(\alpha);H^\psi)$, as claimed.
    
    \bigskip

   \noindent \textsc{$\bullet$ Triangle inequality:} The proof of this fact is rather standard and has been carried out in several context, see \cite{AS10, Monzner_Vichery_Zapolsky_partial_qms_qss_cot_bundles}. By \textsc{Continuity} it suffices to prove this for smooth nondegenerate $H^1$, $H^2$. Recall from \S \ref{subsubsec:Lagr_HF} that we denote by $\Sigma_\star$ the closed disk with three boundary punctures. Pick almost complex structures $J^1$, $J^2$ so that $(H^i,J^i)$, $i=1,2$ are regular Floer data for $L$, and associate them to the negative punctures of $\Sigma_\star$, so that $H^2$ comes before $H^1$ in the cyclic order on $\partial D^2$. Pick $\varepsilon > 0$. By Remark \ref{rema:time-rep-SI} we can assume that $H^1_t$, $H^2_t$ vanish for $t$ close to $0,1$. Let $H^3$ be a nondegenerate Hamiltonian which satisfies $\|H^3 - H^1 \sharp H^2\|_{C^0} \leq \varepsilon$, and let $J^3$ be such that $(H^3,J^3)$ is a regular Floer datum for $L$, and associate it to the positive puncture of $\Sigma_\star$. We will now choose a suitable perturbation datum $(K,I)$ on $\Sigma_\star$. We note that the open set
   $$\Upsilon_\star = \R \times (0,2) \setminus (\infty,0]\times \{1\}$$
   with the conformal structure induced from its inclusion into $\R^2 = \C$ maps biholomorphically onto the interior of $\Sigma_\star$, and we can choose this biholomorphism so that the negative (positive) ends of $\Upsilon_\star$ correspond to the negative (positive) punctures of $\Sigma_\star$. On $\Upsilon_\star$ we put the following perturbation datum $(K,I)$.

The form $K$ is characterized by $K(\partial_s) = 0$ and
$$K(s,t,x)(\partial_t) = \left\{ \begin{array}{ll}  
H^{1}(t,x)  &  \mbox{ if } s\leq 1, \; t \in (0,1), \\ 
H^{2}(t-1,x) &  \mbox{ if } s \leq 1, \; t\in (1,2), \\ 
H^{3}(t,x) &  \mbox{ if } s \geq 2, \; t\in (0,2)  
\end{array}\right.$$
and $|\del_s K(\partial_t)| \leq \varepsilon$ for all $t \in (0,2)$ and all $s \in \bb R$. The family of almost complex structures $I_{(s,t)}$ on $M$ for $s \in \bb R$ and $t \in (0,2)$ is chosen so that $(K,I)$ is regular and
$$I(s,t,x) = \left\{ \begin{array}{ll}  
J^{1}(t,x)  &  \mbox{ if } s\leq 1, \; t \in (0,1), \\ 
J^{2}(t-1,x) &  \mbox{ if } s \leq 1, \; t\in (1,2), \\ 
J^{3}(t,x) &  \mbox{ if } s \geq 2, \; t\in (0,2).  
\end{array}\right.$$
Now we push this perturbation datum on $\Upsilon_\star$ to $\Sigma_\star$ via the aforementioned biholomorphism, and extend the resulting datum on the interior of $\Sigma_\star$ to the whole surface by continuity. Abusing notation we denote this perturbation datum on $\Sigma_\star$ also by $(K,I)$.

Now assume that the matrix element $C(\wt\gamma^2) \otimes C(\wt\gamma^1) \to C(\wt\gamma^3)$ of $\star_{K,I}$ does not vanish, where $\wt\gamma^i \in \Crit \cA_{H^i:L}$. Then there is $u \in \cM(K,I;\{\wt\gamma^i\}_i)$. Computing the energy of $u$ in the coordinates $(s,t)$ on $\Upsilon_\star$ in a manner similar to the computation in \S \ref{sec:cont-property} above leads to 
\begin{align*}
  \m A_{H^1:L}(\wt \gamma^1) + \m A_{H^2:L}(\wt \gamma^2) - \m A_{H^3:L}(\wt \gamma^3) \geq - 2\varepsilon \,.
\end{align*}
This indicates that the restriction of $\star_{K,I}$ to the filtered Floer complexes satisfies:
\begin{align*}
  \star_{K,I} : CF_*^{a_1}(H^1:L) \otimes CF_*^{a_2}(H^2:L) \rightarrow CF_*^{a_1+a_2+2\varepsilon}(H^3:L)
\end{align*}
for any $a_1$, $a_2$. We can deduce the following inequality on spectral invariants:
$$\ell(\alpha \star \beta; H^3) \leq \ell(\beta;H^1) + \ell(\alpha;H^2) + 2\varepsilon\,.$$

Since $\varepsilon$ is arbitrary, we get the expected result 
$$\ell(\alpha \star \beta; H^1\sharp H^2) \leq \ell(\beta;H^1) + \ell(\alpha;H^2)\,.$$

\bigskip

   \noindent \textsc{$\bullet$ Normalization:} For the first assertion, consider the homotopy $H^s$ defined by $H^s=H+sc$ for $s \in [0,1]$. For fixed $s$, $\Spec(H^s:L)$ is $\Spec(H:L)$ translated by $s \int_0^1 c(t) \, dt$. Thus the result holds for smooth $H$ by \textsc{Spectrality} and \textsc{Continuity}. By $C^0$-limit, it also holds for any continuous function. The second assertion follows from Proposition \ref{prop:valuation_sp_invts}. 

   Finally, let us consider the case $\alpha = [L]$. The maximum of a Morse function on $L$ with a single maximum is always a cycle representing $[L]$ (see \S \ref{subsec:Lagr_QH}); this cycle has valuation $0$. It follows from the definition of $\nu$ that $\nu([L]) \leq 0$, that is $\ell_+(0) = \nu([L])\sfA \leq 0$. On the other hand, since $[L] \star [L] = [L]$, \textsc{Triangle inequality} yields
$$\ell([L];0) \leq \ell([L];0) + \ell([L];0)\,,$$
which means $\ell([L];0) \geq 0$, and thus $\ell([L];0)=0$.

\bigskip

   \noindent \textsc{$\bullet$ Module structure:} The proof of this inequality is formally similar to that of \textsc{Triangle inequality}. The only difference is that $H^2$ is time-periodic, and that we map the surface $\Upsilon_\bullet = \Upsilon_\star$ to the surface $\Sigma_\bullet$ so that the complement is the boundary together with a segment connecting the interior puncture to the boundary. The perturbation datum $(K,I)$ on $\Upsilon_\bullet$ is still given by the same formulas as in the proof of \textsc{Triangle inequality} above, and we push it to $\Sigma_\bullet$ and extend the result by continuity. Computing the energy of a solution of the Floer PDE with respect to this perturbation datum implies that the restriction of $\bullet_{K,I}$ to the filtered Floer complexes satisfies 
   $$\bullet_{K,I} \fc CF_*^{a_2}(H^2) \otimes CF_*^{a_1}(H^1:L) \to CF_*^{a_1+a_2+\varepsilon}(H^3:L)\,,$$
   for arbitrary $\varepsilon>0$, from which the desired inequality readily follows.   

   \bigskip

   \noindent \textsc{$\bullet$ Duality:} The first equality is obvious from the definition of homological (see \eqref{eq:definition-SI}) and cohomological (see \eqref{eq:cohomological_sp_invts}) spectral invariants, together with the fact that the duality isomorphism \eqref{eq:Dual_module_isomorphism} flips the sign of the action of the generators.
   
   To prove the inequality we need a general result on filtered based chain complexes and their duals. We present a self-contained account here for the sake of completeness.
   
\begin{defin}\label{def:filtered_chain_cx}
Let $R$ be a ring. A based filtered chain complex over $R$ is a quadruple $\cV = (V,\cB,\cA,\partial)$ where $V$ is a free $R$-module, $\cB$ is a basis for $V$, that is
$$V = \bigoplus_{v\in \cB}R\cdot v\,,$$
$\cA \fc \cB \to \R$ is a function, and $\partial \fc V \to V$ is a boundary operator, that is $\partial^2 = 0$, such that $\langle \partial u,v \rangle \neq 0$ implies $\cA(u) > \cA(v)$ for $u,v \in \cB$.
\end{defin}

Given such a based filtered chain complex, we can define the associated spectral invariants, as follows. For $a \in R$ define
$$V^a = \bigoplus_{\substack{v \in B\\ \cA(v) < a}}R\cdot v\,.$$
It follows that $\partial$ preserves $V^a$, that is $V^a \subset V$ is a subcomplex. We let $i^a \fc H(V^a) \to H(V)$ be the induced map on homology. The spectral invariant of $\alpha \in H(V)$ is
\begin{equation}\label{eq:abstract_sp_invts}
\ell(\alpha,\cV) = \inf\{a \in \R \,|\, \alpha \in \im i^a\}\,.
\end{equation}
\emph{A priori} $\ell(\alpha,\cV)$ may be infinite even if $\alpha \neq 0$, therefore we make the assumption that $\cV$ is such that this does not happen. The above Floer complex of a regular Floer datum is an example.

We define the dual based filtered cochain complex of $\cV$ to be $\cV^\vee = (V^\vee, \cB^\vee, \cA^\vee, \partial^\vee)$. Here $\cB^\vee = \{v^\vee \,|\, v \in \cB\}$, where $v^\vee$ is the functional on $V$ having value $1$ on $v$ and $0$ on the other elements of $\cB$, $V^\vee$ is the subspace of the dual module $\Hom_R(V,R)$ spanned by $\cB^\vee$, $\cA^\vee(v^\vee) = \cA(v)$ for $v \in \cB$ and $\partial^\vee$ is the restriction to $V^\vee$ of the coboundary operator dual to $\partial$.

We can similiraly define spectral invariants of $\cV^\vee$, as follows. For $a \in \R$ let
\begin{equation}\label{eq:abstract_coho_sp_invts}
V^\vee_a = V^\vee \Big / \bigoplus_{\substack{v^\vee \in \cB^\vee \\ \cA^\vee(v^\vee) > a} }R \cdot v^\vee\,.
\end{equation}
Then, since by assumption $\langle \partial^\vee v^\vee, u^\vee \rangle \neq 0$ implies $\cA^\vee(u^\vee) > \cA^\vee(v^\vee)$, we see that $V^\vee_a$ is a quotient cochain complex of $V$. We let $i_a \fc H^\vee(V^\vee) \to H^\vee(V^\vee_a)$ be the induced map on cohomology. The spectral invariant of $\alpha^\vee \in H^\vee(V^\vee)$ is
$$\ell^\vee(\alpha^\vee,\cV^\vee) = \sup \{a \in \R \,|\, \alpha^\vee \in \ker i_a\}\,.$$
Again, we make the assumption that this is a finite number for $\alpha^\vee \neq 0$.

We have the duality pairings
$$\langle \cdot, \cdot \rangle \fc H^\vee(V^\vee) \times H(V) \to R \quad \text{and} \quad \langle \cdot, \cdot \rangle \fc H^\vee(V^\vee_a) \times H(V^a) \to R$$
induced from the duality pairing between $V^\vee$ and $V$. The main result here is the following lemma.
\begin{lemma}
We have, for $\alpha^\vee \in H^\vee(V^\vee)$:
$$\ell^\vee(\alpha^\vee,\cV^\vee) \leq \inf \big\{\ell(\alpha,\cV) \,|\, \alpha \in H(V):\langle \alpha^\vee,\alpha \rangle \neq 0\big\}\,.$$
\end{lemma}

\begin{proof}
We have the following commutative diagram
$$\xymatrix{H^\vee(V^\vee) \ar[d]^{i_a} \ar[r] & \big(H(V)\big)^\vee \ar[d]^{(i^a)^\vee} \\ H^\vee(V^\vee_a) \ar[r]& \big( H(V^a)\big)^\vee}$$
where the horizontal arrows are induced from the duality pairings whereas $(i^a)^\vee$ is the dual map of $i^a$.

Assume $\alpha \in H(V)$ is such that $\langle \alpha^\vee,\alpha \rangle \neq 0$ and $\ell(\alpha,\cV) \leq a \in \R$. Then there is $\alpha' \in H(V^a)$ with $i^a(\alpha') = \alpha$. We then have, using the above diagram:
$$\langle i_a(\alpha^\vee), \alpha' \rangle = \langle \alpha^\vee, i^a(\alpha')\rangle = \langle \alpha^\vee, \alpha \rangle \neq 0\,,$$
meaning $i_a(\alpha^\vee) \neq 0$, that is $\alpha^\vee \notin \ker i_a$, which implies $a \geq \ell^\vee(\alpha^\vee,\cV^\vee)$. This proves the desired inequality.
\end{proof}
The inequality in \textsc{Duality} property now follows by noting that $CF_*(H:L;\cL)$ is a based filtered chain complex in the sense of Definition \ref{def:filtered_chain_cx}, while the cochain complex $CF^*(H:L;\cL)$ is its dual (up to the sign in the differential, which of course does not matter), and the definitions of spectral invariants \eqref{eq:definition-SI}, \eqref{eq:cohomological_sp_invts} are particular cases of \eqref{eq:abstract_sp_invts}, \eqref{eq:abstract_coho_sp_invts}.

In order to prove the last assertion of this item, namely the one about the equality
$$\ell(\alpha^\vee;H) = \inf \big\{\ell(\beta;H)\,|\,\langle \alpha^\vee,\beta\rangle \neq 0 \big\}\,,$$
assume that the ground ring $R$ is a field and that the Floer complex $CF_*(H:L)$ is finite-dimensional in every degree. This is then a particular case of a general statement about filtered graded chain complexes. We say that a filtered chain complex $\cV$ is graded if there is a function $|\cdot| \fc \cB \to \Z$, called the grading, and we have the decomposition
$$V = \bigoplus_{k \in \Z} V_k \quad \text{where} \quad V_k = \bigoplus_{v \in \cB,|v| = k}R\cdot v\,,$$
and the boundary operator $\partial$ maps $V_k$ into $V_{k-1}$. Note that in this case the dual filtered cochain complex $\cV^\vee$ is also graded, in the sense that there is an analogous decomposition by degree and the differential $\partial^\vee$ raises the degree by $1$.
\begin{lemma}
Let $\cV = (V,\cB,|\cdot|,\cA,\partial)$ be a filtered graded chain complex over a field $R$ and assume it is finite-dimensional in every degree. Then
$$\ell^\vee(\alpha^\vee,\cV^\vee) = \inf \big\{\ell(\alpha,\cV)\,|\,\langle \alpha^\vee,\alpha\rangle \neq 0 \big\}\,.$$
\end{lemma}
\begin{proof}
Note that it is enough to prove this for homogeneous $\alpha^\vee \in V^\vee_k$. Next, it suffices to show that for any $a \in \R$ with $a > \ell^\vee(\alpha^\vee,\cV^\vee)$ there is $\alpha \in H(V)$ with $\langle \alpha^\vee,\alpha \rangle \neq 0$ such that $\ell(\alpha,\cV) \leq a$. Assume therefore $a > \ell^\vee(\alpha^\vee,\cV^\vee)$, that is $i_a(\alpha^\vee) \neq 0$. The assumptions on $R$ and the degree-wise finite-dimensionality of $V$, and therefore of $V^\vee$, imply that the natural map
$$H^\vee_k(V^\vee_a) \to \big(H_k(V^a)\big)^\vee$$
is an isomorphism.\footnote{In fact, we only use the injectivity of this map.} Thus there is $\alpha' \in H_k(V^a)$ such that $\langle i_a(\alpha^\vee),\alpha' \rangle \neq 0$. Put $\alpha = i^a(\alpha')$ and then we see that
$$\langle \alpha^\vee,\alpha \rangle = \langle i_a(\alpha^\vee),\alpha' \rangle \neq 0\,,$$
and by construction $\ell(\alpha,\cV) \leq a$.
\end{proof}

\bigskip

   \noindent \textsc{$\bullet$ Novikov action:} By the definition of the action of the Novikov ring $R[F]$ on the Lagrangian Floer complex $CF_*(H:L)$, the multiplication by the element $A \in F$ induces a chain isomorphism
   $$CF_*(H:L) \to CF_{*-\mu(A)}(H:L)\,,$$
   which restricts to a chain isomorphism
   $$CF_*^a(H:L) \to CF_{*-\mu(A)}^{a-\omega(A)}(H:L)$$
   for any $a \in \R$. These chain maps induce the two left vertical arrows in the following commutative diagram:
   $$\xymatrix{HF_*^a(H,J:L) \ar[r]^{i^a_*} \ar[d]^-{A} & HF_*(H:L) \ar[d]^-{A} & QH_*(L) \ar[l]_-{\PSS^{H,J}} \ar[d]^-{A} \\ HF_{*-\mu(A)}^{a-\omega(A)}(H,J:L) \ar[r]^{i^{a-\omega(A)}_*} & HF_{*-\mu(A)}(H,J:L) & QH_{*-\mu(A)}(L) \ar[l]_-{\PSS^{H,J}}}$$
   where the right arrow is the action of the Novikov ring on quantum homology. The equality now follows from the definition of spectral invariants.

   \bigskip

   \noindent \textsc{$\bullet$ Lagrangian control:} Notice that it is sufficient to prove this property for smooth Hamiltonians. Assume that $H$ restricted to $L$ is a function of time, $H_t|_L = c(t)$. Then for any $s \in \bb R$ we see that $H^s=sH$ satisfies $H^s_t|_L = sc(t)$. Since $L$ is Lagrangian, for any fixed $s$, the chords of $H^s$ are included in $L$. Each of these comes with a natural capping (itself), included in $L$ so that it has area 0. This shows that for all $s \in \bb R$
$$\Spec(H^s) = \Big\{ s\cdot\! \int_0^1 c(t) \, dt + k \sfA \,\big|\, k \in \bb Z \Big\} \,.$$
Now, by \textsc{Spectrality} and \textsc{Continuity}, $\ell(\alpha;H^s) =  s\int_0^1 c(t) \, dt + k_0 \sfA$ for some $k_0 \in \bb Z$, independent of $s$. Since $H^0=0$, $k_0=\nu(\alpha)$ by \textsc{Normalization}. We get for $s=1$ that $\ell(\alpha;H) =  \nu(\alpha)\sfA +  \int_0^1 c(t) \, dt$.

In case $H$ satisfies $H_t|_L \leq c(t)$ for all $t$, pick any function $K$ such that $K \geq H$ and $K_t|_L = c(t)$. By \textsc{Monotonicity}, $\ell(\alpha;H) \leq \ell(\alpha;K) = \nu(\alpha)\sfA+\int_0^1 c(t) \, dt$. The last case is similar and the bounds on $\ell(\alpha;H)$ follow.

\bigskip

   \noindent \textsc{$\bullet$ Non-negativity:} By \textsc{Triangle inequality} with $\alpha=\beta=[L]$, we have $\ell_+(H) + \ell_+(\overline{H}) \geq \ell_+(H \sharp \overline H)$. By \textsc{Normalization} and the definition of $\overline{H}$ and $\sharp$, it suffices to prove the result when $H$ is normalized. In that case, $H \sharp \overline{H}$ is also normalized and is equivalent to 0, in the sense of Definition \ref{defi:equivalent-norm-Hamiltonians}. By invariance of spectral invariants, see \S \ref{sec:invariance-ell2}, $\ell_+(H \sharp \overline H) = \ell_+(0)$ which vanishes by \textsc{Normalization}.
   
   \bigskip

   \noindent \textsc{$\bullet$ Maximum:} First, assume that $H$ is normalized. Since $\alpha = [L] \star \alpha$ and $H \sharp 0$ is equivalent to $H$ in the sense of Definition \ref{defi:equivalent-norm-Hamiltonians}, by \S \ref{sec:invariance-ell2} and \textsc{Triangle inequality} we obtain $\ell(\alpha;H) \leq \ell_+(H) + \ell(\alpha;0)$. When $H$ is not normalized, the result follows from the first case and \textsc{Normalization}.
   
   \bigskip

Theorem \ref{thm:main_properties_Lagr_sp_invts}, and consequently Theorem \ref{thm:main_properties_Intro} and Proposition \ref{prop:module_struct_Intro}, are now proved.

\subsubsection{The equivalence of Hamiltonian spectral invariants and Lagrangian spectral invariants of the diagonal}\label{subsubsec:equiv_Ham_sp_invt_Lagr_diagonal}

In this subsection we prove Theorem \ref{thm:spectral_invts_diagonal_Intro} of the introduction.

Recall from \S \ref{subsec:Lagr_diagonal} that there is a canonical isomorphism
$$QH_*(\Delta) = QH_*(M)\,.$$
It is enough to show this for a nondegenerate smooth $H$. Let us therefore assume that $(H,J)$ is a regular time-periodic Floer datum on $M$. In view of Remark 
\ref{rema:time-rep-SI}, we can assume that $H_t \equiv 0$ for $t \in [\tfrac 1 2,1]$.

In \S \ref{subsec:Lagr_diagonal} it is shown that there is a regular Floer datum $(\wh H, \wh J)$ for $\Delta$, for which there is a canonical chain isomorphism
$$\big(CF_*(H), \partial_{H,J} \big) = \big(CF_*(\wh H:\Delta), \partial_{\wh H, \wh J} \big)\,,$$
preserving the grading and the action of the generators. Moreover, for any $a \in \R$ we have the following commutative diagram, where the left two vertical arrows are induced by this isomorphism and the right vertical arrow is an analogous isomorphism on quantum homology, described \emph{ibid.}:
$$\xymatrix{HF_*^a(H,J) \ar@{=}[d] \ar[r]^{i_*^a} & HF_*(H,J) \ar@{=}[d] & QH_*(M) \ar[l]_-{\PSS^{H,J}} \ar@{=}[d] \\ HF_*^a(\wh H, \wh J:\Delta) \ar[r]^{i_*^a} & HF_*(\wh H, \wh J:\Delta) & QH_*(\Delta) \ar[l]_-{\PSS^{\wh H,\wh J}}}$$
From this we have the equality of spectral invariants
$$\ell(\alpha;\wh H) = c(\alpha;H)\,.$$
Recall how $\wh H$ is defined: it is the direct sum $\wh H_t = H_t \oplus 0$ for $t\in [0,\tfrac 1 2]$. The Hamiltonian $\wh H$ therefore differs from $H \oplus 0$ by time reparametrization, which means that $\ell(\alpha; H \oplus 0) = \ell(\alpha; \wh H)$, thereby completing the proof of the theorem.

\subsubsection{Spectral invariants of products of Lagrangians}\label{subsubsec:sp_invts_products}

Here we formulate precisely and prove the property of the spectral invariants mentioned at the end of \S \ref{subsec:main_result}.

In \S \ref{subsec:Lagr_HF_QH_products} we introduced, for monotone Lagrangians $L_i \subset (M_i,\omega_i)$, $i = 1,2$, the canonical map
$$QH_*(L_1) \otimes QH_*(L_2) \to QH_*(L_1 \times L_2)\,.$$
\begin{theo}\label{thm:invts_of_products}
For Hamiltonians $H_i$ on $M_i$ and classes $\alpha_i \in QH_*(L_i)$ we have
$$\ell(\alpha_1\otimes \alpha_2; H_1 \oplus H_2) \leq \ell(\alpha_1;H_1) + \ell(\alpha_2;H_2)$$
with equality if $R$ is a field.
\end{theo}
\begin{remark}
We cannot omit the assumption that $R$ is a field in order to obtain an equality in Theorem \ref{thm:invts_of_products}. The reason for this is that even if both classes $\alpha_1$, $\alpha_2$ are nonzero, their tensor product $\alpha_1 \otimes \alpha_2$ in $QH_*(L_1) \otimes QH_*(L_2)$ may vanish.
\end{remark}

\begin{proof}
We have the following commutative diagram

\noindent\resizebox{\textwidth}{!}{
$$\xymatrix{HF_*^{a_1}(H_1,J_1:L_1) \times HF_*^{a_2}(H_2,J_2:L_2) \ar[r] \ar[d]^{i_*^{a_1} \times i_*^{a_2}} & HF_*^{a_1 + a_2}(H_1 \oplus H_2, J_1 \oplus J_2:L_1 \times L_2) \ar[d]^{i_*^{a_1 + a_2}} \\ HF_*(H_1,J_1:L) \times HF_*(H_2,J_2:L) \ar[r] &  HF_*(H_1 \oplus H_2, J_1 \oplus J_2:L_1 \times L_2) \\ QH_*(L_1) \times QH_*(L_2) \ar[r] \ar[u]_{\PSS \times \PSS} & QH_*(L_1 \times L_2) \ar[u]_{\PSS}}$$
}

\noindent which implies the inequality.

In case $R$ is a field, the equality can be proved using the method developed in \cite{Entov_Polterovich_rigid_subsets_sympl_mfds}. 
\end{proof}

\subsection{Spectral invariants of isotopies}
\label{sec:SI-isotopies}

\begin{theo}\label{thm:properties_sp_invts_isotopies}
  Let $L$ be a closed monotone Lagrangian of $(M,\omega)$ with minimal Maslov number $N_L \geq 2$. The function
$$\ell \fc QH_*(L) \times \wt \Ham(M,\omega) \longrightarrow \bb R \cup \{-\infty\}$$
constructed in \S \ref{sec:invariance-ell2} above satisfies the following properties.
\begin{Properties} 
   \item[Finiteness] $\ell(\alpha;\wt\phi) = -\infty$ if and only if $\alpha = 0$.
   \item[Spectrality] For all $\alpha \neq 0$, $\ell(\alpha;\wt \phi) \in \Spec(\wt \phi:L)$.
   \item[Ground ring action] For all $r \in R$, $\ell(r \cdot \alpha;\wt \phi) \leq \ell(\alpha;\wt \phi)$, with equality if $r$ is a unit. 
   \item[Normalization] We have $\ell(\alpha;\id) = \nu(\alpha) \sfA$ and $\ell_+(\id)=0$.
   \item[Continuity] Assume $H$ and $K$ are normalized, then $$\int_0^1 \min_M (H_t - K_t) dt \leq \ell(\alpha;\wt \phi_H) - \ell(\alpha;\wt \phi_{K}) \leq \int_0^1 \max_M (H_t - K_t) dt \,.$$ 

   \item[Monotonicity] If $M$ is noncompact, $H$ and $K$ have compact support, and $H \leq K$, then $\ell(\alpha;\wt \phi_H) \leq \ell(\alpha;\wt \phi_K)$.

   \item[Triangle inequality] For all $\alpha$ and $\beta$, $$\ell(\alpha \star \beta; \wt \phi \wt \psi) \leq \ell(\beta;\wt \phi) + \ell(\alpha ;\wt\psi) \,.$$
   \item[Module structure] For all $a \in QH_*(M)$ and $\alpha \in QH_*(L)$, $\ell(a \bullet \alpha;\wt \phi \wt \psi) \leq c(a;\wt \psi) + \ell(\alpha;\wt \phi)$.

   \item[Lagrangian control] If for all $t$, $H_t|_L = c(t) \in \bb R$ (respectively $\leq$, $\geq$), then
   $$\ell(\alpha;\wt\phi_H)=\nu(\alpha)\sfA+\int_0^1 \left( c(t) - \int_M H_t \, \omega^n \right) dt \quad \text{(respectively }\leq, \geq\text{)}\,.$$
   Thus, for all $H$:
$$\int_0^1 \min_L H_t \,dt \leq \ell(\alpha; \wt \phi_H) - \nu(\alpha)\sfA + \int_0^1 \int_M H_t \, \omega^n \,dt \leq \int_0^1 \max_L H_t \,dt \,.$$

   \item[Non-negativity] $\ell_+(\wt \phi) + \ell_+(\wt \phi^{-1}) \geq 0$.
   \item[Maximum] $\ell(\alpha;\wt \phi) \leq \ell_+(\wt \phi)+\ell(\alpha;\id)$. 
   \item[Duality] For $\alpha \in QH_*(L)$ and $\alpha^\vee \in QH^{n-*}(L;\cL)$ corresponding to it via the duality isomorphism \eqref{eq:duality_isomorphism_QH} we have
   $$-\ell(\alpha;\wt\phi^{-1}) = \ell(\alpha^\vee;\phi) \leq \inf \big\{\ell(\beta;\phi)\,|\,\langle \alpha^\vee,\beta\rangle \neq 0 \big\}$$
   with equality if the ground ring $R$ is a field and the Floer complexes of nondegenerate Hamiltonians are finite-dimensional in each degree.
   \item[Novikov action] In case there is an action of the Novikov ring $R[F]$ as above, we have, for $A \in F$: $\ell(A \alpha;\wt\phi) = \ell(\alpha;\wt\phi) - \omega(A)$.
   \item[Symplectic invariance] Let $\psi \in \Symp (M,\omega)$, $L' = \psi(L)$ and let $\ell'$ be the associated spectral invariant function. Then\footnote{Here we use the natural action of the group $\Symp(M,\omega)$ on $\wt\Ham(M,\omega)$ by conjugation.} $\ell(\alpha;\wt\phi) = \ell'(\psi_*(\alpha);\psi\wt\phi\psi^{-1})$.
\end{Properties}
\end{theo}

First notice that we assume $M$ to be noncompact for \textsc{Monotonicity}. This comes from the facts that isotopies naturally correspond to \textit{normalized} Hamiltonians, and that there are no normalized Hamiltonians $H$ and $K$ satisfying $H \leq K$ on a compact manifold except if they coincide everywhere. On noncompact manifolds, isotopies also correspond to normalized Hamiltonians and the natural normalization condition is requiring the Hamiltonian to have compact support.

This theorem follows in a straightforward manner from Theorem \ref{thm:main_properties_Lagr_sp_invts}, with the exception of \textsc{Lagrangian control} and \textsc{Symplectic invariance}. To prove \textsc{Lagrangian control} as stated here it suffices to note that the Hamiltonian $H_t - \int_M H_t\,\omega^n$ is normalized and generates the same Hamiltonian flow as $H_t$, therefore
\begin{align*}
\ell(\alpha;\wt\phi_H) = \ell \big(\alpha;H_t - \textstyle \int_M H_t\,\omega^n\big) = \ell(\alpha;H) - \textstyle \int_0^1\int_M H_t\,\omega^n\,dt\,,
\end{align*}
and the property now follows from \textsc{Lagrangian control} for Hamiltonian functions (Theorem \ref{thm:main_properties_Lagr_sp_invts}).

For \textsc{Symplectic invariance} it suffices to use Theorem \ref{thm:main_properties_Lagr_sp_invts} together with the observation that if $\wt\phi = \wt\phi_H$ then $\psi \wt\phi_H \psi^{-1} = \wt\phi_{H \circ \psi^{-1}}$.

  \section{Proof of the Hofer bounds}
  \label{sec:proof-Hofer-bound}

  We now prove Propositions \ref{prop:sp_invts_htpy_class_path_Intro}, \ref{prop:relation_Lagr_Hofer_distance}. The bounds for $\wt\chi$ follow from the bounds for $H$ if we set $\wt\chi = \wt\phi_H$, therefore we establish the latter.
  
Let $L$ and $L'$ be Hamiltonian isotopic Lagrangians. Choose any $\varphi \in \ham(M,\omega)$ such that $\varphi (L)=L'$ and fix a class $\alpha \in QH_*(L)$.

 
 First, by \textsc{Symplectic invariance}, $\ell'(\alpha' ; H) = \ell(\alpha ; H\circ \varphi )$ with $\alpha=\varphi^{-1}_* (\alpha')$ so that we are interested in $| \ell(\alpha; H) - \ell(\alpha; H\circ \varphi) |$. Since adding the same function of time to both $H$ and $H \circ \varphi$ does not affect their difference, by \textsc{Normalization} we get
  \begin{align*}
    | \ell(\alpha; H) - \ell(\alpha; H\circ \varphi) | = | \ell(\alpha; \wt \phi_H) - \ell(\alpha; \varphi^{-1} \wt \phi_H \varphi) |\,.
  \end{align*}

Now, it is useful to notice that for any Hamiltonian $K \co M \times [0,1] \rightarrow \bb R$ such that $\phi_K = \varphi$, the Hamiltonian\footnote{It is well-known that a smooth one-parameter family of Hamiltonian diffeomorphisms is a Hamiltonian isotopy.} isotopies $t \mapsto \varphi^{-1} \phi_H^t \varphi$ and $t \mapsto (\phi_K^t)^{-1} \phi_H^t \phi_K^t$ are homotopic with fixed endpoints and thus define the same element in $\wt \ham (M,\omega)$, that is $\varphi^{-1}\wt\phi_H \varphi = \wt\phi_K^{-1}\wt\phi_H\wt\phi_K$.


Using this, the fact that the fundamental class $[L]$ is the unit of the quantum homology ring, and \textsc{Triangle inequality}, we get
\begin{align*}
   \ell(\alpha; \varphi^{-1} \wt \phi_H \varphi) =  \ell([L] \star \alpha \star [L]; \wt \phi_K^{-1} \wt \phi_H \wt \phi_K)
\leq \ell_+(\wt \phi_K^{-1}) + \ell(\alpha ; \wt \phi_H) + \ell_+(\wt \phi_K)
\end{align*}
from which we deduce that
$$\ell(\alpha ; \wt \phi_H) -  \ell(\alpha; \varphi^{-1} \wt \phi_H \varphi) \geq - \ell_+(\wt \phi_K^{-1}) - \ell_+(\wt \phi_K)\,.$$
In a similar way, by writing $\ell(\alpha; \wt \phi_H) = \ell(\alpha; \wt \phi_K (\wt \phi_K^{-1} \wt \phi_H \wt \phi_K) \wt \phi_K^{-1})$, we get that
$$\ell(\alpha ; \wt \phi_H) -  \ell(\alpha; \varphi^{-1} \wt \phi_H \varphi) \leq \ell_+(\wt \phi_K^{-1}) + \ell_+(\wt \phi_K)$$
and thus conclude that
\begin{align*}
   | \ell(\alpha; H) - \ell(\alpha; H\circ \varphi) | \leq \ell_+(\wt \phi_K^{-1}) + \ell_+(\wt \phi_K) \,.
\end{align*}
By \textsc{Continuity} of spectral invariants, since $\ell_+(\wt \phi_K) \leq \int_0^1 \max_M K_t \, dt$ and $\ell_+(\wt \phi_K^{-1}) \leq \int_0^1 \max_M \overline K_t \, dt = - \int_0^1 \min_M K_t \, dt$, we deduce
\begin{align*}
   | \ell(\alpha; H) - \ell(\alpha; H\circ \varphi) | \leq  \int_0^1 \osc_M K_t \, dt\,. 
\end{align*}
Taking infimum over all $K$ with $\varphi = \phi_K$ concludes the proof of Proposition \ref{theo:lagrangian-nature}.

Proposition \ref{prop:relation_Lagr_Hofer_distance} also follows from the same estimate, since in order to obtain it we need to take infimum over those $K$ generating the given homotopy class of paths $\wt L$. The left-hand side is independent of $K$, as indicated in \S\ref{sec:Hofer-bound}. By definition, the infimum on the right-hand side equals $\| \wt L \|$.

%

\section{Proofs concerning symplectic rigidity}\label{sec:Superheavy}

We note the following fact.
\begin{remark}\label{rem:superheavy_set_enough_one_side_ineq}
It follows from the definition of a quasi-state that $X$ is superheavy if for any $c \in \R$ and $F \in C^0(M)$ with $F|_X = c$ we have $\zeta(F) \geq c$. Indeed, for such $F$ we have $-F|_X = -c$, therefore $\zeta(-F) \geq -c$, but $\zeta(F) = -\zeta(-F) \leq c$, implying $\zeta(F) = c$.
\end{remark}

\begin{proof}[Proof of Proposition \ref{prop:superheavy_if_not_killed_idem}]
By \textsc{Module Structure}, \textsc{Normalization}, and \textsc{Lagrangian Control} we have, for $F \in C^0(M)$ with $F|_L = c$, and $k \in \N$:
$$c(e;kF) + \ell([L];0) \geq \ell(e\bullet [L];kF) = kc + \nu(e\bullet [L])\sfA\,,$$
which yields $\zeta_e(F) = \lim_{k \to \infty}c(e;kF)/k \geq c$, and by Remark \ref{rem:superheavy_set_enough_one_side_ineq} it follows that $L$ is $e$-superheavy.
\end{proof}

\begin{proof}[Proof of Theorem \ref{thm:superheavy_tori}]
The proof is an application of Proposition \ref{prop:superheavy_if_not_killed_idem}, combined with results appearing in \cite{Zapolsky_Canonical_ors_HF}.
\begin{itemize}
\item For a certain choice of twisted coefficients with $\C$ as the ground ring, $QH_*(L_{\C P^2}) \neq 0$ (\emph{ibid.}). Since $[\C P^2]$ is the unit of $QH_*(\C P^2)$, it follows that $[\C P^2]\bullet [L_{\C P^2}] = [L_{\C P^2}] \neq 0$.
\item Abbreviate $M = S^2 \times S^2$ and $L = L_{S^2 \times S^2}$. For a choice of a Spin structure on $L$ and certain twisted coefficients we have $QH_*(L) \neq 0$  (\emph{ibid.}). It is known that the Lagrangian antidiagonal $\ol\Delta \subset M$ is $e_-$-superheavy. Since $L$ is disjoint from the antidiagonal, it follows that it is not $e_-$-superheavy (see \cite{Entov_Polterovich_rigid_subsets_sympl_mfds}), therefore Proposition \ref{prop:superheavy_if_not_killed_idem} implies $e_-\bullet [L] = 0$. Therefore, since $[M] = e_+ + e_-$, we have
$$e_+ \bullet [L] = e_- \bullet [L] + e_+ \bullet [L] = (e_- + e_+) \bullet [L] = [M] \bullet [L] = [L]\,,$$
and the result follows from another use of Proposition \ref{prop:superheavy_if_not_killed_idem}.
\end{itemize}
\end{proof}

\bibliographystyle{alpha}
\bibliography{biblio}

\end{document}